\documentclass[12pt,a4paper,final]{amsart}
\usepackage{amssymb}
\usepackage{hyperref}
\usepackage[notcite, notref]{showkeys}
\usepackage[utf8]{inputenc}
\usepackage[figuresright]{rotating}
\usepackage{comment}

\usepackage{color}
\usepackage{listings} 
\lstdefinelanguage{Sage}[]{Python}
{morekeywords={False,sage,True},sensitive=true}
\definecolor{dblackcolor}{rgb}{0.0,0.0,0.0}
\definecolor{dbluecolor}{rgb}{0.01,0.02,0.7}
\definecolor{dgreencolor}{rgb}{0.2,0.4,0.0}
\definecolor{dgraycolor}{rgb}{0.30,0.3,0.30}

\usepackage{mathtools}

\theoremstyle{plain}
\newtheorem{theorem}{Theorem}[section]
\newtheorem{proposition}[theorem]{Proposition}
\newtheorem{corollary}[theorem]{Corollary}
\newtheorem{lemma}[theorem]{Lemma}
\newtheorem{claim}[theorem]{Claim}

\theoremstyle{definition}

\newtheorem{example}[theorem]{Example}
\newtheorem{remark}[theorem]{Remark}

\newtheorem{notation}[theorem]{Notation}

\newtheorem{algorithm}[theorem]{Algorithm}

\theoremstyle{remark}

\numberwithin{equation}{section}
\newcommand{\N}{\mathbb N}
\newcommand{\Z}{\mathbb Z}

\newcommand{\R}{\mathbb R}
\newcommand{\C}{\mathbb C}

\newcommand{\fe}{\mathfrak e}
\newcommand{\ff}{\mathfrak f}
\newcommand{\fg}{\mathfrak g}
\newcommand{\fh}{\mathfrak h}

\newcommand{\fk}{\mathfrak k}
\newcommand{\fl}{\mathfrak l}
\newcommand{\fm}{\mathfrak m}

\newcommand{\fp}{\mathfrak p}

\newcommand{\ft}{\mathfrak t}

\newcommand{\fz}{\mathfrak z}

\DeclareMathOperator{\GL}{GL}

\DeclareMathOperator{\SO}{SO}
\DeclareMathOperator{\SU}{SU}
\DeclareMathOperator{\Sp}{Sp}
\DeclareMathOperator{\Ut}{U}

\DeclareMathOperator{\Spin}{Spin}

\newcommand{\gl}{\mathfrak{gl}}

\newcommand{\so}{\mathfrak{so}}
\newcommand{\spp}{\mathfrak{sp}}
\newcommand{\su}{\mathfrak{su}}
\newcommand{\ut}{\mathfrak{u}}

\newcommand{\op}{\operatorname}
\newcommand{\Id}{\textup{Id}}

\newcommand{\mi}{\mathrm{i}}

\DeclareMathOperator{\Span}{Span}
\newcommand{\inner}[2]{\langle {#1},{#2}\rangle }
\newcommand{\innerdots}{\langle {\cdot},{\cdot}\rangle }
\newcommand{\ee}{\varepsilon}

\DeclareMathOperator{\Sym}{Sym}

\DeclareMathOperator{\Ad}{Ad}

\DeclareMathOperator{\Cas}{Cas}
\newcommand{\pr}{\op{pr}}
\DeclareMathOperator{\st}{st}
\DeclareMathOperator{\St}{st}

\DeclareMathOperator{\Spec}{Spec}

\DeclareMathOperator{\Scal}{Scal}

\DeclareMathOperator{\Rc}{Rc}

\DeclareMathOperator{\vol}{vol}
\DeclareMathOperator{\rank}{rank}

\newcommand{\kil}{\operatorname{B}}

\newcommand{\PP}{\mathcal{P}} 
\newcommand{\RR}{\mathcal{R}} 

\renewcommand{\AA}{\mathbb{A}}

\newcommand{\DD}{\mathbb{D}} 
\newcommand{\EE}{\mathbb{E}}

\usepackage{hyperref}
\hypersetup{colorlinks=true,linkcolor=blue,citecolor=red}

\usepackage{backref}

\title[Linear stability of Perelman's $\nu$-entropy]
{Linear stability of Perelman's $\nu$-entropy of standard Einstein manifolds}

\author{Emilio~A.~Lauret}
\address{Instituto de Matemática (INMABB), Departamento de Matemática, Universidad Nacional del Sur (UNS)-CONICET, Bahía Blanca, Argentina.}
\email{emilio.lauret@uns.edu.ar}

\author{Alejandro Tolcachier}
\address{FAMAF, Universidad Nacional de Córdoba and CIEM-CONICET, Av. Medina Allende s/n, X5000HUA
	Córdoba, Argentina}
\email{atolcachier@unc.edu.ar}

\date{\today}

\begin{document}

\begin{abstract}
Paul Schwahn recently exhibited 112 non-symmetric, connected, simply connected, compact Einstein manifolds that are stable with respect to the total scalar curvature functional restricted to the space of Riemannian metrics with constant scalar curvature and fixed volume. 
This stability follows from the inequality $\lambda_L > 2E$, where $\lambda_L$ denotes the smallest eigenvalue of the Lichnerowicz Laplacian on TT-tensors and $E$ is the corresponding Einstein factor.  

In this paper, we estimate the smallest positive eigenvalue $\lambda_1$ of the Laplace-Beltrami operator for connected, simply connected, non-symmetric standard Einstein manifolds $(G/H,g_{\operatorname{st}})$ with $G$ a compact and connected simple Lie group. 
We obtain that $\lambda_1>2E$ for all of them excepting $7$ spaces. 
As a consequence of our estimates, we establish that all stable Einstein manifolds found by Schwahn are in fact linearly stable with respect to Perelman's $\nu$-entropy.
\end{abstract} 

\maketitle

\tableofcontents

\section{Introduction}

The spectrum of distinguished differential operators associated with Riemannian manifolds has a strong connection with the underlying topology and geometry. 
The most important case is the Laplace-Beltrami operator $\Delta_g$ associated to a Riemannian manifold $(M,g)$. 
When $M$ is compact and connected, the spectrum of $\Delta_g$ is non-negative and discrete:
\begin{equation*}
0=\lambda_0(M,g)<\lambda_1(M,g)\leq \lambda_2(M,g)
\leq \dots\leq \lambda_k(M,g)\underset{k\to\infty}{\longrightarrow}+\infty.
\end{equation*}
In particular, the multiplicity of every eigenvalue is finite. 

The value of the smallest positive eigenvalue $\lambda_1(M,g)$ of $\Delta_g$ is of particular importance.
Besides the famous conjecture due to Yau on the first eigenvalue of embedded minimal hypersurfaces of unit spheres, $\lambda_1$ is related for instance with the Ricci curvature via the Lichnerowicz-Obata estimate,
Poincaré inequalities \cite{Li82}, 
stability of harmonic maps \cite{Smith75}, 
bifurcation and local rigidity of solutions for the Yamabe problem \cite{BettiolPiccione13a},
among many others. 

The aim of this paper is to estimate $\lambda_1(M,g)$ for certain standard Einstein manifolds to decide their linear stability type of Perelman's $\nu$-entropy (see \cite{CaoHe15}) and their dynamical stability with respect to the Ricci flow (see \cite{Kroencke-thesis}). 
We first introduce a remarkable additional consequence of the $\lambda_1$-estimation which is very related with the subject of this article.

Let $M$ be a compact and connected differentiable manifold. 
It is well known that Einstein metrics $g$ on $M$ (i.e.\ $\Rc(g)=Eg$ for some $E\in\R$) are characterized as critical points of the total scalar curvature functional (also called Hilbert functional or Einstein-Hilbert functional)
\begin{equation*}
\widetilde\Scal(g):=\int_M\Scal(g)\, d\vol_g
\end{equation*}
restricted to the space of Riemannian metrics on $M$ of fixed volume. 
A natural (and classical) problem is to decide the nature of these critical points restricted to the space of Riemannian metrics $\mathcal C_v$ of constant scalar curvature and fixed volume $v$, by considering the second variation of $\widetilde\Scal$ (see \cite[Ch.~4]{Besse} or \cite[Ch.~2--4]{Kroencke-thesis}). 

It turns out that the stability type of an Einstein metric $g$ on $M$ as a critical point of $\widetilde\Scal|_{\mathcal C_v}$ is determined by the value of the smallest eigenvalue $\lambda_L$ of the Lichnerowicz Laplacian of $(M,g)$ restricted to TT-tensors (transversal and traceless symmetric $2$-tensors). 
An Einstein metric $g$ on $M$ is said to be
\begin{itemize}
\item unstable as a critical point of $\widetilde\Scal|_{\mathcal C_{\vol(g)}}$ (abbreviated \emph{H.unstable} following \cite{CaoHe15}) if and only if $\lambda_L<2E$. 
This means that, besides a subspace of finite codimension of TT-tensors where $\widetilde\Scal_g''$ is negative definite, there is a non-trivial TT-tensor $h$ such that $\widetilde\Scal_g''(h,h)>0$, so that $g$ is a saddle point of $\widetilde\Scal|_{\mathcal C_{\vol(g)}}$. 

\item stable as a critical point of $\widetilde\Scal|_{\mathcal C_{\vol(g)}}$ (abbreviated \emph{H.stable}) if and only if $\lambda_L>2E$. 
This means that $\widetilde\Scal_g''$ is negative definite on TT-tensors, so $g$ is a local maximum of $\widetilde\Scal|_{\mathcal C_{\vol(g)}}$.

\end{itemize}
Furthermore, we will use the following nomenclature:
\begin{itemize}
\item $g$ is \emph{neutrally H.stable} if  $\lambda_L=2E$. 
\item $g$ is \emph{H.semistable} if $\lambda_L\geq2E$. (This was called H.stable in \cite{CaoHe15}.)
\end{itemize}

J. Lauret introduced in \cite{JLauret-Stab3} a novel approach to studying this stability under the assumption that a compact Lie group $G$ acts transitively on $M$. 
Let $\lambda_L^G$ be the smallest positive eigenvalue of the Lichnerowicz Laplacian restricted to the space of $G$-invariant TT-tensors.
Given a $G$-invariant Einstein metric $g$, he called it 
\emph{$G$-stable} if $\lambda_L^G>2E$ (this implies that $g$ is a local maximum of $\widetilde\Scal$ restricted to the space $\mathcal M_{v}^G$ of $G$-invariant metrics of fixed volume $v=\op{vol}(g)$), 
\emph{$G$-unstable} if $\lambda_L^G<2E$ (this implies that $g$ is a saddle point or a local minimum of $\widetilde\Scal|_{\mathcal M_{v}^G}$), 
and \emph{neutrally $G$-stable} if $\lambda_L^G=2E$.

Since $\lambda_L\leq \lambda_L^G$, one has that a $G$-unstable Einstein metric is automatically H.unstable, and similarly, an H.stable $G$-invariant Einstein metric is $G$-stable. 
The articles \cite{JLauretWill-Stab2}, \cite{EJlauret-Stab}, \cite{GutierrezLauretJ25}, \cite{GutierrezV} complement \cite{JLauret-Stab3} determining the $G$-stability type of several families of homogeneous Riemannian manifolds.

The study of the H.stability type of irreducible compact symmetric spaces was started by Koiso~\cite{Koiso80}, and recently finished in \cite{SemmelmannWeingart22, Schwahn-symmetric}. 
The first non-symmetric stable Einstein manifold with positive scalar curvature was given in %by Schwahn, Semmelmann and Weingart
~\cite{SchwahnSemmelmannWeingart22}. 
This space is $\op{E}_7/\op{PSO}(8)$, which was shown to be $G$-stable by J.~Lauret and Will in \cite{JLauretWill-Stab2}. 

Although instability is considered by experts a generic property for compact Einstein manifolds with positive scalar curvature, Schwahn~\cite{Schwahn-Lichnerowicz} surprisingly found 112 non-symmetric stable Einstein manifolds with positive scalar curvature among non-symmetric normal homogeneous Einstein manifolds $G/H$ with $G$ simple. 
Note that in this situation, any normal metric on $G/H$ is a positive multiple of the \emph{standard} metric $g_{\st}$ on $G/H$, which is induced by the inner product given by the opposite of the Killing form $\kil_\fg$ of $\fg$. 

The spaces considered by Schwahn consist of (non-symmetric) strongly isotropy irreducible spaces (classified independently by Manturov~\cite{Manturov61a,Manturov61b,Manturov66}, Wolf~\cite{Wolf68}, and Krämer \cite{Kramer}) 
and standard Einstein manifolds $(G/H,g_{\st})$ with $G$ simple and reducible isotropy representation (classified by Wang and Ziller~\cite{WangZiller85}). 
For the latter ones, many of them were known to be $G$-stable from \cite{JLauret-Stab3,JLauretWill-Stab2,EJlauret-Stab}.

We are now in position to discuss the linear stability of Perelman's $\nu$-entropy. 
The corresponding second variation was computed by Cao, Hamilton and Ilmanen~\cite{CaoHamiltonIlmanen} (see also \cite{CaoZhu12}). 
In this case, the $\nu$-stability type of an Einstein metric different from a round sphere is characterized (see \cite[Thm.~1.1]{CaoHe15}) again by $\lambda_L$ (the smallest eigenvalue of the Lichnerowicz Laplacian restricted to TT-tensors) and also by the smallest positive eigenvalue $\lambda_1(M,g)$ of the Laplace-Beltrami operator associated to $(M,g)$ acting on functions of $M$. 

For the purposes of this paper, a closed Einstein manifold $(M,g)$ other than a round sphere with positive Einstein factor $E$ is called:
\begin{itemize}
\item \emph{$\nu$-stable} if $\lambda_L>2E$ and $\lambda_1(M,g)>2E$. 

\item \emph{$\nu$-unstable} if $\lambda_L<2E$ or $\lambda_1(M,g)<2E$. 

\item \emph{neutrally $\nu$-stable} if $\lambda_L\geq2E$ and $\lambda_1(M,g)\geq2E$, and at least one of these inequalities is attained. 

\item \emph{$\nu$-semistable} if $\lambda_L\geq2E$ and $\lambda_1(M,g)\geq2E$. 
\end{itemize} 
Clearly, an H.unstable Einstein manifold (not isometric to a round sphere) is $\nu$-unstable, and a (neutrally) $\nu$-stable Einstein manifold is (neutrally) H.stable. 

The fact $\lambda_1(M,g)<2E$ ensures the existence of \emph{$\nu$-unstable conformal directions}. 
The reason for this name is that these directions are in the tangent space of the space of Riemannian metrics conformal to $(M,g)$.

\begin{remark}\label{rem:spheres}
The decomposition of symmetric $2$-tensors of round spheres are special due to the existence of conformal Killing vector fields (see \cite[Rem.~2.4.2]{Kroencke-thesis}). 
Although they satisfy $\lambda_1(S^d,g_{\text{round}})=\frac{d}{d-1}E<2E$, they are $\nu$-stable in the sense of \cite{CaoHe15}, i.e.\ the second variation of its $\nu$-entropy is semi-negative definite. 
See also the discussion in \cite[\S3.2]{ButtsworthHallgren21}.
\end{remark}

Cao and He~\cite{CaoHe15} determined the $\nu$-stability type of all compact irreducible symmetric spaces. 
Notably, this work was before the H.stability type of these spaces was established, though the unknown cases $(M,g)$ at that time satisfied $\lambda_1(M,g)\leq 2E$, becoming $\nu$-unstable independently of their H.stability. 
Furthermore, in \cite{LauretRodriguez-hearingsymm1} and \cite{LauretRodriguez-hearingsymm2} have been found unstable conformal directions in four non-symmetric homogeneous Einstein manifolds on the underlying smooth manifold of compact irreducible symmetric spaces of rank greater than one.

The main goal of this paper is to study the $\nu$-stability type of standard Einstein manifolds $(G/H,g_{\st})$ with $G$ simple, as a continuation of Schwahn's recent study  \cite{Schwahn-Lichnerowicz} of their H.stability. 
We will estimate (and compute when is possible) the smallest positive Laplace eigenvalue $\lambda_1(G/H,g_{\st})$. 
For non-symmetric strongly isotropy irreducible spaces we obtain the following almost uniform lower bound. 

\begin{theorem}\label{thm4:main}
Let $G/H$ be a non-symmetric strongly isotropy irreducible space. 
Then, the smallest positive eigenvalue of the Laplace-Beltrami operator of the standard metric $g_{\st}$ in $G/H$ satisfies 
$$
\lambda_1(G/H,g_{\st})\geq 1>2E
$$ 
excepting for the spaces $\op{G}_2/\SU(3)$ and $\Spin(7)/\op{G}_2$ which satisfy
\begin{align*}
\lambda_1(\op{G}_2/\SU(3),g_{\st}) &=\frac12 <\frac56=2E,&
\lambda_1(\Spin(7)/\op{G}_2,g_{\st}) &= \frac{21}{40}<\frac{9}{10}=2E. 
\end{align*}
\end{theorem}

\begin{remark}
The exceptions in Theorem~\ref{thm4:main}, $(\op{G}_2/\SU(3),g_{\st})$ and $(\Spin(7)/\op{G}_2,g_{\st})$, are round spheres which are $\nu$-stable (see Remark~\ref{rem:spheres}).
\end{remark}

\begin{corollary}
There are no $\nu$-unstable conformal directions in every
non-symmetric strongly isotropy irreducible space endowed with its standard metric. 
\end{corollary}

We now move to the case when the isotropy representation is reducible. 
Although many of these spaces were shown to be $G$-unstable (c.f.\ \cite{EJlauret-Stab}), and consequently $\nu$-unstable, we estimate its first Laplace eigenvalue for all of them because of its importance on its own. 
Unlike in Theorem~\ref{thm4:main}, we assume that $G/H$ is simply connected in this case.

\begin{theorem}\label{thm5:main}
Let $(G/H,g_{\st})$ be a simply connected standard Einstein manifold with $G$ simple and reducible isotropy representation with Einstein factor $E$. 
Then, the smallest positive eigenvalue of the Laplace-Beltrami operator of the standard metric $g_{\st}$ in $G/H$ satisfies 
$$
\lambda_1(G/H,g_{\st})\geq 1>2E
$$ 
excepting the following cases:
\begin{itemize}
\item $\fg\simeq\spp(n)$, where $\lambda_1(G/H,g_{\st})=\frac{n}{n+1}$. 

\smallskip

\item $G/H=\Spin(8)/\op{G}_2$, where $\lambda_1(G/H,g_{\st})=\frac{7}{12}$. 

\smallskip

\item $G/H=\op{F}_4/\Spin(8)$, where $\lambda_1(G/H,g_{\st})=\frac{2}{3}$. 
\end{itemize}

Moreover, $\lambda_1(G/H,g_{\st})> 2E$ excepting the following cases:
\begin{itemize}
\item $\fg\simeq \spp(3)$ and $\fh\simeq \spp(1)\oplus\spp(1)\oplus \spp(1)$, where $\lambda_1(G/H,g_{\st})=\frac34<\frac78=2E$. 

\smallskip
\item $\fg\simeq \spp(2)$ and $\fh\simeq \spp(1)\oplus\ut(1)$, where $\lambda_1(G/H,g_{\st})=\frac23<\frac56=2E$. 

\smallskip
\item $\fg\simeq \spp(5)$ and $\fh\simeq \spp(2)\oplus\ut(3)$, where $\lambda_1(G/H,g_{\st})=\frac56=2E$. 

\smallskip
\item $G/H=\Spin(8)/\op{G}_2$, where $\lambda_1(G/H,g_{\st})=\frac{7}{12}<\frac{3}{5}=2E$. 

\smallskip
\item $G/H=\op{F}_4/\Spin(8)$, where $\lambda_1(G/H,g_{\st})=\frac{2}{3}<\frac89=2E$. 
\end{itemize} 
\end{theorem}

Notably, only five spaces of those considered in Theorem~\ref{thm5:main} have $\nu$-unstable conformal directions. 
All of them had already been established as $\nu$-unstable for being $G$-unstable (see \cite[Tables 1--2]{EJlauret-Stab} for the references).

We conclude from Theorems~\ref{thm4:main} and \ref{thm5:main} that each of the 112 H.stable (resp.\ 2 H.semistable, 2 neutrally H.stable) Einstein manifolds found by Paul Schwahn in \cite{Schwahn-Lichnerowicz} becomes $\nu$-stable (resp.\ $\nu$-semistable, neutrally $\nu$-stable). 
We next summarize, among all non-symmetric standard Einstein manifolds $(G/H,g_{\st})$ with $G$ simple, 
how many of them have had their $\nu$-stability type decided: 
\begin{itemize}
\item 
Among the ten infinite families of non-symmetric simply connected strongly isotropy irreducible spaces (see Table~\ref{table4:isotropyirred-families-lambda1}), the $\nu$-stability type is decided for 60 cases, always as $\nu$-stable. 
(At the bottom of page 93 in \cite{Schwahn-Lichnerowicz}, the number 51 corresponding to members of isotropy irreducible families should be 60 due to a miscounting.)

\item 
Among the 33 isolated non-symmetric simply connected isotropy irreducible spaces (see Tables~\ref{table4:isotropyirredexcepcion-classical}--\ref{table4:isotropyirredexcepcion-exceptional}), 
17 are $\nu$-stable,
and 1 is $\nu$-semistable. 

\item 
There are nine infinite families (XI--XIX) of isotropy reducible standard Einstein manifolds $(G/H,g_{\st})$ with $G$ simple (see Table~\ref{table5:isotropyred-families}). 
\begin{itemize}

\item Families XV and XVI contain 20 $\nu$-stable cases in total. 

\item Family XVII splits in Families XVIIa and XVIIb. 
Family XVIIa contains 2 neutrally $\nu$-stable cases. 
Family XVIIb has all its members $\nu$-unstable for being $G$-unstable (cf.~\cite{EJlauret-Stab}).

\item all members of the remaining 6 families are $G$-unstable (cf.~\cite{EJlauret-Stab}) and consequently $\nu$-unstable.
\end{itemize}

\item 
Among the 22 isolated isotropy reducible standard Einstein manifolds $(G/H,g_{\st})$ with $G$ simple (see Table~\ref{table5:isotropyredexcepcions}), 
15 are $\nu$-stable,
1 is $\nu$-semistable,
5 are $\nu$-unstable ($G$-unstable). 
\end{itemize}
The H.stability type was not decided for all the remaining cases.

\begin{remark}
A closed Einstein manifold $(M,g)$ (other than a round sphere) with Einstein factor $E$ is $\nu$-unstable if 
\begin{equation}\label{eq:lambda_L<2Eandlambda_1<2E}
\lambda_L(M,g)<2E
\qquad\text{or}\qquad
\lambda_1(M,g)<2E. 
\end{equation}
A natural question is which of these two conditions is more common. 

One can see from \cite[Tables 1--2]{CaoHe15} (and also from \cite[Thm.~1.1]{SemmelmannWeingart22} and \cite[Thm.~1.2]{Schwahn-symmetric}) that the only simply connected compact irreducible symmetric spaces satisfying exactly one of the conditions in \eqref{eq:lambda_L<2Eandlambda_1<2E} are the following:
\begin{itemize}
\item $\lambda_L\geq 2E$ and $\lambda_1<2E$:
$(\SU(n)\times\SU(n))/\SU(n)$ for $n\geq3$, 
$\SU(2n)/\Sp(n)$ for $n\geq3$,
$\Sp(n+1)/(\Sp(n)\times\Sp(1))$ for $n\geq 2$,
$\op{E}_6/\op{F}_4$, and 
$\op{F}_4/\Spin(9)$.

\item $\lambda_L<2E$ and $\lambda_1\geq 2E$:
$\SO(5)/(\SO(3)\times\SO(2))$, and
$\Sp(n)/\Ut(n)$ for $n\geq3$. 
\end{itemize}
%$$
%\begin{array}{l@{\qquad\qquad}l}
%\lambda_L\geq 2E\text{ and }\lambda_1<2E & 
%	\lambda_L<2E\text{ and }\lambda_1\geq 2E
%\\ \hline \rule{0pt}{13pt}
%(\SU(n)\times\SU(n))/\SU(n),\; n\geq3, 
%	&\SO(5)/(\SO(3)\times\SO(2))
%\\
%\SU(2n)/\Sp(n), \; n\geq3,
%	&\Sp(n)/\Ut(n),\; n\geq3.
%\\
%\Sp(n+1)/(\Sp(n)\times\Sp(1)),\; n\geq 2,
%\\
%\op{E}_6/\op{F}_4,\quad
%\op{F}_4/\Spin(9).
%\end{array}
%$$

The situation changes significantly among the spaces considered in this paper (standard Einstein manifolds $(G/H,g_{\st})$ with $G$ simple). 
None of them satisfy $\lambda_L\geq 2E$ and $\lambda_1<2E$. 
Furthermore, every H.unstable case in Tables~\ref{table5:isotropyred-families}--\ref{table5:isotropyredexcepcions}, excepting the five cases at the bottom of Theorem~\ref{thm5:main}, satisfies $\lambda_L<2E$ and $\lambda_1\geq 2E$, namely, Families XI--XIV (with $3$ exceptions), XVIIb, XVIII--XIX in Table~\ref{table5:isotropyred-families}, and isolated cases No.~35, 39, and 43 in Table~\ref{table5:isotropyredexcepcions}. 

We may conclude that, among the cases considered in this article, the condition $\lambda_L<2E$ occurs more often than $\lambda_1<2E$. 
\end{remark}

The $\nu$-stability has a strong relation with the dynamical stability of the Ricci flow. For instance, a $\nu$-unstable Einstein manifold $(M,g)$ is \emph{dynamically unstable} (see \cite[Thm.~1.3]{Kroencke15} and \cite[Cor.~6.2.5]{Kroencke-thesis}), which means that there is an ancient flow $g(t)$, $t\in (-\infty,T]$, such that $g(t)\to g$ as $t\to-\infty$ modulo diffeomorphisms, that is, there exists a family of diffeomorphisms $\varphi_t$, $t\in(-\infty,T]$, such that $\varphi_t^*g(t)\to g$ as $t\to-\infty$.  

Furthermore, a compact shrinking Ricci soliton (e.g. a positive Einstein manifold) $(M,g)$ is said to be \emph{dynamically stable} if for every neighbourhood $U$ of $g$ in the space of Riemannian metrics there exists a smaller neighbourhood $V\subset U$ such that the Ricci flow starting in $V$ stays in $U$ for all $t \geq 0$ and converges to an Einstein metric as $t \to+\infty$. 
Kröncke~\cite[Thm.~1.2]{Kroencke15} showed that a dynamically stable compact shrinking Ricci soliton is necessarily $\nu$-semistable (see also \cite[Thm.~1.5]{Kroencke20}). 
Therefore, the 112 $\nu$-stable positive Einstein manifolds shown in this paper are candidates to be dynamically stable with respect to the Ricci flow. 

\subsection*{Organization of the paper}
Section~\ref{sec:preliminaries} introduces Lie theoretical preliminaries. 
In Section~\ref{sec:spectralpreliminaries}, the spectrum of the Laplace-Beltrami operator of a standard Einstein manifold is presented in terms of representations of compact Lie groups, as well as several conditions to ensure the lower bound $\lambda_1\geq1$. 
Sections~\ref{sec:isotirred} and \ref{sec:isotred} provide the proofs of Theorem~\ref{thm4:main} and \ref{thm5:main} respectively.

\subsection*{Acknowledgments}
The authors wish to express their gratitude to Jorge Lauret and Paul Schwahn for several helpful comments on an earlier version of this article, and to the anonymous referee for many valuable suggestions that improved the readability of the paper.
The second author would like to express his gratitude to the Department of Mathematics at Universidad Nacional del Sur for its hospitality, where this work was initiated.

\section{Lie theoretical notations and conventions}\label{sec:preliminaries}

This section is intended for quick reference: readers may skim or skip it initially and return to it as needed to clarify any notation or conventions. 
The goal is to facilitate the reading of the subsequent sections. 
For general references, see for instance \cite{Bourbaki-Lie4-6}, \cite{BrockerDieck}, \cite{FultonHarris-book}, \cite{Knapp-book-beyond}, \cite{Sepanski}. 

As usual, Lie groups are denoted by capital letters and their Lie algebras by the corresponding lower case Gothic letter. 

\begin{notation}[Compact Lie group]\label{notation2:G}
Given a compact Lie group $G$, we associate the following objects with it: 
the rank of $G$ is given by $\rank(G)=\dim T$, where $T$ is any maximal torus of $G$; 
the complexification $\fg_\C:=\fg\otimes_\R\C$ of $\fg$;
the Killing form $\kil_{\fg}$ of $\fg$ or $\fg_\C$ depending on the context;
the universal enveloping algebra $\mathcal U(\fg_\C)$ of $\fg_\C$; and
if $\fg$ is semisimple, the Casimir element $\Cas_\fg:= \sum_{i}X_i^2\in\mathcal U(\fg_\C)$, where $\{X_i\}_i$ is any orthonormal basis of $\fg$ with respect to the inner product $-\kil_\fg$. 
\end{notation}

\begin{notation}[Maximal torus]\label{notation2:G-T}
Given a compact semisimple Lie group $G$ and a maximal torus $T$ of $G$, we associate with them: 
the root system $\Phi(\fg_\C,\ft_\C)$ of $\fg_\C$ with respect to the Cartan subalgebra $\ft_\C$ of $\fg_\C$ and
the lattice of \emph{$G$-integral weights} $\PP(G):=\{\mu\in\ft_\C^*: \mu(H)\in2\pi\mi\Z\;\; \forall\, H\in\ft\text{ with }\exp(H)=e\}$ of $G$, 
which is a sublattice of 
the lattice of \emph{algebraically integral weights} $\PP(\fg_\C):=\{\mu\in\ft_\C^*:2\frac{\kil_{\fg}^*(\mu,\alpha)}{\kil_{\fg}^*(\alpha,\alpha)}\in\Z\;\; \forall\,\alpha\in\Phi(\fg_\C,\ft_\C)\}$. 
\end{notation}

\begin{notation}[Extensions of the Killing form] \label{notation2:B^*}
In the situation of Notation~\ref{notation2:G-T}, the real form $\mi\ft$ of $\ft_\C$ is very special. 
On the one hand, $\alpha$ is real-valued on $\mi\ft$ for every $\alpha\in\Phi(\fg_\C,\ft_\C)$. %, or equivalently, $\Phi(\fg_\C,\ft_\C)\subset (\mi\ft)^*$. 

On the other hand, $\kil_\fg|_{\mi\ft}$ is positive definite. 
For $\mu\in(\mi\ft)^*\subset\ft_\C^*$, there exists $H_\mu\in \mi \ft$ determined by $\mu(H)=\kil_\fg(H,H_{\mu})$ for all $H\in\mi\ft$. 
We denote by $\kil_\fg^*(\cdot,\cdot)$ the corresponding inner product on $(\mi \ft)^*$ determined by $\kil_\fg$, that is, $\kil_\fg^*(\mu,\nu)=\kil_\fg(H_\mu, H_\nu)$. 
\end{notation}

\begin{notation}[Dominant weights] \label{notation2:dominant}
In the situation of Notation~\ref{notation2:G-T}, we fix a system of simple roots $\Pi(\fg_\C,\ft_\C)=\{\alpha_1,\dots,\alpha_n\}$ ($n=\rank(G)$).
We have the fundamental weights $\omega_1,\dots,\omega_n$ determined by $2\frac{\kil_\fg^*(\omega_i,\alpha_j)} {\kil_\fg^*(\alpha_j,\alpha_j)}=\delta_{ij}$,
the set of positive roots $\Phi^+(\fg_\C,\ft_\C)$, 
the set of dominant algebraically integral weights $\PP^+(\fg_\C)$, 
the set of dominant $G$-integral weights $\PP^+(G)$, 
and
$\rho_\fg:=\frac12\sum_{\alpha\in\Phi^+(\fg_\C,\ft_\C)} \alpha$. 
\end{notation}

One has that 
	$\rho_\fg=\omega_1+\dots+\omega_n$,
	$\PP(\fg_\C)=\bigoplus_{i=1}^n\Z\, \omega_i$, and
	$\PP^+(\fg_\C)=\bigoplus_{i=1}^n\Z_{\geq0}\, \omega_i$.

\begin{notation}[Representation] \label{notation2:pi}
Given a compact, connected and semisimple Lie group $G$, 
a maximal torus $T$ of $G$, 
and $\pi$ a finite-dimensional complex representation of $G$, 
we denote by $V_\pi$ the underlying vector space of $\pi$ (i.e.\ $\pi:G\to\GL(V_\pi)$ is a continuous morphism of groups),  and $\pi$ also denotes the induced representation of $\fg$, $\fg_\C$, and $\mathcal U(\fg_\C)$. 
$\PP(\pi)$ stands for the set of weights of $\pi$ (i.e.\ the subset of $\PP(G)$ with positive multiplicity in $\pi$), 
the differential of $\pi$ is $d\pi:\fg\to\gl(V_\pi)$ and its complexified differential is $(d\pi)_\C:\fg_\C\to\gl(V_\pi)$.

Let $\widehat G$ denote the unitary dual of $G$, that is, the set of equivalence classes of irreducible unitary representations of $G$. 
Also, $1_G\in\widehat G$ is the trivial representation of $G$.

\end{notation}

\begin{remark}[Highest Weight Theorem]\label{rem2:HighestWeightTheorem}
In the situation of Notation~\ref{notation2:dominant}, the Highest Weight Theorem ensures that $\widehat G$ is in correspondence with $\PP^+(G)$. 
More precisely, denoting by $\Lambda_\pi$ the highest weight of $\pi\in\widehat G$ (i.e.\ the largest element in $\PP(\pi)$ with respect to the order induced by $\Pi(\fg_\C,\ft_\C)$), the map $\pi\mapsto\Lambda_\pi$ is a bijection from $\widehat G$ onto $\PP^+(G)$. 
Similarly, at the Lie algebra level, $\PP^+(\fg_\C)$ is in correspondence with the finite-dimensional irreducible representations of the complex semisimple Lie algebra $\fg_\C$. 
\end{remark}

\begin{table}[!htbp]
\renewcommand{\arraystretch}{1.4}
\caption{Data of irreducible root systems}\label{table:Bourbaki}
$
\begin{array}{cccc}
\hline\hline
\text{Type} &  \text{Fundamental weights} &\rho_\fg&\kil_\fg^*(\ee_i,\ee_j) 
\\ \hline\hline \rule{0pt}{14pt}
A_n &  \omega_i=\pr(\ee_1+\dots+\ee_i) &
\sum\limits_{j=1}^{n+1}\frac{n+2-2j}{2}\ee_j&
\frac{1}{2(n+1)}
\\[3mm]\hline
B_n  &
\begin{array}[t]{l}
\omega_i=\varepsilon_1+ \dots+ \varepsilon_i,\, i\leq n-1,\\
\omega_n=\tfrac12( \varepsilon_1+ \dots+ \varepsilon_{n})
\end{array}
&
\sum\limits_{j=1}^{n}(n-\frac{2j-1}{2})\ee_j&
\frac{1}{2(2n-1)}
\\[3mm]\hline
C_n
&
\omega_i=\ee_1+\dots+\ee_i
&
\sum\limits_{j=1}^{n}(n+1-j)\ee_j&
\frac{1}{4(n+1)}
\\[3mm]\hline
D_n
& 
\text{\small$
\begin{array}[t]{r@{\,}l}
\omega_i&= \varepsilon_1+ \dots+ \varepsilon_i,\;  i\leq n-2,\\
\omega_{n-1}&=\tfrac12( \varepsilon_1+ \dots+ \varepsilon_{n-1}-\varepsilon_{n}),\\
\omega_{n}&=\tfrac12( \varepsilon_1+ \dots+\ee_{n-1}+ \varepsilon_{n})
\end{array}
$}
&
\sum\limits_{j=1}^{n}(n-j)\ee_j&
\frac{1}{4n-4}
\\[3mm]\hline
E_8
&\text{\scriptsize$
\begin{array}[t]{l}
\omega_1=2\ee_8,\quad 
\omega_2=\frac12(\sum\limits_{i=1}^7\ee_i+5\ee_8),\\
\omega_3=\frac12(-\ee_1+\sum\limits_{i=2}^7\ee_i+7\ee_8),\\
\omega_i=\sum\limits_{j=i-1}^7\ee_i+(9-i)\ee_8, \; 4\leq i\leq 8
\end{array}
$}
&
\sum\limits_{i=1}^7 (i-1)\ee_i+23\ee_8
&
\frac{1}{60}
\\[3mm]\hline
\op{E}_7
&\text{\scriptsize $
\begin{array}[t]{l}
\omega_1=\ee_8-\ee_7,\quad %\\
\omega_2=\frac12(\sum\limits_{i=1}^6\ee_i-2\ee_7+2\ee_8),\\
\omega_3=\frac12(-\ee_1+\sum\limits_{i=2}^6\ee_i-3\ee_7+3\ee_8),\\
\omega_i=\sum\limits_{j=i-1}^6\ee_i+\frac{8-i}{2}(\ee_8-\ee_7), \; 4\leq i\leq 7
\end{array}
$}
&
\begin{array}{l}
\sum\limits_{i=1}^6 (i-1)\ee_i-\frac{17}{2}\ee_7+\frac{17}{2}\ee_8
\end{array}
&
\frac{1}{36}
\\[3mm]\hline
\op{E}_6
&\text{\scriptsize$
\begin{array}[t]{l}
\omega_1=\frac23(\ee_8-\ee_7-\ee_6), \quad %\\
\omega_2=\frac12\big( 
	\sum\limits_{i=1}^5\ee_i
	-\ee_6-\ee_7+\ee_8\big),\\
\omega_3=\frac{5}{6}(\ee_8-\ee_7-\ee_6) 	+\frac12(-\ee_1+\sum\limits_{i=2}^5\ee_i),\\
\omega_i=\frac{7-i}{3}(\ee_8-\ee_7-\ee_6)
+ \sum\limits_{j=i-1}^5\ee_i, \; 4\leq i\leq 6
\end{array}
$}
&
\begin{array}[t]{l}
\sum\limits_{i=1}^5 (i-1)\ee_i\\+4(\ee_8-\ee_7-\ee_6)
\end{array}
&
\frac{1}{24}
\\[3mm]\hline
F_4
&
\begin{array}[t]{l}
\omega_1=\ee_1+\ee_2,\quad %\\
\omega_2=2\ee_1+\ee_2+\ee_3,\\
\omega_3=\tfrac12(3\ee_1+\ee_2+\ee_3+\ee_4),\quad %\\
\omega_4=\ee_1
\end{array}
&
11\ee_1+5\ee_2+3\ee_3+\ee_4&
\frac{1}{18}
\\[3mm]\hline
G_2
&
%\begin{array}[t]{l}
\omega_1=-\varepsilon_2+\varepsilon_3,\quad 
%\\
\omega_2=-\varepsilon_1-\varepsilon_2+2\varepsilon_3
%\end{array}
&
	-\ee_1-2\ee_2+3\ee_3&
\frac{1}{24}
\\ \hline 
\end{array}$
\end{table}

\begin{notation}[Irreducible root systems]\label{notation2:Bourbaki}

The simply connected compact simple Lie groups of classical types $A_n$, $B_n$, $C_n$, $D_n$ are respectively $\SU(n+1)$, $\Spin(2n+1)$, $\Sp(n)$, and $\Spin(2n)$. 
For the exceptional cases, we denote by $\op{E}_n$,  $\op{F}_4$, and $\op{G}_2$ the simply connected compact simple Lie groups of types $E_n,F_4$ and $G_2$, respectively ($n=6,7,8$), and by $\fe_n,\ff_4,\fg_2$ their corresponding Lie algebras.

Let $\fg_\C$ be the complexified (simple) Lie algebra of a Lie group as above, and let $\ft_\C$ be a Cartan subalgebra of $\fg_\C$. 
We take a basis $\ee_1,\ee_2,\dots$ of $(\mi\ft)^*$ (or sometimes of some extension) as in Bourbaki's book \cite[\S{}VI.4]{Bourbaki-Lie4-6}, as well as the particular ordering of simple roots. 
The corresponding fundamental weights are as in Table~\ref{table:Bourbaki}. 

Only for type $A_n$, we define the projection $\pr:\bigoplus_{i=1}^{n+1}\R\ee_i\to
(\mi\ft)^* = \{\sum_{i=1}^{n+1}a_i\ee_i\in \bigoplus_{i=1}^{n+1}\R\ee_i: \sum_{i=1}^{n+1}a_i=0\}$ given by 
\begin{equation*}
\pr\left(\sum_{i=1}^{n+1}a_i\ee_i \right)= \sum_{i=1}^{n+1}a_i \ee_i-\frac{a_1+\dots+a_{n+1}}{n+1} (\ee_1+\dots+\ee_{n+1}). 
\end{equation*}
\end{notation}

\begin{remark}[Integral weight lattices]\label{rem2:integralweightlattice}
In the situation of Notation~\ref{notation2:G-T} (so $\fg$ is a compact semisimple Lie algebra), let $G_{sc}$ denote the unique (up to isomorphism) connected and simply connected Lie group with Lie algebra $\fg$, which is compact. 
Connected Lie groups having Lie algebra $\fg$ are in correspondence with subgroups of the center $Z=Z(G_{sc})$ (which is finite and abelian) of $G_{sc}$ as follows: 
given $\Gamma$ a subgroup of $Z$, $G_{\Gamma}:=G_{sc}/\Gamma$ is a connected Lie group with Lie algebra $\fg$ and fundamental group isomorphic to $\Gamma$. 

At one extreme we have that $\PP(G_{sc})=\PP(\fg_\C)$. 
At the other, the integral weight lattice of the (centerless) compact Lie group $G_Z=G_{sc}/Z$ (called the \emph{adjoint group} of $\fg$) coincides with the \emph{root lattice} $\RR(\fg_\C):=\Span_\Z( \Phi(\fg_\C,\ft_\C))$.
One has that $\PP(\fg_\C)/\RR(\fg_\C)\simeq Z$. 

In general, given $\Gamma$ a subgroup of $Z$, the weight integral lattice $\PP(G_\Gamma)$ of $G_\Gamma$ satisfies 
$\RR(\fg_\C)\subset \PP(G_\Gamma)\subset \PP(\fg_\C)$,  
$\PP(\fg_\C)/\PP(G_\Gamma)\simeq \Gamma$ and
$\PP(G_\Gamma)/\RR(\fg_\C)\simeq Z/\Gamma$.

From now on we assume that $\fg_\C$ is simple of rank $n$. 
Let $V$ be the extension of $(\mi\ft)^*$ mentioned in Notation~\ref{notation2:Bourbaki}, i.e.\
$V=(\mi\ft)^*$ for types $B_n,C_n,D_n,F_4$, 
$V=\bigoplus_{i=1}^{n+1}\R\ee_i$ for type $A_n$, 
$V=\bigoplus_{i=1}^{8}\R\ee_i$ for type $E_n$ with $n=6,7,8$, and
$V=\bigoplus_{i=1}^{3}\R\ee_i$ for type $G_2$. 
When $\dim V=n$, we identify $\Z^n=\bigoplus_{i=1}^n\Z\ee_i$. 
Let $\innerdots$ denote the inner product on $V$ with orthonormal basis $\{\ee_1,\ee_2,\dots\}$. 
Given a (full) lattice $L$ on $V$, we set $L^*=\{\mu\in V: \inner{\mu}{\nu}\in \Z\;\;\forall\,\nu\in L\}$, the dual lattice of $L$. 
For instance, $(\Z^n)^*=\Z^n$, and the root lattice 
\begin{equation}
\DD_n:=\RR(\so(2n)_\C)=
\left\{{\textstyle \sum\limits_{i=1}^{n}}a_i\ee_i: a_i\in\Z\;\forall \,i,\; \; {\textstyle \sum\limits_{i=1}^{n}} a_i\text{ is even}\right\}
\end{equation}
has dual lattice $\DD_n^* =\Span_\Z(\Z^n,\tfrac12(\ee_1+\dots+\ee_n))$.

We next describe the integral weight lattice of all compact connected simple Lie groups. 

If $\fg_\C$ is of type $E_8$, $Z\simeq\{e\}$, and $\PP(\op{E}_8)=\RR(\fg_\C)=:\EE_8$. 
One can check that $\EE_8^*=\EE_8=\Span(\DD_8,\tfrac12(\ee_1+\dots+\ee_8))$.

If $\fg_\C$ is of type $E_7$, $Z\simeq\Z_2$,  $\EE_7:=\PP(\op{E}_7/\Z_2)=\RR(\fg_\C)=\EE_8\cap \{\ee_7+\ee_8\}^\perp$, and $\PP(\op{E}_7)=\EE_7^*$. 

If $\fg_\C$ is of type $E_6$, $Z\simeq\Z_3$, 
$\EE_6:=\PP(\op{E}_6/\Z_3)=\RR(\fg_\C)=\EE_8\cap\{\ee_7+\ee_8,\, \ee_6+\ee_7+2\ee_8\}^\perp$, and 
$\PP(\op{E}_6)=\EE_6^*$. 

If $\fg_\C$ is of type $F_4$, $Z\simeq\{e\}$, and $\PP(\op{F}_4)=\DD_4^*$. 

If $\fg_\C$ is of type $G_2$, $Z\simeq\{e\}$, and $\PP(\op{G}_2)=\{\sum_{i=1}^{3}a_i\ee_i: a_i\in\Z\;\;\forall\,i,\, \sum_{i=1}^3 a_i=0\}=:\AA_2$.

If $\fg_\C$ is of type $C_n$, $G_{sc}=\Sp(n)$, $Z=\{\pm\Id\}\simeq\Z_2$, 
$\PP(\Sp(n)/\Z_2)=\RR(\fg_\C)=\DD_n$, and 
$\PP(\Sp(n))=\Z^n$. 

If $\fg_\C$ is of type $B_n$, $G_{sc}=\Spin(2n+1)$, $Z=\{\pm1\}\simeq\Z_2$, $\Spin(2n+1)/Z\simeq\SO(2n+1)$, 
$\PP(\SO(2n+1))=\RR(\fg_\C)=\Z^n$, and 
$\PP(\Spin(2n+1))=\DD_n^*$.

Suppose that $\fg_\C$ is of type $D_n$, thus $G_{sc}=\Spin(2n)$.
There is an element $z\in\Spin(2n)$ ($z=e_1\dots e_{2n}$ in the notation of \cite[Ex.~1.35]{Sepanski}) such that $Z=\{\pm1,\pm z\}$ and $z^2=(-1)^n$.
Hence $Z=\langle z\rangle\simeq \Z_4$ if $n$ is odd and $Z=\langle -1\rangle\oplus\langle z\rangle\simeq\Z_2\oplus\Z_2$ if $n$ is even.  
One has $\Spin(2n)/\{\pm1\}\simeq\SO(2n)$, 
$\PP(\Spin(2n)/Z)=\RR(\fg_\C)=\DD_n$,
$\PP(\Spin(2n))=\DD_n^*$, and
$\PP(\SO(2n))=\Z^n$. 
In addition, if $n$ is even, $\PP(\Spin(2n)/\{1,\pm z\}) = \Span_\Z(\DD_n,\tfrac12(\ee_1+\dots+\ee_{n-1}\pm\ee_n) )=:\DD_n^{\pm}$.

We now suppose that $\fg_\C$ of type $A_n$. 
For future convenience, we write $N=n+1$, thus $\fg_{\C}$ is of type $A_{N-1}$ and $G_{sc}=\SU(N)$. 
We have $\AA_{N-1}:=\PP(\SU(N)/\Z_N)= \RR(\fg_\C)=
\{\sum_{i=1}^{N}a_i\ee_i: a_i\in\Z\;\;\forall\,i,\, \sum_{i=1}^N a_i=0\}$ and $\PP(\SU(N))=\AA_{N-1}^*=\Span_\Z(\AA_{N-1}, \pr(\ee_1))$.
The rest of the integral weight lattices are more delicate. 

The center of $\SU(N)$ is 
$
Z=\big\{e^{\frac{2h\pi\mi}{N}}\Id:0\leq h<N \big\}
\simeq\Z_{N}
.
$
For each divisor $d$ of $N$, we set 
\begin{equation}
\Gamma_d= \big\{e^{\frac{2h\pi\mi}{d}}\Id_N:0\leq h<d \big\}
\simeq\Z_{d}
, 
\end{equation}
the only subgroup of $Z$ of order $d$. 
We have that 
\begin{equation}\label{eq:A_N[d]}
\PP(\SU(N)/\Gamma_d)=
\bigcup_{{1\leq p\leq N \,:\, d\mid p}} (\omega_p+\AA_{N-1})
,
\end{equation}
where $\omega_1,\dots,\omega_{N-1}$ are the fundamental weights of type $A_{N-1}$ shown in Table~\ref{table:Bourbaki} and $\omega_N=0$. 
(It is a simple matter to check that the right-hand side of \eqref{eq:A_N[d]} is a lattice.)
\end{remark}

\begin{notation}[Closed subgroup] \label{notation2:G-H}
Given a closed subgroup $H$ of a compact semisimple Lie group $G$, we associate the following objects with it:
$\fp$ denotes the orthogonal complement of $\fh$ in $\fg$ with respect to $\kil_\fg$ (see Notation~\ref{notation2:G}), so
$\fg=\fh{\oplus} \fp$
and $\fp$ is naturally identified with $T_{eH}G/H$; 
the spherical representations of the pair $(G,H)$ are $\widehat G_H= \{\pi\in\widehat G: \dim V_\pi^H> 0\}$, where $V_{\pi}^H$ are the vectors fixed in the representation by $H$.

There is a maximal torus $T$ of $G$ such that $T_H:=T\cap H$ is a maximal torus of $H$. 
We will always request this property to $T$ in this situation. 

The next table describes the notation used for each group in irreducible representations, the highest weights, the fundamental weights, and the basis of $(\mi\ft)^*$ or $\big(\mi(\ft\cap\fh)\big)^*$ (or some extension) as in Notation~\ref{notation2:Bourbaki}:
\begin{equation*}
\begin{array}{ccccc}
\text{group} &\text{irred.\ rep.} & \text{highest weight} & \text{fund.\ weights} & 
\text{Not.~\ref{notation2:Bourbaki}}
\\ \hline
G & \pi & \Lambda & \omega_i & \ee_i\\
H & \sigma & \zeta & \eta_i & \xi_i\\
\end{array}
.
\end{equation*}

When $\fh$ is semisimple, we will use prime symbols to designate the second factor, double prime symbols to designate the third one, and so on. 
For instance, when $\fh=\fh_1\oplus\fh_2$ is the decomposition of $\fh$ in simple ideals, 
any irreducible representation of $\fh$ is of the form $\sigma\widehat\otimes \,\sigma'$ with $\sigma$ an irreducible representation of $\fh_1$ and $\sigma'$ an irreducible representation of $\fh_2$.
Similarly, 
the highest weights are of the form $\zeta+\zeta'$, 
the fundamental weights are $\eta_i,\eta_j'$, 
and the basis of $\big(\mi (\ft\cap\fh)\big)^*=\big(\mi(\ft\cap\fh_1)\big)^*\oplus \big(\mi(\ft\cap\fh_2)\big)^*$ is $\xi_i,\xi_j'$. 

The notation $\pi_{\Lambda}$ (resp.\ $\sigma_\zeta$) will be used for irreducible representations of the compact Lie group $G$ (resp.\ $H$) as well as for $\fg_\C$ (resp.\ $\fh_\C$). 
\end{notation}

In the situation of Notation~\ref{notation2:G-H}, 
it will be useful the identity 
\begin{equation}\label{eq2:dual}
\dim V_\pi^H= \dim V_{\pi^*}^H
\qquad\forall\,\pi\in\widehat G. 
\end{equation}
Here and subsequently, $\pi^*$ stands for the dual (or contragradient) representation of $\pi$.

\begin{notation}[Embedding] \label{notation2:embedding}
Let $H$ be a closed subgroup of a compact Lie group $G$. 
Several times, the subgroup $H$ will be the image of a Lie group homomorphism $\rho:H'\to H\subset G$ from an abstract compact Lie group $H'$, such that the complexified differential $(d\rho)_\C :\fh_\C'\to\gl(V_\rho)$ is a finite-dimensional representation of $\fh_\C'$. 
Although we have 
$$
d\rho(\fh')=\fh\subset \fg\subset \gl(V_\rho),
$$
the analogous inclusions at the Lie group level are not valid in general. 
In other words, $G$ will not be necessarily a subgroup of $\GL(V_\rho)$. 
In this case, we will always take a maximal torus $T'$ of $H'$ such that $d\rho(\ft')=\ft\cap\fh$.
\end{notation}

\begin{remark}[Representations of real, complex and quaternionic types]\label{rem2:types-of-representations}
Let $H'$ be a compact Lie group and let $\rho:H'\to \GL(V_\rho)$ be a complex finite-dimensional representation of $H'$ of dimension $N$. 
Since $H'$ is compact, $V_\rho$ admits an $H'$-invariant complex inner product $\innerdots$. 
One can always choose one of them such that 
$$
\rho(H')\subset \SU(V_\rho)\simeq\SU(N). 
$$ 

A \emph{real structure} (resp.\ \emph{quaternionic structure}) on $\rho$ is a conjugate-linear $H'$-map $\mathcal J:V_\rho\to V_\rho$ such that $\mathcal J^2=\Id_{V_\rho}$ (resp.\ $\mathcal J^2=-\Id_{V_\rho}$). 
If $\rho^*\not\simeq\rho$, then $\rho\oplus\rho^*$ admits real and quaternionic structures simultaneously. 
\begin{itemize}
\item $\rho\in\widehat{H'}$ is said to be of \emph{real type} if it admits a real structure $\mathcal J$. 
We denote by $V_{\rho}^\pm$ the eigenspace of $\mathcal J$ with eigenvalue $\pm1$. $V_{\rho}^\pm$ is a ${\rho}(H')$-invariant real subspace of $V_{\rho}$ of dimension $N$, $\innerdots|_{V_\rho^\pm}$ is a real inner product on $V_\rho^\pm$, and $V_{\rho} =V_{\rho}^+\oplus V_{\rho}^-$. 
Thus 
$$
{\rho}(H')\subset \SO(V_{\rho}^+)\simeq\SO(N)
.
$$

\item $\rho\in\widehat{H'}$ is said to be of \emph{quaternionic type} if it admits a quaternionic structure $\mathcal J$. 
Then $N$ is even and
$$
{\rho}(H')\subset \Sp(V_{\rho},\mathcal J)=\{T\in \SU(V_{\rho}):T^t\mathcal JT=\mathcal J \}\simeq \Sp(\tfrac{N}{2})
.
$$

\item $\rho\in\widehat{H'}$ is said to be of \emph{complex type} if it is not of real or quaternionic type. 
\end{itemize}
Any $\rho\in\widehat{H'}$ is either of real, or quaternionic, or complex type. 
Furthermore, 
$\rho$ is of complex type if and only if $\rho\not\simeq\rho^*$.

Suppose $\fh'\simeq\fh_1'\oplus\fh_2'$ with $\fh_1',\fh_2'$ simple. 
If $\rho_i$ is a finite-dimensional irreducible representation of $(\fh_i')_\C$ for each $i=1,2$, then the irreducible representation $\rho_1\widehat\otimes\rho_2$ of $\fh_\C'$ is 
of complex type if $\rho_1$ or $\rho_2$ is of complex type, 
of real type if $\rho_1,\rho_2$ are both simultaneously of real or quaternionic type, and 
of quaternionic type if $\rho_i^*\simeq\rho_i$ for $i=1,2$, and $\rho_1,\rho_2$ are of different type (i.e.\ one is of real type and the other of quaternionic type). 
This can be easily extended to an arbitrary number of simple factors. 

For more information about representations of real, complex and quaternionic type, see \cite[\S{}II.6]{BrockerDieck} or \cite[\S26.3]{FultonHarris-book}.
\end{remark}

\begin{remark}[Exterior and symmetric product]

Let $\fl$ be an arbitrary compact Lie algebra.
Given $\theta$ a finite-dimensional representation of $\fl_\C$, we denote by $\bigwedge^p(\theta)$  (resp.\ $\Sym^k(\theta)$) the $p$-exterior (resp.\ the $k$-symmetric) power of $\theta$, which has underlying vector space equal to $\bigwedge^p (V_\theta)$ (resp.\ $\Sym^k(V_\rho)$), for any integer $0\leq p\leq \dim V_\theta$ (resp.\ $k\geq0$). 
Clearly we have that $\Lambda^0(\theta) =\Sym^0(\theta) =1_{\fl_\C}$ and $\Lambda^1(\theta) =\Sym^1(\theta) =\theta$.  

Given finite-dimensional representations $\theta_1,\theta_2$ of $\fl$, we have that
\begin{equation}
\label{eq2:Lambda^2(V+W)-Sym^2(V+W)}
\begin{aligned}
{\textstyle \bigwedge^2}(\theta_1\oplus \theta_2) &
\simeq {\textstyle \bigwedge^2}(\theta_1)
\oplus {\textstyle \bigwedge^2}(\theta_1)
\oplus \theta_1\otimes \,\theta_2
,
\\
\Sym^2(\theta_1\oplus \theta_2) &
\simeq \Sym^2(\theta_1)
\oplus \Sym^2(\theta_1)
\oplus \theta_1\otimes \,\theta_2
,
\end{aligned}
\end{equation}
as $\fl_\C$-modules.

Now, let $\fl_1, \fl_2$ be arbitrary compact Lie algebras and $\tau_1,\tau_2$ finite-dimensional representations of $(\fl_1)_\C,(\fl_2)_\C$ respectively. 
It is well known that
\begin{equation}
\label{eq2:Lambda^2(VxW)-Sym^2(VxW)}
\begin{aligned}
{\textstyle \bigwedge^2}(\tau_1\widehat\otimes \, \tau_2) &
\simeq {\textstyle \bigwedge^2}(\tau_1) \widehat\otimes \, \Sym^2(\tau_2)
\oplus \Sym^2(\tau_1) \widehat\otimes \, {\textstyle \bigwedge^2}(\tau_2)
,
\\
\Sym^2(\tau_1\widehat\otimes \, \tau_2) &
\simeq \Sym^2(\tau_1) \widehat\otimes \, \Sym^2(\tau_2)
\oplus {\textstyle \bigwedge^2}(\tau_1) \widehat\otimes \, {\textstyle \bigwedge^2}(\tau_2)
,
\end{aligned}
\end{equation}
\end{remark}

\section{Spectral estimates for standard manifolds}
\label{sec:spectralpreliminaries}

This section describes the spectrum of the Laplace-Beltrami operator associated to a standard manifold and it provides several estimates of its smallest positive eigenvalue in terms of representation theory of compact Lie groups.

\subsection{Spectra of standard homogeneous manifolds}

A Riemannian manifold $(M,g)$ is called \emph{homogeneous} if its isometry group acts transitively on $M$. If there is an action by isometries of a Lie group $G$ on $M$ that is transitive, then $(M,g)$ is called \emph{$G$-homogeneous}. 
In this case, $M$ is diffeomorphic to $G/H$, where $H$ is the isotropy subgroup of the $G$-action at some element in $M$, and $(M,g)$ is isometric to $(G/H,g')$ for some $G$-invariant metric $g'$ on $M$, that is, $g'_{abH}\big(dL_a(v), dL_a(w)\big) = g'_{bH}\big(v,w\big)$ for every $v,w\in T_{bH}G/H$ and $a,b\in G$. 

A $G$-homogeneous Riemannian manifold $(G/H,g)$ (here, $g$ is $G$-invariant) is called \emph{standard} if the inner product $g_{eH}$ on the tangent space at the point $eH$ coincides with $-\kil_{\fg}|_{\fp}$ (see Notations~\ref{notation2:G} and \ref{notation2:G-H}). 
Note that $G$ must be compact and semisimple in order $-\kil_{\fg}|_{\fp}$ is positive definite (see \cite[\S7.G]{Besse} for further details). 
In addition, we assume that $G$ is connected. 

Every Riemannian manifold has a distinguished differential operator which is called the \emph{Laplace-Beltrami operator}. 
We next describe its spectrum on a standard Riemannian manifold $(G/H,g_{\st})$.

Since $\Cas_\fg$ (see Notation~\ref{notation2:G}) is in the center of $\mathcal U(\fg_\C)$, $(d\pi)_\C(\Cas_\fg)$ commutes with $\pi(a)$ for all $a\in G$. 
Schur's Lemma implies that $(d\pi)_\C(\Cas_\fg)$ acts on $V_\pi$ by an scalar $\lambda^\pi$, for any $\pi\in\widehat G$, that is,
\begin{equation*}
(d\pi)_\C(\Cas_\fg)\cdot v=\lambda^\pi\, v
\qquad\forall\, v\in V_\pi.
\end{equation*}  
Moreover, Freudenthal's formula yields that (see Notations~\ref{notation2:G-T} and \ref{notation2:B^*}) 
\begin{equation}\label{eq3:Freudenthal}
\lambda^\pi=\kil_\fg^*(\Lambda_\pi,\Lambda_\pi+2\rho_\fg)
.
\end{equation}

\begin{theorem}\label{thm:Spec(standard)}
Let $(G/H,g_{\st})$ be a connected standard homogeneous space. 
The spectrum of its associated Laplace-Beltrami operator is the multiset given by 
$$
\Spec(G/H,g_{\st})
:=\Big\{\!\!\Big\{ 
	\underbrace{\lambda^\pi,\dots,\lambda^\pi}_{d_\pi\times d_\pi^H\text{-times} } 
	: \pi\in\widehat G_H 
\Big\}\!\!\Big\}
,
$$
where $d_\pi=\dim V_\pi$ and $d_\pi^H=\dim V_\pi^H$. 
In particular, the smallest positive eigenvalue is given by 
$$
\lambda_1(G/H,g_{\st})
=\min \left\{\lambda^\pi: \pi\in\widehat G_H\smallsetminus\{1_G\}\right\}
.
$$
\end{theorem}

This result is well known; see for instance \cite[\S5.6]{Wallach-book}, or \cite[\S2]{BLPhomospheres} for a more general situation. 
The second assertion follows immediately from the first one since, for any $\pi\in\widehat G$, $\lambda^\pi>0$ if and only if $\pi\not\simeq1_G$.

\subsection{Lower bounds for the first eigenvalue} 
A Riemannian manifold $(M,g)$ is called \emph{Einstein} if its Ricci curvature is a multiple of the metric $g$, that is, $\op{Rc}(g)=Eg$ for some $E\in\R$. 
Wang and Ziller~\cite[Cor.~1.6]{WangZiller85} proved that the \emph{Einstein constant} $E$ of a standard Einstein manifold $(G/H,g_{\st})$  satisfies 
\begin{equation}\label{eq:1/4<E<1/2}
\frac12\leq  2E\leq 1.
\end{equation}
Moreover, $2E=1$ if and only if $G/H$ is locally symmetric, and similarly, $2E=\frac12$ if and only if $H=\{1\}$. 

Our main interest is to decide whether $\lambda_1(G/H,g_{\st})>2E$ holds or not for each non-symmetric standard Einstein manifold $(G/H,g_{\st})$ with $G$ simple. 
It follows immediately from Theorem~\ref{thm:Spec(standard)} the following sufficient condition (via \eqref{eq:1/4<E<1/2}) for a standard Einstein manifold $(G/H,g_{\st})$ that is not locally symmetric:
\begin{equation}\label{eq:sufficientcondition}
\text{If $V_{\pi}^{H}=0$ for all $\pi\in\widehat G\smallsetminus\{1_G\}$ satisfying $\lambda^\pi<1$, then  $\lambda_1(G/H,g_{\st})\geq 1$. }
\end{equation}

Since $G$ will be always simple in our cases of study, it will be useful to identify which are the dominant algebraically integral weights $\Lambda\in\PP^+(\fg_\C)$ satisfying $\lambda^{\pi_{\Lambda}} \leq 1$, where $\fg_\C$ is a complex simple Lie algebra. 
This was done in \cite[\S3]{SemmelmannWeingart22} and we reproduce it in Table~\ref{table:CasimirEigenvalues}.  
(See Table~\ref{table:Bourbaki} and Notation~\ref{notation2:Bourbaki} for the ordering of the fundamental weights chosen for each type.)

\begin{table}%[!ht]
	
	\renewcommand{\arraystretch}{1.17}
	
	\caption{Irreducible representations $\pi$ of complex simple Lie algebras with $\lambda^{\pi }\leq1$.} \label{table:CasimirEigenvalues}
	
	\centering
	\begin{tabular}[t]{c c c c c} 
		\hline
		\hline 
		Type  & $\Lambda_\pi$ & $\pi$ & $\lambda^\pi$  & Constraints \\ [0.5ex] 
		\hline\hline
		%$\sll(m,\C)$   
		$A_m$
		& 0                   & $\C$             & 0                                          & $m\geq 2$     \\
		& $\omega_1$          & $\St$            & $\frac{(m-1)(m+1)}{2m^2}$      & $m\geq 2$  \\ 
		& $\omega_{m-1}$      & $\St^*$          & $\frac{(m-1)(m+1)}{2m^2}$      & $m\geq 3$  \\
		& $\omega_2$          & $\bigwedge^2\St$   & $\frac{(m-2)(m+1)}{m^2}$   & $m\geq 4$  \\
		& $\omega_{m-2}$      & $\bigwedge^2 \St^*$& $\frac{(m-2)(m+1)}{m^2}$   & $m\geq 5$  \\
		& $\omega_3$          & $\bigwedge^3 \St$  & $\frac{7}{8},\;\; \frac{48}{49}$ & $m=6,7$    \\
		& $\omega_4$          & $\bigwedge^3 \St^*$& $\frac{48}{49}$                  & $m=7$      \\
		& $\omega_1+\omega_{m-1}$ & $\Ad$ 
		& $1$     & $m\geq 2$   \\
		\hline 
		%$\so(2m+1,\C)$ 
		$B_m$
		& 0          &  $\C$                          & 0
		& $m\geq 1$ \\
		& $\omega_1$ & $\St$                        & $\frac{m}{2m-1}$       & $m\geq 1$ \\
		& $\omega_2$ & $\Ad$ 
		& 1 
		& $m\geq 2$ \\
		& $\omega_m$ & $\Sigma_{2m+1}$                     & $\frac{m(2m+1)}{8(2m-1)}$ 
		& $1\leq m\leq 6$ \\
		\hline 
		
		%$\spp(2m,\C)$ 
		$C_m$
		& 0          &   $\C$                           & 0
		& $m\geq 1$  \\
		& $\omega_1$ &  $\St$                          
		& $\frac{2m+1}{4(m+1)}$ 
		& $m\geq 1$ \\
		& $\omega_2$ & $\bigwedge^2_\circ \, \St$      &  $\frac{m}{m+1}$      & $m\geq 2$ \\
		& $\omega_3$ & $\bigwedge^3_\circ \, \St$      & $\frac{15}{16}$            & $m=3$ \\
		& $2\omega_1$& $\Ad$ 
		& 1                          
		& $m\geq 1$ \\
		\hline 
		
		%$\Spin(2m)$   
		$D_m$
		& 0          &    $\C$                          & 0                                         
		& $m\geq 3$\\
		& $\omega_1$ & $\St$ & $\frac{2m-1}{2(2m-2)}$ &  $m\geq 3$ \\ 
		& $\omega_2$ & $\Ad$ 
		& 1                          & $m\geq 3$  \\
		& $\omega_{m-1}$ &   $\Sigma_{2m}^-$              & $\frac{m(2m-1)}{16(m-1)}$       & $3\leq m\leq 7$\\
		& $\omega_m$     & $\Sigma_{2m}^+$               &  $\frac{m(2m-1)}{16(m-1)}$     &  $3\leq m\leq 7$\\  [1ex]
		\hline
%		%$\op{E}_6$   
%		$E_6$
%		& 0        & $\C$     & 0               \\
%		& $\omega_1$ & $[27]$   & $\frac{13}{18}$ \\ 
%		& $\omega_6$ & $[27]^*$ & $\frac{13}{18}$ \\
%		& $\omega_2$ & $\Ad$    & $1$             \\
%		\hline 
%		%$\op{E}_7$ 
%		$E_7$
%		& 0        & $\C$   & 0\\
%		& $\omega_7$ & $[56]$ & $\frac{19}{24}$ \\
%		& $\omega_1$ & $\Ad$  & 1 \\
%		\hline 
%		%$\op{E}_8$ 
%		$E_8$
%		& 0        & $\C$  & 0\\
%		& $\omega_8$ & $\Ad$ & $1$  \\
%		\hline 
%		
%		%$\op{F}_4$   
%		$F_4$
%		& 0        & $\C$   & 0 \\
%		& $\omega_4$ & $[26]$ & $\frac{2}{3}$ \\ 
%		& $\omega_1$ & $\Ad$  & 1 \\
%		\hline 
%		%$\op{G}_2$   
%		$G_2$
%		& 0          &    $\C$                          & 0 \\
%		& $\omega_1$ & $[7]$ & $\frac{1}{2}$ \\ 
%		& $\omega_2$ & $\Ad$ & 1 \\
%		\hline
	\end{tabular}	
\qquad 
$
\begin{array}[t]{ccccc}
\hline\hline
\text{Type}&\Lambda_\pi&\pi&\lambda^\pi
\\ [0.5ex] 
		\hline\hline
		%$\op{E}_6$   
		E_6
		& 0        & \C     & 0               \\
		& \omega_1 & [27]   & \frac{13}{18} \\ 
		& \omega_6 & [27]^* & \frac{13}{18} \\
		& \omega_2 & \Ad    & 1             \\
		\hline 
		%\op{E}_7 
		E_7
		& 0        & \C   & 0\\
		& \omega_7 & [56] & \frac{19}{24} \\
		& \omega_1 & \Ad  & 1 \\
		\hline 
		%\op{E}_8 
		E_8
		& 0        & \C  & 0\\
		& \omega_8 & \Ad & 1  \\
		\hline 
		
		%\op{F}_4   
		F_4
		& 0        & \C   & 0 \\
		& \omega_4 & [26] & \frac{2}{3} \\ 
		& \omega_1 & \Ad  & 1 \\
		\hline 
		%\op{G}_2   
		G_2
		& 0          &    \C                          & 0 \\
		& \omega_1 & [7] & \frac{1}{2} \\ 
		& \omega_2 & \Ad & 1 \\
		\hline
\end{array}
$	
%{\small 
%	$[n]$ stands for an $n$-dimensional irreducible representation. 
%}
\end{table}

We have for any $\pi\in\widehat G$ that
\begin{equation}
V_\pi^H\subseteq V_\pi^{\fh_\C} 
:= \{v\in V_\pi: (d\pi)_\C(X)\cdot v=0\quad\forall\, X\in\fh_\C\},
\end{equation}
and the equality holds when $H$ is connected.

\begin{proposition}\label{prop:sufficient}
Let $(G/H,g_{\st})$ be a standard homogeneous Riemannian manifold with $G$ and $H$ connected. 
Then $\lambda_1(G/H,g_{\st})\geq 1$ in the following cases:
\begin{enumerate}
\item\label{item3:su-N!=6,7} 
If $\fg\simeq\su(N)$ with $N\neq6,7$, and $V_{\pi_{\Lambda}}^{\fh_\C}=0$ for all $\Lambda\in \{\omega_1,\omega_2\}$. 

\item\label{item3:su-N==6,7} 
If $\fg\simeq\su(6)$ or $\su(7)$, and $V_{\pi_{\Lambda}}^{\fh_\C}=0$ for all $\Lambda\in \{\omega_1,\omega_2,\omega_3\}$. 	

\item\label{item3:SU/Zd-d!=1,2,3} 
If $G\simeq\SU(N)/\Z_d$ with $d\neq 1,2,3,\; d\mid N$. 

\item\label{item3:SU/Zd-d!=1,2} 
If $G\simeq\SU(N)/\Z_d$ with $N>7,\; d\neq 1,2, \; d\mid N$. 

\item\label{item3:SU/Zd-d=2} 
If $G\simeq\SU(N)/\Z_2$ with $N$ even, and $V_{\pi_{\omega_2}}^H=0$. 

\item\label{item3:SU/Zd-d=3,N=6} 
If $G\simeq\SU(6)/\Z_3$, and $V_{\pi_{\omega_3}}^H=0$. 

%%%
% Type C
%%%

\item\label{item3:sp-N!=3}
If $\fg\simeq\spp(N)$ with $N\neq3$, and $V_{\pi_{\Lambda}}^{\fh_\C}=0$ for $\Lambda\in\{\omega_1,\omega_2\}$. 

\item\label{item3:sp-N==3}
If $\fg\simeq\spp(3)$, and $V_{\pi_{\Lambda}}^{\fh_\C}=0$ for $\Lambda\in\{\omega_1,\omega_2,\omega_3\}$. 

\item\label{item3:Sp(N)/Z_2}
If $G\simeq\Sp(N)/\Z_2$, and $V_{\pi_{\omega_2}}^H=0$.

%%%
% Type B and D
%%%

\item\label{item3:so-N>=15} 
If $\fg\simeq\so(N)$ with $N\geq15$, and	$V_{\pi_{\omega_1}}^{\fh_\C}=0$. 

\item\label{item3:so-N<15} 
If $\fg\simeq\so(N)$ with $N\leq14$, and	
$V_{\pi_{\Lambda}}^{\fh_\C}=0$ for $\Lambda\in\{\omega_1,\omega_{\lfloor\frac{N}{2}\rfloor}\}$. 

\item\label{item3:SO} 
If $G\simeq\SO(N)$, and	$V_{\pi_{\omega_1}}^H=0$. 

\item\label{item3:SO/Z2} 
If $G\simeq\SO(N)/\Z_2$ with $N$ even.

%%%
% Excepcional types
%%%

\item \label{item3:E8}
If $\fg\simeq\fe_8$. 

\item \label{item3:E7}
If $\fg\simeq\fe_7$, and $V_{\pi_{\omega_7}}^{\fh_\C}=0$.

\item \label{item3:E7/Z2}
If $G\simeq\op{E}_7/\Z_2$.

\item \label{item3:E6}
If $\fg\simeq\fe_6$, and $V_{\pi_{\omega_1}}^{\fh_\C}=0$.

\item \label{item3:E6/Z3}
If $G\simeq\op{E}_6/\Z_3$.

\item \label{item3:F4}
If $\fg\simeq\ff_4$, and $V_{\pi_{\omega_4}}^{\fh_\C}=0$.

\item \label{item3:G2}
If $\fg\simeq\fg_2$, and $V_{\pi_{\omega_1}}^{\fh_\C}=0$.
\end{enumerate}
\end{proposition}

\begin{proof}
In each item we will check that the hypotheses force the sufficient condition in \eqref{eq:sufficientcondition}, which translates via the Highest Weight Theorem (see Remark~\ref{rem2:HighestWeightTheorem}) to
\begin{equation}\label{eq2:sufficientcondition-weights}
\text{
	$V_{\pi_{\Lambda}}^{H}=0$ for all $\Lambda\in\PP^+(G)\smallsetminus\{0\}$ satisfying $\lambda^{\pi_\Lambda}<1$.
}
\end{equation} 
We will first identify 
\begin{equation}\label{eq2:lambda^pi<1}
\Lambda\in\PP^+(\fg_\C)\smallsetminus\{0\}\text{ satisfying }
\lambda^{\pi_{\Lambda}}<1
\end{equation}
from Table~\ref{table:CasimirEigenvalues}, and then check which of them belong to $\PP(G)$ using the information in Remark~\ref{rem2:integralweightlattice}. 
In the cases when there are $\Lambda_1,\Lambda_2$ satisfying \eqref{eq2:lambda^pi<1} and $\pi_{\Lambda_1}^*\simeq\pi_{\Lambda_2}$, one of them is omitted in the hypotheses because $V_{\pi_{\Lambda_1}}^{\fh_\C}=0$ if and only if $V_{\pi_{\Lambda_2}}^{\fh_\C}=0$ by \eqref{eq2:dual}.

We first consider the cases when $\fg_\C\simeq\su(N)$. 
By Table~\ref{table:CasimirEigenvalues}, 
\eqref{eq2:lambda^pi<1} is only satisfied for
\begin{itemize}
\item $\omega_1,\omega_{N-2}$ for $N\geq2$ ($\pi_{\omega_{N-2}}^*\simeq\pi_{\omega_1}$ if $N>3$, so $\omega_{N-2}$ will be omitted anyway),

\item $\omega_2,\omega_{N-3}$  for $N\geq3$ ($\pi_{\omega_{N-3}}^*\simeq\pi_{\omega_2}$ if $N>5$, so $\omega_{N-3}$ will be omitted anyway),

\item $\omega_3,\omega_{N-4}$  for $N=6,7$ ($\pi_{\omega_{4}}^*\simeq\pi_{\omega_3}$ if $N=7$, so $\omega_{4}$ will be omitted anyway).
\end{itemize}

We next check whether they belong to $\PP(G)$ for each connected Lie group $G$ with Lie algebra $\su(N)$, namely, $\SU(N)/\Z_d$ for $d\mid N$. 
Let $d$ be a positive divisor of $N$. 
It follows immediately from \eqref{eq:A_N[d]} that
\begin{enumerate}
\renewcommand{\labelenumi}{(\roman{enumi})}
\item $\omega_1\in \PP(\SU(N)/\Z_d)$ if and only $d=1$,

\item $\omega_2\in \PP(\SU(N)/\Z_d)$ if and only if  $d=1$, or $d=2$ when $N\equiv 0\pmod 2$,

\item $\omega_3\in \PP(\SU(N)/\Z_d)$ if and only if  $d=1$, or $d=3$ when $N\equiv 0\pmod 3$.
\end{enumerate}
We thus obtain the following conclusions:

\smallskip
\noindent 
\eqref{item3:su-N!=6,7} 
$\fg\simeq\su(N)$ with $N\neq6,7$: 
$\lambda^{\pi_{\omega_3}}>1$ since $N\neq 6,7$, thus \eqref{eq2:sufficientcondition-weights} is equivalent to $V_{\pi_{\Lambda}}^{\fh_\C}=0$ for all $\Lambda\in \{\omega_1,\omega_2\}$.

\smallskip
\noindent 
\eqref{item3:su-N==6,7} 
$\fg\simeq\su(6)$ or $\su(7)$: 
\eqref{eq2:sufficientcondition-weights} is equivalent to $V_{\pi_{\Lambda}}^{\fh_\C}=0$ for all $\Lambda\in \{\omega_1,\omega_2,\omega_3\}$.

\smallskip
\noindent 
\eqref{item3:SU/Zd-d!=1,2,3} 
$G\simeq\SU(N)/\Z_d$ with $d\neq 1,2,3,\; d\mid N$:
It follows from (i)--(iii) that $\omega_1,\omega_2,\omega_3 \notin\PP(G)$, thus \eqref{eq2:sufficientcondition-weights} always holds. 

\smallskip
\noindent 
\eqref{item3:SU/Zd-d!=1,2} 
$G\simeq\SU(N)/\Z_d$ with $N>7,\; d\neq 1,2, \; d\mid N$: 
$\lambda^{\pi_{\omega_3}}>1$  because $N\neq 6,7$.
Furthermore, $\omega_1$ and $\omega_2$ are not in $\PP^+(G)$ by (i) and (ii) respectively. 
Thus \eqref{eq2:sufficientcondition-weights} always holds.

\smallskip
\noindent 
\eqref{item3:SU/Zd-d=2} 
$G\simeq\SU(N)/\Z_2$ with $N$ even:
$\omega_1$ and $\omega_3$ are not in $\PP^+(G)$ by (i) and 
(iii) respectively. 
Thus, \eqref{eq2:sufficientcondition-weights} is equivalent to $V_{\pi_{\omega_2}}^H=0$.

\smallskip
\noindent 
\eqref{item3:SU/Zd-d=3,N=6} 
$G\simeq\SU(6)/\Gamma_3$:
$\omega_1$ and $\omega_2$ are not in $\PP^+(G)$ by (i) and 
(ii) respectively. 
Thus, \eqref{eq2:sufficientcondition-weights} is equivalent to $V_{\pi_{\omega_3}}^H=0$.

\medskip

Suppose $\fg\simeq\spp(N)$. 
\eqref{eq2:lambda^pi<1} holds for $\omega_1=\ee_1$, $\omega_2=\ee_1+\ee_2$, and also $\omega_3=\ee_1+\ee_2+\ee_3$ when $N=3$.

\smallskip
\noindent
\eqref{item3:sp-N!=3}
$\fg\simeq\spp(N)$ with $N\neq3$:
\eqref{eq2:sufficientcondition-weights} is equivalent to $V_{\pi_{\Lambda}}^{\fh_\C}=0$ for $\Lambda\in\{\omega_1,\omega_2\}$ because $N\neq3$. 

\smallskip
\noindent
\eqref{item3:sp-N==3}
$\fg\simeq\spp(3)$:
\eqref{eq2:sufficientcondition-weights} is equivalent to $V_{\pi_{\Lambda}}^{\fh_\C}=0$ for $\Lambda\in\{\omega_1,\omega_2,\omega_3\}$.

\smallskip
\noindent
\eqref{item3:Sp(N)/Z_2} 
$G\simeq\Sp(N)/\Z_2$: 
$\PP(\Sp(n)/\Z_2)=\DD_n$ (see Remark~\ref{rem2:integralweightlattice}), so $\omega_1,\omega_3$ are not there because the sum of their coordinates are not even. 
Consequently, \eqref{eq2:sufficientcondition-weights} is equivalent to $V_{\pi_{\omega_2}}^H=0$.

\medskip

Suppose $\fg\simeq\so(N)$. 
\eqref{eq2:lambda^pi<1} holds for $\omega_1=\ee_1$, and also for $\omega_m=\frac12(\ee_1+\dots+\ee_m)$ when $N=2m+1\leq 13$ and $\omega_{m-1}=\frac12(\ee_1+\dots+\ee_{m-1}-\ee_m),\omega_{m} =\frac12(\ee_1+\dots+\ee_m)$ when $N=2m\leq 14$. 
The irreducible representation of $\fg_\C$ associated with $\omega_m$ is a spin representation, which does not descend to a representation of $\SO(N)$. 
Moreover, $\pi_{\omega_{m-1}}^*\simeq \pi_{\omega_{m}}$ when $N=2m$, thus $\omega_{m-1}$ will be omitted.

\noindent
\eqref{item3:so-N>=15} 
$\fg\simeq\so(N)$ with $N\geq15$:
\eqref{eq2:lambda^pi<1} holds only for $\omega_1=\ee_1$ because $N\geq15$, thus \eqref{eq2:sufficientcondition-weights} reduces to $V_{\pi_{\omega_1}}^{\fh_\C}=0$. 

\noindent
\eqref{item3:so-N<15} 
$\fg\simeq\so(N)$ with $N\leq14$:
\eqref{eq2:sufficientcondition-weights} is equivalent to $V_{\pi_{\omega_1}}^{\fh_\C} = V_{\pi_{\omega_m}}^{\fh_\C}=0$, where $N=2m$ or either $N=2m+1$.

\noindent
\eqref{item3:SO}
$G\simeq\SO(N)$: 
since $\omega_m\notin \PP(\SO(N))=\Z^m$, \eqref{eq2:sufficientcondition-weights} is equivalent to $V_{\pi_{\omega_1}}^H=0$. 

\smallskip
\noindent
\eqref{item3:SO/Z2} 
$G\simeq\SO(2m)/\Z_2$: 
since $\omega_1,\omega_m\notin \PP(\SO(2m)/\Z_2)=\DD_m$, \eqref{eq2:sufficientcondition-weights} always holds. 

\medskip

We finish with the exceptional types. 

\smallskip
\noindent
\eqref{item3:E8}
$\fg\simeq\fe_8$:
\eqref{eq2:lambda^pi<1} is empty, so \eqref{eq2:sufficientcondition-weights} always holds.

\smallskip
\noindent
\eqref{item3:E7}
$\fg\simeq\fe_7$:
\eqref{eq2:lambda^pi<1} holds only for $\omega_7=\ee_6+\frac12(\ee_8-\ee_7)$, thus 
\eqref{eq2:sufficientcondition-weights} is equivalent to $V_{\pi_{\omega_7}}^{\fh_\C}=0$.

\smallskip
\noindent
\eqref{item3:E7/Z2} 
$G\simeq\op{E}_7/\Z_2$: 
since $\omega_7$ is not in $\EE_8\supset \PP(\op{E}_7/\Z_2)=\EE_7$ (see Remark~\ref{rem2:integralweightlattice}), \eqref{eq2:sufficientcondition-weights} always holds.

\smallskip
\noindent
\eqref{item3:E6}
$\fg\simeq\fe_6$:
\eqref{eq2:lambda^pi<1} is only satisfied for $\omega_1=\frac23(\ee_8-\ee_7-\ee_6)$ and $\omega_6=\frac23(\ee_8-\ee_7-\ee_6)+\ee_5$.
However, $\pi_{\omega_6}^*\simeq\pi_{\omega_1}$. 
Thus 
\eqref{eq2:sufficientcondition-weights} is equivalent to $V_{\pi_{\omega_1}}^{\fh_\C}=0$.

\smallskip
\noindent
\eqref{item3:E6/Z3} 
$G\simeq\op{E}_6/\Z_3$: 
$\omega_1$ and $\omega_6$ are clearly not in $\EE_8\supset \PP(\op{E}_6/\Z_3)=\EE_6$, thus \eqref{eq2:sufficientcondition-weights} always holds.

\smallskip
\noindent
\eqref{item3:F4}
$\fg\simeq\ff_4$: 
\eqref{eq2:lambda^pi<1} holds only for $\omega_4$, thus 
\eqref{eq2:sufficientcondition-weights} is equivalent to $V_{\pi_{\omega_4}}^{\fh_\C}=0$.

\smallskip
\noindent
\eqref{item3:G2}
$\fg\simeq\fg_2$:
\eqref{eq2:lambda^pi<1} holds only for $\omega_1$, thus 
\eqref{eq2:sufficientcondition-weights} is equivalent to $V_{\pi_{\omega_1}}^{\fh_\C}=0$.
\end{proof}

We now end the section by giving conditions which ensure $V_\pi^H=0$ under certain hypotheses. 

\begin{lemma}\label{lem3:dimV_pi^H}
Let $H$ be a closed subgroup of a compact Lie group $G$, and let $\pi$ be a finite-dimensional representation of $G$.  
Then $V_\pi^H=0$ if the trivial representation $1_H$ of $H$ does not occur in the decomposition in irreducible $H$-representations of the restriction $\pi|_H$ of $\pi$ to $H$. 
\end{lemma}

\begin{proof}
This is obvious since $\dim V_\pi^H$ is precisely the number of times that $1_H$ occurs in $\pi|_H$. 
\end{proof}

\begin{notation}\label{notation:G_rho}
Let $H'$ be a compact and connected Lie group. 
For an $N$-dimensional complex representation $\rho$ of $H'$, we denote by $G_\rho$ the subgroup of $\GL(V_\rho)$ satisfying that $\rho(H')\subset G_\rho$ with $G_\rho\simeq\SU(N),\SO(N),\Sp(N/2)$ according to $\rho$ is of complex, real or quaternionic type as in Remark~\ref{rem2:types-of-representations}. 
In any of these cases, $(\fg_\rho)_\C$ is a classical complex Lie algebra and the highest weight of its standard representation $\st_{(\fg_\rho)_\C}$ is always the first fundamental weight $\omega_1$, i.e.\ $\st_{(\fg_\rho)_\C}=\pi_{\omega_1}$. 
\end{notation}

\begin{lemma}\label{lem3:standard}
Let $H'$ be a compact and connected Lie group, let $\rho$ be a complex finite-dimensional representation ${H}'$, and let $G_\rho$ be as above. 
Then 
$$
\dim V_{\pi_{\omega_1}}^{(d\rho)_\C(\fh_\C')}
= \dim V_{(d\rho)_\C}^{\fh_\C'}
,
$$
which coincides with the number of times that the trivial representation $1_{\fh_\C'}$ occurs in the decomposition of $(d\rho)_\C$ as irreducible representations. 

In particular, if $G$ is any connected Lie group with Lie algebra $\fg_\rho$ such that $\omega_1\in\PP(G)$,  $\rho_0:H'\to G$ is the only Lie group homomorphism satisfying $d\rho_0=d\rho$, and $H:=\rho_0(H')\simeq H'/\op{Ker}(\rho_0)$, then 
$$
V_{\pi_{\omega_1}}^{H}=V_{\rho}^{H'}
.
$$
\end{lemma}

\begin{proof}
The underlying vector spaces of the representations $\rho$ and $\pi_{\omega_1}\circ\rho$ coincide.
Hence
$$
V_{\pi_{\omega_1}}^{(d\rho)_\C(\fh_\C')}
= V_{\pi_{\omega_1}\circ(d\rho)_\C }^{\fh_\C'}
.
$$
Since $\pi_{\omega_1}$ is the standard representation, $\pi_{\omega_1}\circ(d\rho)_\C$ is essentially $(d\rho)_\C$, and the first assertion follows. 
Since $H'$ is connected, so is $H$.
Then  $V_{\pi_{\omega_1}}^{H}=V_{\pi_{\omega_1}}^{\fh_\C}$ and  $V_{\rho}^{H'}=\dim V_{(d\rho)_\C}^{\fh_\C'}$, and the second assertion follows too.  
\end{proof}

\section{Strongly isotropy irreducible spaces} \label{sec:isotirred}

A homogeneous space $G/H$ is called \emph{isotropy irreducible} if the isotropy representation on the tangent space $T_{eH}G/H$ of the isotropy group $H$ is irreducible. 
Similarly, $G/H$ is \emph{strongly isotropy irreducible} if the identity connected component $H^o$ of $H$ acts irreducibly on $T_{eH}G/H$. 
Of course, both notions coincide when $H$ is connected. 

Strongly isotropy irreducible spaces were classified independently by Manturov~\cite{Manturov61a,Manturov61b,Manturov66}, Wolf~\cite{Wolf68}, and Krämer~\cite{Kramer}. 
The remaining isotropy irreducible spaces were classified by Wang and Ziller~\cite{WangZiller91}.

The aim of this section is to estimate the smallest positive eigenvalue $\lambda_1(G/H,g_{\st})$ of the Laplace-Beltrami operator associated to the standard metric $g_{\st}$ on a compact strongly isotropy irreducible space $G/H$. 
There are relatively few articles in the literature that study the Laplace–Beltrami operator associated with these spaces; one recent example is \cite{ProvenzanoStubbe25}. 

We will omit the compact irreducible symmetric spaces since for those cases $\lambda_1(G/H,g_{\st})$ was explicitly computed by Urakawa~\cite{Urakawa86}. 
Every non-symmetric strongly isotropy irreducible space $M$ is compact, therefore they are compact Lie groups $H\subset G$ such that $M=G/H$.

The simply connected non-symmetric strongly isotropy irreducible spaces consist of 10 infinite families described in Table~\ref{table4:isotropyirred-families} and 33 isolated cases described in Tables~\ref{table4:isotropyirredexcepcion-classical}--\ref{table4:isotropyirredexcepcion-exceptional}.
For each entry $G/H$ in Tables~\ref{table4:isotropyirred-families}--\ref{table4:isotropyirredexcepcion-exceptional}, 
\begin{itemize}
\item $H\subset G$ are compact Lie groups,
\item $G/H$ is simply connected, 
\item $G$ is connected and acts effectively on $G/H$, 
\item the linear isotropy action of $H$ on the tangent space of $G/H$ is irreducible. 
\end{itemize}
(Note that $H$ is automatically connected because $G$ is connected and $G/H$ is simply connected, thus $G/H$ is isotropy irreducible.)
This information was extracted from the four tables in pages 107--110 in 
\cite{Wolf68}, and from \cite{Wolf84-errata-isot-irred}. 
An arbitrary (not necessarily simply connected) non-symmetric strongly isotropy irreducible space is covered by one element in the tables. 

We next explain the embedding of $H$ into $G$ for each case using Notation~\ref{notation:G_rho}.

\begin{remark}[Embedding of $H$ into $G$]
\label{rem4:embedding}
Let $G/H$ be an entry in  Tables~\ref{table4:isotropyirred-families}--\ref{table4:isotropyirredexcepcion-exceptional}.
Let $H'$ denote the connected and simply connected Lie group with Lie algebra $\fh'=\fh$. 
The Lie algebra representation $(d\rho)_\C$ of $\fh'$ corresponding to $G/H$ (indicated in the tables) defines a classical Lie algebra $\fg_\rho\simeq \su(N),\so(N),\spp(N/2)$ according to whether $(d\rho)_\C$ is of complex, real or quaternionic type, where $N$ is the dimension of $(d\rho)_\C$ (see Remark~\ref{rem2:types-of-representations}). 

If $G$ is classical (i.e.\ $G/H$ is in  Table~\ref{table4:isotropyirred-families} or \ref{table4:isotropyirredexcepcion-classical}), one can check that $\fg=\fg_\rho$.  
If $\rho:H'\to G$ is the only Lie group homomorphism with $(d\rho)_\C$ indicated in the table, then $H=\rho(H')\simeq H'/\op{Ker}(\rho)$. 
Note that $G$ and $G_\rho$ share the same universal cover, but they do not necessarily coincide. 

If $G$ is exceptional (i.e.\ $G/H$ is in Table~\ref{table4:isotropyirredexcepcion-exceptional}), $(d\rho)_\C$ is not necessarily irreducible, but it decomposes in irreducible representations of the same type.
If $(d\rho)_\C=(d\rho_1)_\C\oplus\dots\oplus (d\rho_l)_\C$, we set $\fg_\rho=\fg_{\rho_1}\oplus\dots\oplus \fg_{\rho_l}$. 
There is a connected Lie group $G_\rho'$ with Lie algebra $\fg_\rho'=\fg_\rho$ containing $G$, so $\fg\subset \fg_\rho$. 
The only Lie group homomorphism $\rho:H'\to G_\rho'$ with $(d\rho)_\C$ indicated in the tables satisfies $\rho(H')\subset G$. 
In this case, $H=\rho(H')\simeq H'/\op{Ker}(\rho)$. 

Additionally, still when $G$ is exceptional, one has that $\fh$ is a maximal subalgebra of $\fg$, and the inclusion $\fh\hookrightarrow\fg$ determines the inclusion $H\hookrightarrow G$ at the Lie group level. 
\end{remark} 

The goal is to prove 
Theorem~\ref{thm4:main}:  $\lambda_1(G'/H',g_{\st})\geq 1>2E$ for all non-symmetric strongly isotropy irreducible spaces $G'/H'$ excepting for $\op{G}_2/\SU(3)$ and $\Spin(7)/\op{G}_2$.

We recall that $\lambda_1(M,g)\leq \lambda_1(M',g')$ for every Riemannian covering $(M,g)\to (M',g')$. 
Indeed, eigenfunctions of the Laplacian of $(M',g')$ are pulled back to eigenfunctions of the Laplacian of $(M,g)$.  
Consequently, if $\lambda_1(G/H,g_{\st})\geq 1$ for a simply connected choice $G/H$ as above, then $\lambda_1(G'/H',g_{\st})\geq1$ for every space $G'/H'$ covered by $G/H$. 
Therefore, it will be sufficient to show that 
\begin{itemize}
\item $\lambda_1(G/H,g_{\st})\geq 1$ for each entry in the tables different to $\op{G}_2/\SU(3)$ and $\Spin(7)/\op{G}_2$, which is done in Subsections~\ref{subsec4:Sp}--\ref{subsec4:exceptional} divided according to the type of $\fg_\C$, and

\item $\lambda_1(G'/H',g_{\st})\geq 1$ for any strongly isotropy irreducible space $G'/H'$ properly covered by $\op{G}_2/\SU(3)$ or $\Spin(7)/\op{G}_2$, which is done in Subsection~\ref{subsec4:roundspheres}.
\end{itemize}

Additionally, Subsection~\ref{subsec4:nu-stability} describes the consequences of Theorem~\ref{thm4:main} on the $\nu$-stability type, and Subsection~\ref{subsec4:computational} gives some numerical computations of $\lambda_1$ of independent interest. 

\begin{remark}\label{rem4:Einsteinfactor}
For each non-symmetric simply connected strongly isotropy irreducible space, the value of the Einstein factor $E$ in  Tables~\ref{table4:isotropyirred-families}--\ref{table4:isotropyirredexcepcion-exceptional} was computed by using the expression (see \cite[p.\ 568]{WangZiller85} or \cite[(7.94)]{Besse})
\begin{equation*}
E=\frac14 + \frac12 \sum_{i=0}^{r}\frac{\dim\fh_i(1-c_i)}{\dim(G/H)}
,
\end{equation*}
where $\fh=\fh_0\oplus\fh_1\oplus \dots\oplus\fh_r$ is the decomposition in irreducible ideals with $\fh_0=\fz(\fh)$ and $\kil_{\fh_i}=c_i\kil_{\fg}|_{\fh_i}$ for $i=0,\dots,r$. 
The factors $c_i$ were determined using the general procedure described in \cite[p.\ 38--40]{DAtriZiller}. 
These values correct those in \cite{Schwahn-Lichnerowicz} for Families IV and VI. 
\end{remark}

%\begin{comment} 

\begin{table}
\caption{The 10 families of simply connected strongly isotropy irreducible spaces.}\label{table4:isotropyirred-families}

{\renewcommand{\arraystretch}{1.5}\small
 \begin{tabular}{ccccccc} 
 No. & $\dfrac{G}{H}$ & $(d\rho)_\C$ &Cond. 
 \\ [1ex]
 \hline \hline \\ [-4ex]
I & 
	$\dfrac{\SU\big(\frac{n(n-1)}{2}\big)/\Z_d}{\SU(n)/\Z_n}$ &
	$\sigma_{\eta_2}$ &
	{\tiny 
	$n\geq 5$,
	\begin{tabular}{c}
		$d=\frac{n}{2}$, $n$ even\\[-3pt]
		$d=n$, $n$ odd
	\end{tabular}
	}& 
\\ 
 \hline \\ [-4ex]
 
 II & 
	$\dfrac{\SU\big(\frac{n(n+1)}{2}\big)/\Z_d}{\SU(n)/\Z_n}$ & 
	$\sigma_{2\eta_1}$  &
	{\tiny 
	$n\geq 3$,
	\begin{tabular}{c}
		$d=\frac{n}{2}$, $n$ even\\[-3pt]
		$d=n$, $n$ odd
	\end{tabular}
	}&
 \\ \hline \\ [-4ex]

III	& 
	$\dfrac{\SU(pq)/\Z_d}{\big(\SU(p)/\Z_p\big)\times \big(\SU(q)/\Z_q\big)}$ &
	$\sigma_{\eta_1}\widehat\otimes \,\, \sigma_{\eta_1'}'$ & 
	{\tiny $\begin{array}{c} 
		2\leq p\leq q, \;
		p+q\neq 4 \\[-3pt] 
		d=\op{lcm}(p,q)
	\end{array}$ } &  
\\ \hline \\ [-3.5ex]
IV & 
	$\begin{array}{c}
		\dfrac{\Sp(n)/\Z_2} {(\Sp(1)/\Z_2)\times (\SO(n)/\Z_2)}, \text{ $n$ even} \\[9pt]
		\dfrac{\Sp(n)/\Z_2}{(\Sp(1)/\Z_2)\times \SO(n)}, \text{ $n$ odd} 
	\end{array}$
	& 
	$\sigma_{\eta_1}\widehat\otimes \,\, \sigma_{\eta_1'}'$ &
	$n\geq 3$ & 
\\ \hline \\ [-3ex]

V & 
	$\begin{array}{c}
		\dfrac{\Spin(n^2-1)}{\SU(n)/\Z_n}, \text{ $n$ odd} \\[9pt]
		\dfrac{\SO(n^2-1)}{\SU(n)/\Z_n}, \text{ $n$ even} 
	\end{array}$ & 
	$\sigma_{\eta_1+\eta_{n-1}}$ & 
	$n\geq 3$ & 
\\ \hline \\ [-3ex]
 
VI & 
	$\begin{array}{c}
		\dfrac{\Spin(2n^2-n-1)}{\Sp(n)/\Z_2}, \text{ $n\equiv 0,1\pmod 4$} \\[9pt]
		\dfrac{\SO(2n^2-n-1)}{\Sp(n)/\Z_2}, \text{ $n\equiv 2,3\pmod 4$} 
	\end{array}$ &
	$\sigma_{\eta_2}$  & 
	$n\geq 3$ & 
\\ \hline \\ [-3ex]
 
VII  & 
	$\begin{array}{c}
		\dfrac{\Spin(2n^2+n)}{\Sp(n)/\Z_2}, \text{ $n\equiv 0,3\pmod 4$} \\[9pt]
		\dfrac{\SO(2n^2+n)}{\Sp(n)/\Z_2}, \text{ $n\equiv 1,2\pmod 4$} 
	\end{array}$ &
	$\sigma_{2\eta_1}$ &
	$n\geq 3$ & 

 \\
 \hline \\ [-3ex]
 
VIII & 
	$\dfrac{\SO(4n)/\Z_2}{(\Sp(1)/\Z_2)\times(\Sp(n)/\Z_2)}$ & 
	$\sigma_{\eta_1}\widehat\otimes \,\, \sigma_{\eta_1'}'$ & 
	$n\geq 2$ & 
\\
 \hline \\ [-3ex]
 
IX   & 
	$\begin{array}{c}
		\dfrac{\SO\big(\frac{n(n-1)}{2}\big)}{\SO(n)}, \text{ $n\equiv 1\pmod 2$} \\[9pt]
		\dfrac{\Spin\big(\frac{n(n-1)}{2}\big)}{\SO(n)/\Z_2}, \text{ $n\equiv 0\pmod 4$} \\[9pt]
		\dfrac{\SO\big(\frac{n(n-1)}{2}\big)}{\SO(n)/\Z_2}, \text{ $n\equiv 2\pmod 4$} 
	\end{array}$ &
	$\begin{array}{cc} 
		\sigma_{2\eta_2} \\
		\sigma_{\eta_2}
		\end{array}$  & 
	$\begin{array}{cc} n=5, \\ n\geq 7 \end{array}$ 
 \\
 \hline \\ [-3ex]
X    & 
	$\begin{array}{c}
		\dfrac{\SO\big(\frac{(n-1)(n+2)}{2}\big)}{\SO(n)}, \text{ $n\equiv 1\pmod 2$} \\[9pt]
		\dfrac{\Spin\big(\frac{(n-1)(n+2)}{2}\big)}{\SO(n)/\Z_2}, \text{ $n\equiv 0\pmod 4$} \\[9pt]
		\dfrac{\SO\big(\frac{(n-1)(n+2)}{2}\big)}{\SO(n)/\Z_2}, \text{ $n\equiv 2\pmod 4$} \\[9pt]
	\end{array}$ &
 	$\sigma_{2\eta_1}$ & 
 	$n\geq 5$ & 
\\ %[1ex] 
\hline
\end{tabular}

\smallskip

{\scriptsize $\SO(20)/(\SU(4)/\Z_4)$} is not listed as in \cite{Wolf68} for being a member of Family X ($n=6$) (see \cite{Schwahn-Lichnerowicz}).
}
\end{table}

{
\renewcommand{\arraystretch}{1.5}
\begin{table}[!ht]
\caption{Isolated simply connected strongly isotropy irreducible spaces with $G$ classical.}\label{table4:isotropyirredexcepcion-classical}

\begin{tabular}[t]{c c c c c c c c} 
\hline 
\hline
No & $G$
 & $H$
 & $(d\rho)_\C$ & $2E$& 
\begin{tabular}{c}
$\nu$-stability \\[-11pt] type
\end{tabular}
& 
 \begin{tabular}{c}
 \tiny computational \\[-12pt] 
 \tiny calculation of \\[-12pt]
 \tiny $\lambda_1(G/H,g_{\st})$ 
 \end{tabular}
\\ [0.5ex]
\hline 
\hline 
 1&$\SU(16)/\Z_4$ & $\SO(10)/\Z_2$  & $\sigma_{\eta_4}$  &$\frac{11}{16}$  &unknown 
 &$\lambda^{\pi_{2\omega_2}}=\frac{63}{32}$
 \\  \hline 
 2&$\SU(27)/\Z_3$ & $\op{E}_6/\Z_3$    & $\sigma_{\eta_1}$ & $\frac{11}{18}$ &$\nu$-semistable 
 &$\lambda^{\pi_{3\omega_1}}=\frac{130}{81}$
 \\ \hline
 3&$\Spin(7)$  & $\op{G}_2$    & $\sigma_{\eta_1}$ & $\frac{9}{10}$  &$\nu$-stable*
 &$\lambda^{\pi_{\omega_3}}=\frac{21}{40}$
  \\ \hline
 4& $\SO(133)$ & $\op{E}_7/\Z_2$   & $\sigma_{\eta_3}$ & $\frac{135}{262}$ &$\nu$-stable
 &not computed
 \\ \hline
 5&$\Sp(2)/\Z_2$  & $\SO(3)$  & $\sigma_{3\eta_1}$ & $\frac{9}{10}$  &unknown 
 &$\lambda^{\pi_{4\omega_1}}=\frac{8}{3}$
 \\ \hline
 6&$\Sp(7)/\Z_2$  & $\Sp(3)/\Z_2$ & $\sigma_{\eta_3}$  & $\frac{29}{40}$  &unknown
 &$\lambda^{\pi_{4\omega_1}}=\frac{9}{4}$
 \\ \hline
 7&$\Sp(10)/\Z_2$ & $\SU(6)/\Z_6$  & $\sigma_{\eta_3}$  & $\frac{15}{22}$  &unknown 
 &$\lambda^{\pi_{4\omega_1}}=\frac{24}{11}$
 \\ \hline
 8&$\Sp(16)/\Z_2$ & $\SO(12)/\Z_2$ & $\sigma_{\eta_5}$ & $\frac{43}{68}$   &unknown
 &$\lambda^{\pi_{4\omega_1}}=\frac{36}{17}$
 \\ \hline
 9&$\Sp(28)/\Z_2$ & $\op{E}_7/\Z_2$  &  $\sigma_{\eta_7}$ & $\frac{17}{29}$  &$\nu$-stable 
 &$\lambda^{\pi_{4\omega_1}}=\frac{60}{29}$
 \\ \hline
10&$\Spin(14)$ & $\op{G}_2$  &  $\sigma_{\eta_2}$ & $\frac{2}{3}$    &unknown 
&$\lambda^{\pi_{\omega_3}}=\frac{11}{8}$
\\ \hline
11& $\Spin(16)/\Z_2 \;\dag$
& $\SO(9)$ &  $\sigma_{\eta_4}$ & $\frac{23}{28}$  &unknown 
&$\lambda^{\pi_{2\omega_2}}=\frac{15}{7}$
\\ \hline
12&$\Spin(26)$ & $\op{F}_4$  &  $\sigma_{\eta_4}$ & $\frac{2}{3}$    &unknown 
&$\lambda^{\pi_{3\omega_1}}=\frac{27}{16}$
\\ \hline
13&$\Spin(42)$ & $\Sp(4)/\Z_2$& $\sigma_{\eta_4}$  & $\frac{19}{35}$   &$\nu$-stable
&$\lambda^{\pi_{2\omega_2}}=\frac{41}{20}$
\\ \hline
14&$\Spin(52)$ & $\op{F}_4$&  $\sigma_{\eta_1}$  & $\frac{27}{50}$  &$\nu$-stable 
&not computed
\\ \hline
15&$\SO(70)/\Z_2$ & $\SU(8)/\Z_8$ & $\sigma_{\eta_4}$  & $\frac{179}{340}$ &$\nu$-stable
&not computed
\\ \hline
16& $\Spin(78)$ & $\op{E}_6/\Z_3$ & $\sigma_{\eta_2}$   & $\frac{10}{19}$  &$\nu$-stable 
&not computed
\\ \hline
17&$\Spin(128)/\Z_2 \; \dag $
& $\SO(16)/\Z_2$ & $\sigma_{\eta_7}$ & $\frac{173}{336}$ &$\nu$-stable
&not computed
\\ \hline
18&$\Spin(248)$ & $\op{E}_8$ &  $\sigma_{\eta_8}$  & $\frac{125}{246}$ &$\nu$-stable
&not computed
\\  \hline 
\end{tabular}

\

* $(\Spin(7)/\op{G}_2,g_{\st})\simeq (S^7,g_{\text{round}})$ is $\nu$-stable by Remark~\ref{rem:spheres}.

$\dag$ The quotient $\Spin(N)/\Z_2$ is not isomorphic to $\SO(N)$. 
\end{table}

\begin{table} 
\caption{Isolated simply connected strongly isotropy irreducible spaces with $G$ exceptional} \label{table4:isotropyirredexcepcion-exceptional}

\begin{tabular}[t]{cc c c cc c} 
 \hline 
 \hline
No & $G$ & $H$ &$(d\rho)_\C$ &$2E$ & 
\begin{tabular}{c}
$\nu$-stability \\[-11pt] type
\end{tabular}
& 
 \begin{tabular}{c}
 \tiny computational \\[-12pt] 
 \tiny calculation of \\[-12pt]
 \tiny $\lambda_1(G/H,g_{\st})$ 
 \end{tabular}
\\ 
 \hline 
 \hline 
19&
	$\op{E}_6$ &  
	$\SU(3)/\Z_3$             & 
	$\sigma_{2(\eta_1+\eta_2)}$ &
	$\frac{11}{18}$ &
	$\nu$-stable&
	$\lambda^{\pi_{2\omega_1}}=\frac{14}{9}$ 
\\  \hline 
20&
	$\op{E}_6/\Z_3$ & 
	$\SU(3)^3/(\Z_3\times\Z_3)$  & 
	$\sigma_{\eta_1}\widehat{\otimes} \,\sigma_{\eta_1'}'\widehat{\otimes} \,\sigma_0''$&
	$\frac{5}{6}$ &
	unknown   &
	$\lambda^{\pi_{\omega_1+\omega_6}}=\frac{3}{2}$
\\ \hline 
21&
	$\op{E}_6$ & 
	$\op{G}_2$               & 
	$\sigma_{2\eta_1}$&
	$\frac{25}{36}$  &
	$\nu$-stable &
	$\lambda^{\pi_{2\omega_1}}=\frac{14}{9}$
\\ \hline  
22&
	$\op{E}_6/\Z_3$ & 
	$(\SU(3)/\Z_3) \times \op{G}_2$   & 
	\scriptsize
	$
	\sigma_{\eta_1}\widehat{\otimes} \,\sigma_{\eta_1'}'\oplus
	\sigma_{2\eta_2}\widehat{\otimes} \,\sigma_0'
	$&
	$\frac{19}{24}$ &
	unknown   &
	$\lambda^{\pi_{\omega_1+\omega_6}}=\frac{3}{2}$
\\ \hline 
23&
	$\op{E}_7$ & 
	$\SU(3)/\Z_3$ & 
	$\sigma_{6\eta_1}\oplus \,\sigma_{6\eta_2}$&
	$\frac{71}{126}$  &
	$\nu$-stable &
	$\lambda^{\pi_{2\omega_7}}=\frac{5}{3}$
\\ \hline 
24&
	$\op{E}_7/\Z_2$ & 
	$(\SU(6)\times\SU(3))/\Z_6$& 
	$\sigma_{\eta_1}\widehat{\otimes} \,\sigma_{\eta_1'}'$&
	$\frac{5}{6}$ &
	unknown     &
	$\lambda^{\pi_{\omega_6}}=\frac{14}{9}$
\\  \hline 
25&
	$\op{E}_7/\Z_2$ & 
	$\op{G}_2\times (\Sp(3)/\Z_2)$  & 
	$\sigma_{\eta_1}\widehat{\otimes} \,\sigma_{\eta_1'}'$&
	$\frac{7}{9}$  &
	$\nu$-stable &
	$\lambda^{\pi_{\omega_6}}=\frac{14}{9}$
\\ \hline 
26&
	$\op{E}_7/\Z_2$ & 
	$\SO(3)\times\op{F}_4$   & 
	$\sigma_{\eta_1}\widehat{\otimes} \,\sigma_{\eta_1'}'$&
	$\frac{47}{54}$  &
	unknown  &
	$\lambda^{\pi_{\omega_6}}=\frac{14}{9}$
\\ \hline 
27&
	$\op{E}_8$ & 
	$\SU(9)/\Z_3$      & 
	$\sigma_{\eta_3}\oplus \sigma_{\eta_6}\oplus \sigma_{\eta_1+\eta_8}$&
	$\frac{5}{6}$ &
	$\nu$-stable &
	$\lambda^{\pi_{\omega_7}}=2$
\\ \hline 
28&
	$\op{E}_8$ & 
	$(\op{E}_6\times\SU(3))/\Z_3$  & 
	\tiny
	$
	\begin{array}{l}
	\sigma_{\eta_1}\widehat{\otimes} \,\sigma_{\eta_1'}' \oplus%\\[-6pt]
	\sigma_{\eta_6}\widehat{\otimes}\,\sigma_{\eta_2'}' \oplus\\[-6pt]
	\sigma_{0}\widehat{\otimes} \,\sigma_{\eta_1'+\eta_2'}' \oplus%\\[-6pt]
	\sigma_{\eta_2}\widehat{\otimes} \,\sigma_{0}' 
	\end{array}
	$&
	$\frac{5}{6}$ &
	$\nu$-stable &
	$\lambda^{\pi_{\omega_1}}=\frac{8}{5}$
\\ \hline 
29&
	$\op{E}_8$ & 
	$\op{G}_2\times\op{F}_4$    & 
	$\begin{array}{c}
	\sigma_{\eta_1}\widehat{\otimes} \sigma'_{\eta'_1}\oplus \sigma_{0}\widehat{\otimes} \sigma'_{\eta'_1} \oplus \\ [-6pt]
	\sigma_{\eta_2}\widehat{\otimes} \sigma'_0 
	\end{array}$&
	$\frac{23}{30}$  &
	$\nu$-stable &
	$\lambda^{\pi_{\omega_1}}=\frac{8}{5}$
\\ \hline 
30&
	$\op{F}_4$ & 
	$(\SU(3)\times\SU(3))/\Z_3$ & 
	$
	\begin{array}{c} 
	\sigma_{\eta_1}\widehat{\otimes} \sigma'_{\eta'_1}\oplus \sigma_{\eta_2}\widehat{\otimes} \sigma'_{\eta'_2}
	\oplus\\   \sigma_0 \widehat{\otimes} \sigma'_{\eta'_1+\eta'_2}
	\end{array}$&
	$\frac{5}{6}$  &
	unknown  &
	$\lambda^{\pi_{\omega_3}}=\frac{4}{3}$
\\ \hline 
31&
	$\op{F}_4$ & 
	$\SO(3)\times\op{G}_2$ & 
	$\sigma_{2\eta_1}\widehat{\otimes} \sigma'_{\eta'_1}\oplus \sigma_{4\eta_1}\widehat{\otimes} \sigma'_{0}$&
	$\frac{29}{36}$  &
	unknown &
	$\lambda^{\pi_{2\omega_3}}=\frac{13}{9}$
\\ \hline 
32&
	$\op{G}_2$ & 
	$\SO(3)$          & 
	$\sigma_{6\eta_1}$&
	$\frac{43}{56}$  &
	unknown &
	$\lambda^{\pi_{2\omega_2}}=\frac{5}{2}$
\\ \hline 
33&
	$\op{G}_2$ & 
	$\SU(3)$           & 
	$\sigma_{\eta_1}\oplus \sigma_{\eta_2}$&
	$\frac{5}{6}$    &
	$\nu$-stable* &
	$\lambda^{\pi_{\omega_1}}=\frac{1}{2}$
\\  \hline 
\end{tabular}

\

* $(\op{G}_2/\SU(3),g_{\st})\simeq (S^6,g_{\text{round}})$ is $\nu$-stable by Remark~\ref{rem:spheres}.

\end{table} 
}

%\end{comment}

\subsection{Symplectic cases for isotropy irreducible spaces}\label{subsec4:Sp}
In this subsection we prove that $\lambda_1(G/H,g_{\st})\geq1$ for all cases $G/H$ from Tables~\ref{table4:isotropyirred-families}--\ref{table4:isotropyirredexcepcion-exceptional} with $\fg=\spp(N)$ for some $N\in\N$, namely, Family IV and isolated cases No.~5--9. 
Note that $G=\Sp(N)/\Z_2$ in all these cases. 
Now, Proposition~\ref{prop:sufficient}\eqref{item3:Sp(N)/Z_2} yields that it suffices to show that 
\begin{equation}\label{eq4-Sp:V_(omega2)^H=0}
V_{\pi_{\omega_2}}^H=0. 
\end{equation}
In all cases within this subsection, the goal will be to obtain the decomposition in irreducible $H$-modules of $\pi_{\omega_2}|_H$ and check that the trivial representation $1_H$ of $H$ does not occurs in it, concluding that \eqref{eq4-Sp:V_(omega2)^H=0} holds by Lemma~\ref{lem3:dimV_pi^H}. 

\smallskip
\noindent
$\bullet$
Family IV: 
We have from \cite{Wolf84-errata-isot-irred} that, for any $n\geq3$, $G=\Sp(n)/\Z_2$ and 
$$
H\simeq \begin{cases}
(\Sp(1)/\Z_2)\times(\SO(n)/\Z_2)&\text{ if $n$ is even,}\\
(\Sp(1)/\Z_2)\times \SO(n)&\text{ if $n$ is odd. }
\end{cases}
$$

We set $H'=\SU(2)\times\SO(n)$ and $\rho=\st_{\SU(2)} \widehat\otimes \,\st_{\SO(n)}:H'\to\GL(\C^2\otimes \C^{n})$.
Since $\st_{\SU(2)}=\sigma_{\eta_1}$ is of quaternionic type and $\st_{\SO(n)}$ is of real type, it follows that $\rho$ is of quaternionic type. 
Thus $\rho:H'\to \Sp(n)$, and $H=\rho_0(H')\simeq H'/\op{Ker}(\rho_0)$, where $\rho_0:H'\to G=\Sp(n)/\Z_2$ is the only Lie group homomorphism satisfying $d\rho_0=d\rho$.

We will work only at the Lie algebra level without losing generality because $H$ and $H'$ are connected. 
We have that 
$\st_{\SO(3)} = \sigma_{2\eta_1'}'$, 
$\st_{\SO(n)} = \sigma_{\eta_1'}'$ for $n\geq5$, 
and 
$\st_{\SO(4)}=\sigma_{2\eta_1'}'\widehat\otimes \, \sigma_{2\eta_1''}''$ using $\so(4)\simeq\su(2)\oplus\su(2)$.

It is well known that the standard representation $\pi_{\omega_1}$ of $\fg_\C$ satisfies 
\begin{equation}\label{eq4-Sp:Lambda^2st}
{\textstyle\bigwedge^2}\pi_{\omega_1}
\simeq  \pi_{\omega_2}\oplus  1_{\fg_\C}
\end{equation}
as representations of $\fg_\C$. 
Moreover, $\pi_{\omega_1}|_{\fh_\C'}= (d\rho)_\C$. 
(Note that $\pi_{\omega_1}$ does not descend to a representation of $G=\Sp(n)/\Z_2$; this is the reason why we work at the Lie algebra level.)

It follows immediately from \eqref{eq2:Lambda^2(VxW)-Sym^2(VxW)} that 
\begin{equation*}
\begin{aligned}
\big({\textstyle \bigwedge^2}(\pi_{\omega_1})\big)|_{\fh_\C'}&
=
{\textstyle \bigwedge^2}((d\rho)_\C) 
=
{\textstyle \bigwedge^2(\sigma_{\eta_1}\widehat\otimes \,\st_{\SO(n)})} 
\\ &
\simeq {\textstyle \bigwedge^2}(\sigma_{\eta_1}) \,\widehat\otimes \, \Sym^2(\st_{\SO(n)})
\oplus \Sym^2(\sigma_{\eta_1}) \,\widehat\otimes \, {\textstyle \bigwedge^2}(\st_{\SO(n)})
.
\end{aligned} 
\end{equation*}
It is well known that
${\textstyle \bigwedge^2}(\sigma_{\eta_1})\simeq \sigma_0$,
$\Sym^2(\sigma_{\eta_1})\simeq \sigma_{2\eta_1}$, 
$\Sym^2(\st_{\SO(3)})=\Sym^2(\sigma_{2\eta_1'}') \simeq \sigma_{4\eta_1'}'\oplus \sigma_0'$, 
$\Sym^2(\st_{\SO(4)})
	=\Sym^2(\sigma_{2\eta_1'}'\oplus \sigma_{2\eta_1''}'')
	\simeq \sigma_{0}'\widehat\otimes  \sigma_{0}''\oplus \sigma_{2\eta_1'}'\widehat\otimes \sigma_{2\eta_1''}''$,
$\Sym^2(\st_{\SO(n)})=\Sym^2(\sigma_{\eta_1'}') \simeq \sigma_{2\eta_1'}'\oplus \sigma_0'$ for $n\geq5$, 
Furthermore, ${\textstyle \bigwedge^2}(\st_{\SO(n)})$ is equivalent to 
	$\sigma_{2\eta_1'}'$ for $n=3$, 
	$\sigma_{2\eta_1'}'\widehat\otimes  \sigma_{0}''\oplus \sigma_{0}'\widehat\otimes \sigma_{2\eta_1''}''$ for $n=4$, 
	$\sigma_{2\eta_2'}'$ for $n=5$, 
	$\sigma_{\eta_2'+\eta_3'}'$ for $n=6$, and 
	$\sigma_{\eta_2'}'$ for $n\geq7$. 
Hence  
\begin{equation}\label{eq4:Sp:Lambda^2st|_h}
{\textstyle \bigwedge^2}(\pi_{\omega_1})|_{\fh_\C'}
\simeq 
\begin{cases}
\sigma_{0}\widehat\otimes \, \sigma_{0}' 
\oplus 
\sigma_{0}\widehat\otimes \, \sigma_{4\eta_1'}' 
\oplus 
\sigma_{2\eta_1}\widehat\otimes \, \sigma_{2\eta_1'}' 
	&\text{for }n=3,
\\
\sigma_{0}\widehat\otimes\sigma_{0}'\widehat\otimes  \sigma_{0}''
\oplus 
\sigma_{0}\widehat\otimes \sigma_{2\eta_1'}'\widehat\otimes \sigma_{2\eta_1''}''
\\ \qquad 
\oplus
\sigma_{2\eta_1}\widehat\otimes \, 
\sigma_{2\eta_1'}'\widehat\otimes  \sigma_{0}''
\oplus 
\sigma_{2\eta_1}\widehat\otimes \, 
\sigma_{0}'\widehat\otimes \sigma_{2\eta_1''}''

	&\text{for }n=4,
\\

\sigma_{0}\widehat\otimes \, \sigma_{0}' 
\oplus
\sigma_{0}\widehat\otimes \, \sigma_{2\eta_1'}' 
\oplus 
\sigma_{2\eta_1}\widehat\otimes \, \sigma_{2\eta_2'}' 
	&\text{for }n=5,
\\
\sigma_{0}\widehat\otimes \, \sigma_{0}' 
\oplus 
\sigma_{0}\widehat\otimes \, \sigma_{2\eta_1'}' 
\oplus 
\sigma_{2\eta_1}\widehat\otimes \, \sigma_{\eta_2'+\eta_3'}' 
	&\text{for }n=6,
\\
\sigma_{0}\widehat\otimes \, \sigma_{0}' 
\oplus 
\sigma_{0}\widehat\otimes \, \sigma_{2\eta_1'}' 
\oplus 
\sigma_{2\eta_1}\widehat\otimes \, \sigma_{\eta_2'}' 
	&\text{for }n\geq 7.
\end{cases}
\end{equation}
Cancelling the trivial representation of $\fh_\C'$ (i.e.\ $\sigma_0\widehat\otimes \,\sigma_{0}'$ for $n\neq 4$ and $\sigma_0\widehat\otimes \,\sigma_{0}'\widehat\otimes \,\sigma_{0}''$ for $n=4$)
in \eqref{eq4:Sp:Lambda^2st|_h} with the trivial representation $1_{\fg_\C}|_{\fh_\C'}=1_{\fh_\C'}$ within $\bigwedge^2\pi_{\omega_1}|_{\fh_\C'}$ from \eqref{eq4-Sp:Lambda^2st}, we conclude that 
\begin{equation}\label{eq4SpFliaIV:branching-omega2}
\pi_{\omega_2}|_{\fh_\C'}
\simeq 
\begin{cases}
\sigma_{0}\widehat\otimes \, \sigma_{4\eta_1'}' 
\oplus 
\sigma_{2\eta_1}\widehat\otimes \, \sigma_{2\eta_1'}' 
	&\text{for }n=3,
\\
\sigma_{0}\widehat\otimes \sigma_{2\eta_1'}'\widehat\otimes \sigma_{2\eta_1''}''
\oplus
\sigma_{2\eta_1}\widehat\otimes \, 
\sigma_{2\eta_1'}'\widehat\otimes  \sigma_{0}''
\oplus 
\sigma_{2\eta_1}\widehat\otimes \, 
\sigma_{0}'\widehat\otimes \sigma_{2\eta_1''}''

	&\text{for }n=4,
\\

\sigma_{2\eta_1}\widehat\otimes \, \sigma_{2\eta_2'}' 
\oplus \sigma_{0}\widehat\otimes \, \sigma_{2\eta_1'}' 
	&\text{for }n=5,
\\
\sigma_{2\eta_1}\widehat\otimes \, \sigma_{\eta_2'+\eta_3'}' 
\oplus \sigma_{0}\widehat\otimes \, \sigma_{2\eta_1'}' 
	&\text{for }n=6,
\\
\sigma_{2\eta_1}\widehat\otimes \, \sigma_{\eta_2'}' 
\oplus \sigma_{0}\widehat\otimes \, \sigma_{2\eta_1'}' 
	&\text{for }n\geq 7.
\end{cases}
\end{equation}
Since the trivial representation of $\fh_\C'$ does not occur in $\pi_{\omega_2}|_{\fh_\C'}$, we conclude by Lemma~\ref{lem3:dimV_pi^H} that $V_{\pi_{\omega_2}}^{H}=V_{\pi_{\omega_2}}^{\rho_0(H')}=0$, so \eqref{eq4-Sp:V_(omega2)^H=0} holds for $n\geq3$ as claimed.

\smallskip
\noindent
$\bullet$
Isolated case No.~5: 
$\fg\simeq\spp(2)$ and $\fh'=\so(3)$. 

We might argue as above to determine $\pi_{\omega_2}|_{\fh'}$. 
Fortunately, there are quicker ways such as using a computer program (e.g.\ \cite{Sage}) or looking at branching rules tables (e.g.\ \cite{LieART}). 
We will pick the second choice most of the times.

We have that $\dim V_{\pi_{\omega_2}}=5$.
The second row of the table beginning on page 330 in \cite{LieART} says `$\textbf{5}=\textbf{5}$', ensuring $\pi_{\omega_2}|_{\fh'}\simeq \sigma_{4\eta_1}$. 
Indeed, the value $\textbf{5}$ at the left-hand side corresponds via \cite[Table A.43 in pp.~85]{LieART} to $\pi_{\omega_2}$ marked with Dynkin label $(0,1)$, and similarly the value $\textbf{5}$ at the right-hand side corresponds to $\sigma_{4\eta_1}$. 
(Although the value $5$ is the dimension of the representation, there might be in general other irreducible representations with the same dimension.) 

Since $\pi_{\omega_2}|_{\fh'}\simeq\sigma_{4\eta_1}$, Lemma~\ref{lem3:dimV_pi^H} implies that \eqref{eq4-Sp:V_(omega2)^H=0} holds. 

\begin{remark}[References to LieART] \label{rem4:LieART}
The article \cite{LieART} supporting the computer program LieART 2.0 will be cited several times in the sequel without giving details as above. 
The page number in \cite{LieART} will always correspond to the arXiv's version of \cite{LieART}, which includes the long tables with the branching rules in Appendix A. 
The authors double checked all branching laws using \cite{Sage}. 
For instance, the following code returns $\pi_{\omega_2}|_{\fh'}$:
\begin{lstlisting}
	sage: G=WeylCharacterRing("C2", style="coroots")
	sage: H=WeylCharacterRing("A1", style="coroots")
	sage: G(0,1).branch(H,rule=branching_rule(G,H, rule="symmetric_power"))
	A1(4)
\end{lstlisting}
\end{remark}

\smallskip
\noindent
$\bullet$
Isolated case No.~6: 
$\fg\simeq\spp(7)$ and $\fh'=\spp(3)$. 
We have that $\pi_{\omega_2}|_{\fh'}\simeq \sigma_{2\eta_2}$ 
(see e.g.\ \cite[pp.\ 350]{LieART}; $\dim V_{\pi_{\omega_2}}=90$), 
thus \eqref{eq4-Sp:V_(omega2)^H=0} holds by Lemma~\ref{lem3:dimV_pi^H}. 

\smallskip
\noindent
$\bullet$
Isolated case No.~7: 
$\fg\simeq\spp(10)$ and $\fh'=\su(6)$. 
We have that $\pi_{\omega_2}|_{\fh'}\simeq \sigma_{\eta_2+\eta_4}$ 
(see e.g.\ \cite[pp.\ 359]{LieART}; $\dim V_{\pi_{\omega_2}}=189$), 
thus \eqref{eq4-Sp:V_(omega2)^H=0} holds by Lemma~\ref{lem3:dimV_pi^H}. 

\smallskip
\noindent
$\bullet$
Isolated case No.~8: 
$\fg\simeq\spp(16)$ and $\fh'=\so(12)$. 
One can see (e.g.\ using \cite{Sage}) that $\pi_{\omega_2}|_{\fh'}\simeq \sigma_{\eta_4}$, 
thus \eqref{eq4-Sp:V_(omega2)^H=0} holds by Lemma~\ref{lem3:dimV_pi^H}. 
 
\smallskip
\noindent
$\bullet$
Isolated case No.~9: 
$\fg\simeq\spp(28)$ and $\fh'=\fe_7$. 
One can see (e.g.\ using \cite{Sage}) that $\pi_{\omega_2}|_{\fh'}\simeq \sigma_{\eta_6}$, 
thus \eqref{eq4-Sp:V_(omega2)^H=0} holds by Lemma~\ref{lem3:dimV_pi^H}.

\subsection{Special unitary cases for isotropy irreducible spaces}\label{subsec4:SU}
In this subsection we prove that $\lambda_1(G/H,g_{\st})\geq1$ for each $G/H$ from Tables~\ref{table4:isotropyirred-families}--\ref{table4:isotropyirredexcepcion-exceptional} with $\fg=\su(N)$ for some $N\in\N$. 

\smallskip
\noindent
$\bullet$
Family I: 
We have that $G=\SU(N)/\Z_d$ with $N=\frac{n(n-1)}2$ and $H\simeq \SU(n)/\Z_n$ for $n\geq5$, where $d=n/2$ if $n$ is even, or $d=n$ if $n$ is odd. 
Since $N\geq 10>7$ and $d\geq3$,  $\lambda_1(G/H,g_{\st})\geq1$ immediately follows from Proposition~\ref{prop:sufficient}\eqref{item3:SU/Zd-d!=1,2}.

\smallskip
\noindent
$\bullet$
Family II: 
$G=\SU(N)/\Z_d$ with $N=\frac{n(n+1)}2$ and $H\simeq \SU(n)/\Z_n$ for $n\geq3$, where $d=n/2$ if $n$ is even, or $d=n$ if $n$ is odd. 

If $n\geq 5$, then $N\geq 15>7$ and $d\geq3$, thus the assertion follows from Proposition~\ref{prop:sufficient}\eqref{item3:SU/Zd-d!=1,2}.

For $n=4$ we have that $N=10$ and $d=2$.  Proposition~\ref{prop:sufficient}\eqref{item3:SU/Zd-d=2} ensures that it is sufficient to prove that $V_{\pi_{\omega_2}}^H=0$.

For $n=3$ we have that $N=6$ and $d=3$.
Proposition~\ref{prop:sufficient}\eqref{item3:SU/Zd-d=3,N=6} ensures that it is sufficient to prove that $V_{\pi_{\omega_3}}^H=0$.

Both follow from Lemma~\ref{lem3:dimV_pi^H} since $\pi_{\omega_2}|_{\fh_\C'}\simeq \sigma_{2\eta_1+\eta_2}$ for $n=4$ (see e.g.\ \cite[pp.\ 171]{LieART}; $\dim V_{\pi_{\omega_2}}=45$) and
$\pi_{\omega_3}|_{\fh_\C'}\simeq \sigma_{3\eta_1}\oplus \sigma_{3\eta_2}$ for $n=3$ (see e.g.\ \cite[pp.\ 157]{LieART}; $\dim V_{\pi_{\omega_3}}=20$).

\smallskip
\noindent
$\bullet$
Family III: 
$G=\SU(pq)/\Z_d$ and $H\simeq (\SU(p)/\Z_p)\times(\SU(q)/\Z_q)$ for $p\geq q\geq 2$, $pq>4$, where $d=\op{lcm}(p,q)$. 

The cases when $d>3$ follow from Proposition~\ref{prop:sufficient}\eqref{item3:SU/Zd-d!=1,2,3}. 
One can check that $d=3$ only if $p=q=3$, in which case follows from
Proposition~\ref{prop:sufficient}\eqref{item3:SU/Zd-d!=1,2} since $N=pq>7$.

\smallskip
\noindent
$\bullet$
Isolated case No.~1:
$G=\SU(16)/\Z_4$ and $H\simeq \SO(10)/\Z_2$. 
This case follows from  Proposition~\ref{prop:sufficient}\eqref{item3:SU/Zd-d!=1,2,3} since $d=4>3$.

\smallskip
\noindent
$\bullet$
Isolated case No.~2:
$G=\SU(27)/\Z_3$ and $H\simeq \op{E}_6/\Z_3$. 
This case follows from  Proposition~\ref{prop:sufficient}\eqref{item3:SU/Zd-d!=1,2} since $N=27>7$.

\subsection{Orthogonal cases for isotropy irreducible spaces} \label{subsec4:SO}

In this subsection we prove that $\lambda_1(G/H,g_{\st})\geq1$ for each $G/H$ from Tables~\ref{table4:isotropyirred-families}--\ref{table4:isotropyirredexcepcion-exceptional} with $\fg=\so(N)$ for some $N\in\N$ with the exception of $\Spin(7)/\op{G}_2$, namely, Families V--X and isolated cases No.~4 and 10--18.

As we mentioned in Remark~\ref{rem4:embedding}, the embedding of $H$ into $G$ is given by $H=\rho(H')\simeq H'/\op{Ker}(\rho)$, where $H'$ is the connected and simply connected Lie group with Lie algebra $\fh$ and $\rho$ is the Lie group homomorphism with complexified  differential $(d\rho)_\C:\fh'\to \fg=\fg_{\rho}\subset\gl(V_\rho)$ as indicated in the tables, which is irreducible as a representation of $\fh_\C'$. 
Now, Lemma~\ref{lem3:standard} yields
\begin{equation}\label{eq4-SO:V_st^H=0}
V_{\pi_{\omega_1}}^{\fh_\C}=0. 
\end{equation}

\smallskip
\noindent
$\bullet$
Family V: 
For $n\geq3$, we have $G=\Spin(n^2-1)$ if $n$ is odd, $G=\SO(n^2-1)$ if $n$ is even, and $H=\SU(n)/\Z_n$. 
%\begin{align*}
%G&=
%\begin{cases}
%\Spin(n^2-1)&\text{ $n$ odd},\\
%\SO(n^2-1)&\text{ $n$ even},
%\end{cases}
%&
%\text{and}\quad H&\simeq \SU(n)/\Z_n.
%\end{align*}

For $n\geq4$, we have $N=n^2-1\geq 15$.  Proposition~\ref{prop:sufficient}\eqref{item3:so-N>=15} and \eqref{eq4-SO:V_st^H=0} give  $\lambda_1(G/H,g_{\st})\geq1$.  

For $n=3$, $G=\Spin(8)$.  
By Proposition~\ref{prop:sufficient}\eqref{item3:so-N<15} and \eqref{eq4-SO:V_st^H=0}, it is enough to show $V_{\pi_{\omega_4}}^{\fh_\C}=0$, which follows from Lemma~\ref{lem3:dimV_pi^H} since $\pi_{\omega_4}|_{\fh_\C'}\simeq \sigma_{\eta_1+\eta_2}$ (see e.g.\ \cite[pp.\ 209]{LieART}; $\dim V_{\pi_{\omega_4}} =8$).

\smallskip
\noindent
$\bullet$
Family VI: 
For $n\geq3$, we have 
$G=\Spin(2n^2-n-1)$ if $n\equiv0,1\pmod 4$, 
$G=\SO(2n^2-n-1)$ if $n\equiv2,3\pmod 4$,
and $H\simeq \Sp(n)/\Z_2$. 
%\begin{align*}
%G&=
%\begin{cases}
%\Spin(2n^2-n-1)&n\equiv0,1\pmod 4,\\
%\SO(2n^2-n-1)&n\equiv2,3\pmod 4,
%\end{cases}
%&
%\text{and}\quad H&\simeq \Sp(n)/\Z_2.
%\end{align*}

For $n\geq4$, $N=2n^2-n-1\geq 27$, and the case follows from Proposition~\ref{prop:sufficient}\eqref{item3:so-N>=15} and \eqref{eq4-SO:V_st^H=0}. 

For $n=3$, $G=\SO(14)$, and the case follows from Proposition~\ref{prop:sufficient}\eqref{item3:SO} and \eqref{eq4-SO:V_st^H=0}.

\smallskip
\noindent
$\bullet$
Family VII: 
For $n\geq2$, we have
$G=\Spin(2n^2+n)$ if $n\equiv0,3\pmod 4$, 
$G=\SO(2n^2+n)$ if $n\equiv1,2\pmod 4$,
and $H\simeq \Sp(n)/\Z_2$. 
%\begin{align*}
%G&=
%\begin{cases}
%\Spin(2n^2+n)&n\equiv0,3\pmod 4,\\
%\SO(2n^2+n)&n\equiv1,2\pmod 4,
%\end{cases}
%&
%\text{and}\quad H&\simeq \Sp(n)/\Z_2.
%\end{align*}

For $n\geq3$, $N=2n^2+n\geq 20\geq15$, so the case follows from Proposition~\ref{prop:sufficient}\eqref{item3:so-N>=15} and \eqref{eq4-SO:V_st^H=0}. 

For $n=2$, $G=\SO(10)$, and the case follows from Proposition~\ref{prop:sufficient}\eqref{item3:SO} and \eqref{eq4-SO:V_st^H=0}.

\smallskip
\noindent
$\bullet$
Family VIII: 
For $n\geq2$, we have $G=\SO(4n)/\Z_2$ and $H\simeq (\Sp(1)/\Z_2)\times (\Sp(n)/\Z_2)$. 

Since $N=4n$ is even, the case follows from  Proposition~\ref{prop:sufficient}\eqref{item3:SO/Z2}.

\smallskip
\noindent
$\bullet$
Family IX: 
For $n\geq7$, we have $\fg=\so\big(\frac{n(n-1)}{2}\big)$ and $\fh\simeq \so(n)$. 

For $n\geq7$ we have $N=\binom{n}{2}\geq 21$, so the case follows from Proposition~\ref{prop:sufficient}\eqref{item3:so-N>=15} and \eqref{eq4-SO:V_st^H=0}.

\smallskip
\noindent
$\bullet$
Family X: 
For $n\geq5$, we have $\fg=\so\big(\frac{(n-1)(n+2)}{2}\big)$ and $\fh\simeq \so(n)$. 

For $n\geq6$ we have $N=\frac{(n-1)(n+2)}{2}\geq 20$, so the case follows from Proposition~\ref{prop:sufficient}\eqref{item3:so-N>=15} and \eqref{eq4-SO:V_st^H=0}. 

For $n=5$ we have $G=\SO(14)$, and the case follows from Proposition~\ref{prop:sufficient}\eqref{item3:SO} and \eqref{eq4-SO:V_st^H=0}.

\smallskip
\noindent
$\bullet$
Isolated case No.~10:
$G=\Spin(14)$ and $H\simeq\op{G}_2$. 

By Proposition~\ref{prop:sufficient}\eqref{item3:so-N<15} and \eqref{eq4-SO:V_st^H=0}, it is enough to prove $V_{\pi_{\omega_7}}^H=0$. 
This follows from Lemma~\ref{lem3:dimV_pi^H} since $\pi_{\omega_7}|_{\fh_\C}\simeq \sigma_{\eta_1+\eta_2}$ (see e.g.\ \cite[pp.~255]{LieART}; $\dim V_{\pi_{\omega_7}}=64$).

\smallskip
\noindent
$\bullet$
Isolated cases No.~11--18:
We always have $\fg=\so(N)$ for some $N\geq16$, thus the case follows from Proposition~\ref{prop:sufficient}\eqref{item3:so-N>=15} and \eqref{eq4-SO:V_st^H=0}.

\subsection{Exceptional cases for isotropy irreducible spaces} \label{subsec4:exceptional}

In this subsection we prove that $\lambda_1(G/H,g_{\st})\geq1$ for each $G/H$ from Tables~\ref{table4:isotropyirred-families}--\ref{table4:isotropyirredexcepcion-exceptional} with $\fg_\C$ an exceptional complex simple Lie algebra excluding $\op{G}_2/\SU(3)$, namely, Isolated cases No.~19--32.

\smallskip
\noindent
$\bullet$
Isolated cases No.~19 and 21:
We have $G=\op{E}_6$ in the three cases. Proposition~\ref{prop:sufficient}\eqref{item3:E6} ensures that it is sufficient to show that $V_{\pi_{\omega_1}}^H=0$. 
This follows from Lemma~\ref{lem3:dimV_pi^H} since $\pi_{\omega_1}|_{\su(3)_\C}\simeq \sigma_{2(\eta_1+\eta_2)}$ 
and 
$\pi_{\omega_1}|_{(\fg_2)_\C}\simeq \sigma_{2\eta_1}$ 
(see e.g.\ \cite[pp.~378, 377]{LieART}; $\dim V_{\pi_{\omega_1}}=27$).

\smallskip
\noindent
$\bullet$
Isolated cases No.~20 and 22:
We have $G=\op{E}_6/\Z_3$, and the case follows from Proposition~\ref{prop:sufficient}\eqref{item3:E6/Z3}.

\smallskip
\noindent
$\bullet$
Isolated case No.~23:
We have $G=\op{E}_7$, and the case follows from Proposition~\ref{prop:sufficient}\eqref{item3:E7} and Lemma~\ref{lem3:dimV_pi^H} since 
	$\pi_{\omega_7}|_{\fh_\C'}\simeq \sigma_{6\eta_1}\oplus\sigma_{6\eta_2}$ 
(see e.g.\ \cite[pp.~381]{LieART}; $\dim V_{\pi_{\omega_7}}=56$).

\smallskip
\noindent
$\bullet$
Isolated cases No.~24--26:
We have $G=\op{E}_7/\Z_2$, and the case follows from Proposition~\ref{prop:sufficient}\eqref{item3:E7/Z2}. 

\smallskip
\noindent
$\bullet$
Isolated cases No.~27--29:
We have $G=\op{E}_8$, and the case follows from Proposition~\ref{prop:sufficient}\eqref{item3:E8}. 

\smallskip
\noindent
$\bullet$
Isolated cases No.~30--31:
We have $G=\op{F}_4$, and the cases follow from  Proposition~\ref{prop:sufficient}\eqref{item3:F4} and Lemma~\ref{lem3:dimV_pi^H} since
	$\pi_{\omega_4}|_{\su(3)_\C\oplus\su(3)_\C}\simeq \sigma_{\eta_1}\widehat\otimes\,\sigma_{\eta_1'}' \oplus \sigma_{\eta_2}\widehat\otimes\,\sigma_{\eta_2'}' \oplus \sigma_{0}\widehat\otimes\,\sigma_{\eta_1'+\eta_2'}'
$ 
and 
	$\pi_{\omega_4}|_{\su(2)_\C\oplus (\fg_2)_\C}\simeq \sigma_{2\eta_1}\widehat\otimes\,\sigma_{\eta_1'}' \oplus \sigma_{4\eta_1}\widehat\otimes\,\sigma_{0}'$ 
(see e.g.\ \cite[pp.~387, 391]{LieART}; $\dim V_{\pi_{\omega_4}}=26$).

\smallskip
\noindent
$\bullet$
Isolated case No.~32:
We have $G=\op{G}_2$, and the case follows from  Proposition~\ref{prop:sufficient}\eqref{item3:G2} and Lemma~\ref{lem3:dimV_pi^H} since
	$\pi_{\omega_1}|_{\fh_\C'}\simeq \sigma_{6\eta_1}$ 
(see e.g.\ \cite[pp.~400]{LieART}; $\dim V_{\pi_{\omega_1}}=7$).

\subsection{Round spheres} \label{subsec4:roundspheres}

The isotropy irreducible simply connected spaces $\op{G}_2/\SU(3)$ (No.~33) and $\Spin(7)/\op{G}_2$ (No.~3) endowed with their standard metrics are round spheres of dimension $6$ and $7$ respectively. 
Let $G/H$ denote any of them. 
We will show first that $\lambda_1(G/H,g_{\st})<2E$. 
Additionally, we will show that $\lambda_1(G'/H',g_{\st})>2E'=2E$ for any strongly isotropy irreducible space $G'/H'$ associated to the pair $(\fg,\fh)$ as above.

We first mention a useful fact. 
Let $g_{\text{round}}$ denote the Riemannian metric with constant sectional curvature $1$ on the $d$-dimensional real projective space $P^d(\R)$.
It is well known that the Einstein factor of $(P^d(\R),g_{\text{round}})$ is $E=d-1$ and  
\begin{equation}\label{eq:lambda_1/E-roundP^d(R)}
\lambda_1(P^d(\R),g_{\text{round}})
=2(d+1)
=2E+4>2E. 
\end{equation}

\smallskip
\noindent
$\bullet$
Isolated case No.~33:
$G=\op{G}_2$ and $H\simeq \SU(3)$.
We have that
$\pi_{\omega_1}|_H\simeq \sigma_0\oplus\sigma_{\eta_1}\oplus\sigma_{\eta_2}$ (see e.g.\ \cite[pp.~ 392]{LieART}; $\dim V_{\pi_{\omega_1}}=7$).
Therefore $V_{\pi_{\omega_1}}^{H}\neq0$ by Lemma~\ref{lem3:dimV_pi^H}. 
Since $\lambda^{\pi_{\omega_1}}$ is minimum among $\lambda^{\pi}$ for $\pi\in\widehat {G}\smallsetminus\{1_{G}\}$ by Table~\ref{table:CasimirEigenvalues}, 
Theorem~\ref{thm:Spec(standard)} implies that 
$
\lambda_1(\op{G}_2/\SU(3),g_{\st})=\lambda^{\pi_{\omega_1}} =\frac12<\frac56=2E.
$

Let $G'/H'$ be a strongly isotropy irreducible space associated to the pair $(\fg_2,\su(3))$ such that $(G'/H',g_{\st})$ is not isometric to $(G/H,g_{\st})$. 
Since $(G/H,g_{\st})\to (G'/H',g_{\st})$ is a proper Riemannian covering and $(G/H,g_{\st})$ is isometric to a round $6$-sphere, we deduce that $(G'/H',g_{\st})$ is homothetic to $(P^6(\R),g_{\text{round}})$. 
It follows from \eqref{eq:lambda_1/E-roundP^d(R)} that 
$
\lambda_1(G'/H',g_{\st})>2E,
$
as claimed.

\smallskip
\noindent
$\bullet$
Isolated case No.~3:
$G=\Spin(7)$ and $H\simeq \op{G}_2$. 
One gets
$\pi_{\omega_3}|_H\simeq \sigma_0\oplus\sigma_{\eta_1}$
(see e.g.\ \cite[pp.~195]{LieART}; $\dim V_{\pi_{\omega_3}}=8$).
Thus $V_{\pi_{\omega_3}}^{H}\neq0$ by Lemma~\ref{lem3:dimV_pi^H}. 
Similarly as above, Theorem~\ref{thm:Spec(standard)} implies that 
$
\lambda_1(\Spin(7)/\op{G}_2,g_{\st}) =\lambda^{\pi_{\omega_3}} =\frac{21}{40}<\frac{9}{10}=2E.
$

Let $G'/H'$ be a strongly isotropy irreducible space associated to the pair $(\so(7),\fg_2)$ such that $(G'/H',g_{\st})$ is not isometric to $(G/H,g_{\st})$.
The strategy used in the previous case does not work here since odd-dimensional round spheres covers infinitely many Riemannian manifolds. 

We can assume that $G'$ acts effectively on $G'/H'$ without losing generality. 
According to the second part of Theorem~11.1 in \cite{Wolf68}, there is a central subgroup $Q$ of the center $Z:=Z(\Spin(7))$ of $\Spin(7)$ and a subgroup $H''$ of the normalizer $N_G(H)$ of $H$ in $G$ with $H=H_o''=N_G(H)_o$, such that
$
G'=G/Q
$
and 
$K'=(Q\cdot K'')/Q$.
Since $Z\simeq\Z_2$ and $N_G(H)=Z\cdot H$, it follows that $Q=Z(\Spin(7))$, $G'=\SO(7)$ and $H'=\SU(3)$. 
Hence $(G'/H',g_{\st})$ is homothetic to
$(P^6(\R),g_{\text{round}})$, thus 
$\lambda_1(G'/H',g_{\st})>2E$ by \eqref{eq:lambda_1/E-roundP^d(R)}. 

This completes the proof of Theorem~\ref{thm4:main}.

\subsection{Consequences on \texorpdfstring{$\nu$}{nu}-stability} \label{subsec4:nu-stability}

Theorem~\ref{thm4:main} ensures that there are no $\nu$-unstable conformal directions in every non-symmetric strongly isotropy irreducible space (notice the two exceptions are round spheres, so they are $\nu$-stable; see Remark~\ref{rem:spheres}). 
Consequently, any of these spaces established as H.stable is automatically $\nu$-stable. 
Schwahn in \cite{Schwahn-Lichnerowicz} did it for several of them. 
All this information is present in Table~\ref{table4:isotropyirred-families-lambda1} for the families, and inside Tables~\ref{table4:isotropyirredexcepcion-classical}--\ref{table4:isotropyirredexcepcion-exceptional} for the isolated cases.

\begin{table}
\caption{$\nu$-stability and $\lambda_1$ for the 10 families of simply connected strongly isotropy irreducible spaces.}\label{table4:isotropyirred-families-lambda1}

{\renewcommand{\arraystretch}{1.3}
	\begin{tabular}{ccccc} 
		No. & $\tfrac{\fg}{\fh}$ 
		& $2E$ & $\nu$-stable & 
		\begin{tabular}{c}
			\tiny computational calculation \\[-10pt] 
			\tiny of $\lambda_1(G/H,g_{\st})$
		\end{tabular}
		\\ [0.5ex]
		\hline \hline 
%		\\ [-3ex]
		I & $\tfrac{\su\left(\frac{n(n-1)}{2}\right)}{\su(n)}$ &
		$\frac{1}{2}+\frac{4}{n(n-2)}$ & 
		$8\leq n\leq 11$ &
		\tiny
		\begin{tabular}{c}
			$\lambda^{\pi_{3\omega_1}}=\tfrac{42}{25}$\;\; $n=6$\\
			$\lambda^{\pi_{2\omega_1+2\omega_{N}}}$\;\; 
			$n=5,7,8,9$
			\\
			where $N=\frac{n(n-1)}{2}-1$
		\end{tabular} 
		\\ [1.5ex]
		\hline 
%		\\ [-3ex]
		II & 
		$\tfrac{\su\left(\frac{n(n+1)}{2}\right)}{\su(n)}$  & 
		$\frac{1}{2}+\frac{4}{n(n+2)}$ & 
		$6\leq n\leq 8$ & 
		\tiny
		\begin{tabular}{cc}
			$\lambda^{\pi_{3\omega_1}}=\tfrac{15}{8}$\;\;$n=3$, 
			\\
			$\lambda^{\pi_{2\omega_1+2\omega_{N}}}$\;\;$4\leq n\leq 8$ 
			\\
			where $N=\frac{n(n+1)}{2}-1$
		\end{tabular}  
		
		\\[1.5ex] \hline 
%		\\ [-3ex]

		III	& 
		$\tfrac{\su(pq)}{\su(p)\oplus \su(q)}$ &
		$\frac{1}{2}+\frac{p^2+q^2}{p^2 q^2}$ & 
		{\tiny
			\begin{tabular}{l}
				$p=3$, $12\leq q\leq 16$, or
				\\[-3pt]
				$4\leq p\leq q$, $24\leq pq\leq 49$
			\end{tabular}
		} & $\lambda^{\pi_{\omega_2+\omega_{pq-2}}}\quad pq\leq 25$
		\\ [1.5ex] \hline 
%		\\ [-3ex]
		IV & $\tfrac{\spp(n)}{\spp(1)\oplus \so(n)}$ &
		$\frac{3}{4}+\frac{8-n}{4n(n+1)}$ && 
		$\lambda^{\pi_{2\omega_2}}$ \, $3\leq n\leq 14$
		\\ \hline 
%		\\ [-3ex]
		
		V & $\tfrac{\so(n^2-1)}{\su(n)}$ & 
		$\frac{1}{2}+\frac{2}{n^2-3}$ &&
		\scriptsize 
		\begin{tabular}{l}
			$\lambda^{\pi_{\omega_3+\omega_4}}=\frac{5}{4}\;\; n=3$\\[-2pt]
			$\lambda^{\pi_{\omega_3}}\;\; 4\leq n\leq 8$
		\end{tabular}
		\\ [1.5ex] \hline 
		\\ [-3ex]
		
		VI & $\tfrac{\so(2n^2-n-1)}{\spp(n)}$ & 
		$\frac{1}{2}+\frac{2}{(n-1)(n+1)(2n-3)}$ &&
		$\lambda^{\pi_{3\omega_1}} \quad  3\leq n\leq 6$
		\\[1.5ex] \hline 
		\\ [-3ex]
		
		VII  & $\tfrac{\so(2n^2+n)}{\spp(n)}$ & 
		$\frac{1}{2}+\frac{2}{2n^2+n-2}$  &
		$3\leq n\leq 13$&
		$\lambda^{\pi_{\omega_3}}\;\;  3\leq n\leq 6$
		\\ [1.5ex]
		\hline 
%		\\ [-3ex]
		
		VIII & $\tfrac{\so(4n)}{\spp(1)\oplus \spp(n)}$ &
		$\frac{3}{4}+\frac{n+4}{4n(2n-1)}$ &&
		$\lambda^{\pi_{\omega_4}}\;\; 2\leq n\leq 7$
		\\ [1.5ex]
		\hline 
		\\ [-3ex]
		
		IX   & $\tfrac{\so\left(\frac{n(n-1)}{2}\right)}{\so(n)}$ &
		$\frac{1}{2}+\frac{4}{n^2-n-4}$ & 
		$7\leq n\leq 27$ &
		$\lambda^{\pi_{\omega_3}}\;\; 7\leq n\leq 10$
		\\ [1.5ex]
		\hline 
		\\ [-3ex]

		X    & $\tfrac{\so\left(\frac{(n-1)(n+2)}{2}\right)}{\so(n)}$ &
		$\frac{1}{2}+\frac{4n}{(n-2)(n+2)(n+3)}$&&
		$\lambda^{\pi_{3\omega_1}}\;\; 5\leq n\leq 8$
		\\ [0.5ex] 
		\hline
	\end{tabular}
	
	\

}
\end{table}

\subsection{Computational results for isotropy irreducible cases}\label{subsec4:computational}
In addition to the (almost uniform) lower bound in Theorem~\ref{thm4:main}, we present computational results for the exact value of $\lambda_1(G/H,g_{\st})$. 
We covered several cases within the families of simply connected non-symmetric strongly isotropy irreducible spaces where computations were feasible, as well as for most of the isolated cases. 
In some instances, the computations exceeded the computer's memory capacity. These computational values are shown in the column $\lambda_1(G/H,g_{\st})$ in Tables~ \ref{table4:isotropyirredexcepcion-classical},  \ref{table4:isotropyirredexcepcion-exceptional} and \ref{table4:isotropyirred-families-lambda1}.

We next describe the algorithm that the authors implemented in \cite{Sage} to obtain the computational results and we show the implementation with a concrete example.

\begin{algorithm}[The smallest Laplace eigenvalue of a standard manifold]\label{algorithm}
Input: a compact, connected and semisimple Lie group $G$ of rank $n$, $H$ a closed subgroup of $G$. 

Output: $\lambda_1(G/H,g_{\st})$

Steps:
\begin{enumerate}
\item Find $\pi_0\in\widehat G$ such that $V_{\pi_0}^H\neq0$. 
For instance, running over $\pi_\Lambda$ with $\Lambda=\sum_{i=1}^na_i\omega_i\in \mathcal{P}^+(G)$,  $a_1+\dots+a_n\leq 5$. 

\item Compute $\lambda^{\pi_0}$ using \eqref{eq3:Freudenthal}.

\item Determine the set $\mathcal A\subset \PP^+(G)$ such that $\lambda^{\pi_{\Lambda}}\leq \lambda^{\pi_0}$ and $\Lambda\neq0$.

\item Return $\lambda_1(G/H,g_{\st})=\min\{\lambda^{\pi_{\Lambda}} : \Lambda\in\mathcal A\}$. 
\end{enumerate}
\end{algorithm}

\begin{example}
We run Algorithm \ref{algorithm} for the Isolated case No.~10, that is, we compute $\lambda_1(\Spin(14)/\op{G}_2, g_{\st})$, where the embedding of $\op{G}_2$ into $\Spin(14)$ is given by $\sigma_{\eta_2}$.

	\begin{lstlisting}
sage: G=WeylCharacterRing("D7",style="coroots")
sage: H=WeylCharacterRing("G2",style="coroots")
sage: b=branching_rule(G,"G2(0,1,0,0)",rule="plethysm")
sage: fw=list(G.fundamental_weights())
sage: fwH=list(H.fundamental_weights())
sage: rho=sum(fw[i] for i in range(7))
sage: def eigenvalue(mu):
sage:     return (1/24)*sum(mu[i]*(mu[i]+2*rho[i]) for i in range(7))
sage: def simplex(m,K):
sage:     if m==1: return [[i] for i in range(K+1)]
sage:     sol=[]
sage:     for k in range(K+1):
sage:         for v in simplex(m-1,K-k):
sage:             sol.append(v+[k])
sage:     return sol    
sage: lower_bound=2
sage: minimal_weight=0
sage: for a in simplex(7,5):
sage:     Lambda=sum(a[i]*fw[i] for i in range(7))
sage:     if eigenvalue(Lambda)<lower_bound and eigenvalue(Lambda)>0:
sage:         tau=G(Lambda).branch(H,rule=b)
sage:         mono=tau.monomials()
sage:         if H(0*fwH[0]) in mono:
sage:             lower_bound=eigenvalue(Lambda)
sage:             minimal_weight=Lambda
sage: print("Minimum eigenvalue =", lower_bound)
sage: print("Minimal weight =", minimal_weight)
Minimum eigenvalue = 11/8
Minimal weight = (1, 1, 1, 0, 0, 0, 0)
	\end{lstlisting}
The result means that $\lambda_1(\operatorname{Spin}(14)/\operatorname{G}_2,g_{st})=\lambda^{\pi_{\omega_3}}=\frac{11}{8}$.

\end{example}

\section{Isotropy reducible spaces}\label{sec:isotred}

%\begin{comment}

\begin{sidewaystable}

\rule{0mm}{150mm}

\renewcommand{\arraystretch}{1.9} 

\caption{The 9 families of standard Einstein manifolds $(G/H,g_{\st})$ with $G$ simple and reducible isotropy representation}
\label{table5:isotropyred-families}

\centering
 \begin{tabular}{c c c c c c } 
 \hline 
 \hline 
No. & $\fg/\fh$  & Condition & $2E$ & H.stability &$\nu$-stability\\ [0.5ex]
\hline 
\hline 
XIa  & $\frac{\su(n)}{\R^{n-1}}$ 
	& $n\geq 3$ & $\frac{1}{2}+\frac{1}{n}$ & $G$-unstable & $\nu$-unstable \\ \hline 
XIb  & $\frac{\su(kn)}{k\su(n)\oplus(k-1)\R}$
	& $k\geq 3, n\geq 2$  & $\frac{1}{2}+\frac{1}{n}$ & $G$-unstable & $\nu$-unstable  \\ \hline 
XII  & $\frac{\su(l+pq)}{\su(l)\oplus\su(p)\oplus\su(q)\oplus \R}$
	& {\small $\begin{array}{c}	2\leq p\leq q \\[-8pt] lpq=p^2+q^2+1 \end{array}$ }
	&  $\frac{1}{2}+\frac{p^2+q^2}{(p^2+1)(q^2+1)}$ & $G$-unstable & $\nu$-unstable \\ \hline 
XIII & $\frac{\spp(kn)}{k\spp(n)}$
	& $k\geq 3, n\geq 1$  &  $\frac{1}{2}+\frac{2n+1}{2(kn+1)}$ & $G$-unstable & $\nu$-unstable  \\ \hline 
XIV  & $\frac{\spp(3n-1)}{\su(2n-1)\oplus\spp(n)\oplus\R}$
	& $n\geq 1$           & $\frac{5}{6}$             & $G$-unstable & $\nu$-unstable   \\ \hline 
XV   & $\frac{\so(4n^2)}{2\spp(n)}$ 
	& $n\geq 2$ & $\frac{1}{2}+\frac{2n+1}{2n(2n^2-1)}$ & 
	{\begin{tabular}{c}
	H.stable for $3\leq n\leq 9$,\\[-13pt]
	unknown otherwise
	\end{tabular}} & 
	{\begin{tabular}{c}
	$\nu$-stable for $3\leq n\leq 9$,\\[-13pt]
	unknown otherwise
	\end{tabular}} \\ \hline 
XVI  & $\frac{\so(n^2)}{2\so(n)}$
	& $n\geq 3$& $\frac{1}{2}+\frac{2n-2}{n(n^2-2)}$ &  
	{\small 
	\begin{tabular}{c}
	H.stable for $4\leq n\leq 16$, \\[-13pt] 
	unknown otherwise\\[-13pt]
	($G$-stable for $n\geq3$)
	\end{tabular} }& 
	{\begin{tabular}{c}
	$\nu$-stable for $4\leq n\leq 16$, \\[-13pt]
	unknown otherwise
	\end{tabular}} \\ \hline 
XVIIa & $\frac{\so(2n)}{\R^n}$
	& $n\geq 3$ & $\frac{1}{2}+\frac{1}{2(n-1)}$ &
	{\small \begin{tabular}{c}
	H.n.\ stable for $6\leq n\leq 7$, \\[-13pt]
	unknown otherwise\\[-13pt]
	($G$-neutrally stable for $n\geq3$)
	\end{tabular}} &
	{\begin{tabular}{c}
	n.\ $\nu$-stable for $6\leq n\leq 7$, \\[-13pt]
	unknown otherwise
	\end{tabular}} \\ \hline 
XVIIb & $\frac{\so(kn)}{k\so(n)}$ 
	& $k,n\geq 3$ & $\frac{1}{2}+\frac{n-1}{kn-2}$ & $G$-unstable & $\nu$-unstable \\ \hline 
XVIII & $\frac{\so(3n+2)}{\su(n+1)\oplus\so(n)\oplus\R}$
	& $n\geq 3$ & $\frac{5}{6}$ & $G$-unstable & $\nu$-unstable\\ \hline 
XIX  & $\frac{\so(n)}{\fh_1\oplus\cdots\oplus \fh_l}$
	& {$l\geq 1$, [9,\S7]} & $\frac{1}{2}+\frac{2\dim \fh_i}{(n-2)\dim\fp_i}$ & $G$-unstable & $\nu$-unstable \\ [1ex]
 \hline
 \end{tabular}
\end{sidewaystable}

\begin{sidewaystable}

\rule{0mm}{150mm}

\renewcommand{\arraystretch}{1.25} 

\centering 

\caption{Isolated standard Einstein manifolds $(G/H,g_{\st})$ with $G$ simple and reducible isotropy representation}
\label{table5:isotropyredexcepcions}
 \begin{tabular}{cc c c cccc} 
 \hline 
No.&  $\fg$   & $\fh$    & Embedding   & $2E$& H.stability &$\nu$-stability &\begin{tabular}{c}
lower bound \\[-6pt]
for $\lambda_1$
\end{tabular}\\ [0.5ex]
 \hline 
 \hline 
\rule{0pt}{2.5ex} 
34&$\so(8)$ &  $\mathfrak{g}_2$    & $\mathfrak{g}_2 \xhookrightarrow{\eta_1} \so(7)\subset \so(8)$          &  $\frac{5}{6}$  
	&$G$-unstable & $\nu$-unstable 
	&$\lambda^{\pi_{\omega_1}}=\frac{7}{12}$
	\\ [1ex] \hline 
\rule{0pt}{2.5ex} 
35&$\so(26)$ & {\tiny $\spp(1)\oplus\spp(5)\oplus\so(6)$} & $\spp(1)\oplus \spp(5) \xhookrightarrow{\eta_1+\eta_1'} \so(20) $             &  $\frac{29}{40}$   
	&$G$-unstable 
	& $\nu$-unstable  
	&$\lambda^{\pi_{2\omega_1}}=\frac{13}{12}$
	\\ [1ex]\hline 
\rule{0pt}{2.5ex} 
36&$\mathfrak{f}_4$ & $\so(8)$    & $\so(8)\subset\so(9) \overset{\text{max.}}{\subset} \ff_4$           & $\frac{8}{9}$  
	&$G$-unstable 
	& $\nu$-unstable  
	&$\lambda^{\pi_{\omega_4}}=\frac{2}{3}$
	\\[1ex] \hline  
\rule{0pt}{2.5ex} 
37&$\fe_6$ & $3\su(2)$   & 
	{\small \begin{tabular}{c}
		$\su(2)\xhookrightarrow{2\eta_1} \su(3)$,\\[-4pt]
		$3\su(3) \overset{\text{max.}}{\subset} \fe_6$
	\end{tabular}}  & 
	$\frac{5}{8}$  
	&H.semistable 
	& $\nu$-semistable 
	&  $\lambda^{\pi_{\omega_1+\omega_6}}=\frac{3}{2}$
	\\[1ex] \hline 
\rule{0pt}{2.5ex} 
38&$\fe_6$ & $\su(2)\oplus\so(6)$ & 
	{\small \begin{tabular}{c}
	$\so(6) \subset \su(6)$,\\[-2pt] 
	$\su(2)\oplus\su(6) \overset{\text{max.}}{\subset} \fe_6 $
	\end{tabular}} & $\frac{3}{4}$  
	&n.\ $G$-stable 
	& \begin{tabular}{c}
	unknown\\[-6pt]
	\tiny
	($\nu$-stable in \\[-8pt]
	\tiny conformal directions)
	\end{tabular}
	&$\lambda^{\pi_{\omega_3}}=\frac{25}{18}$
	\\ [1ex]\hline 
\rule{0pt}{2.5ex} 
39&$\fe_6$ & $\so(8)\oplus\R^2$  & 
	{\begin{tabular}{c}
	$\so(8)\oplus \R\subset\so(10)$,\\[-4pt]
	$\so(10)\oplus \R \overset{\text{max.}}{\subset} \fe_6$ 
	\end{tabular}}
	& $\frac{5}{6}$ 
	& $G$-unstable 
	&  $\nu$-unstable    
	&$\lambda^{\pi_{\omega_2}}=1$
	\\  [1ex]\hline 
\rule{0pt}{2.5ex} 
40&$\fe_6$ & $\R^6$  & \text{max.\ torus} & $\frac{7}{12}$ 
	& H.stable  
	& $\nu$-stable  
	&$\lambda^{\pi_{\omega_2}}=1$
	\\ [1ex]\hline 
\rule{0pt}{2.5ex} 
41&$\fe_7$ & $7\su(2)$   & 
	{\begin{tabular}{c}
	$3\so(4)\subset \so(12)$, \\[-3pt] 
	$\so(12)\oplus \su(2) \overset{\text{max.}}{\subset} \fe_7$ 
	\end{tabular}}
	& $\frac{2}{3}$ 
	& H.stable 
	& $\nu$-stable 
	&$\lambda^{\pi_{\omega_6}}=\frac{14}{9}$
	\\ [1ex]\hline 
\rule{0pt}{2.5ex} 
42&$\fe_7$ & $\so(8)$              & 
	$\so(8)\subset \su(8) \overset{\text{max.}}{\subset} \fe_7$ 
	& $\frac{13}{18}$  
	& H.stable 
	& $\nu$-stable  
	&$\lambda^{\pi_{\omega_2}}=\frac{35}{24}$
	 \\ [1ex]\hline 
\rule{0pt}{2.5ex} 
43&$\fe_7$ & $3\su(2)\oplus\so(8)$   & 
	{\begin{tabular}{c}
	$\so(8)\oplus\so(4)\subset \so(12)$, \\[-3 pt]
	$\so(12)\oplus \su(2) \overset{\text{max.}}{\subset}\fe_7$
	\end{tabular}} 
	& $\frac{7}{9}$   
	& $G$-unstable 
	& $\nu$-unstable 
	&$\lambda^{\pi_{\omega_6}}=\frac{14}{9}$
	\\ [1ex]\hline 
\rule{0pt}{2.5ex} 
44&$\fe_7$ & $\R^7$    & \text{max.\ torus} & $\frac{5}{9}$  
	&H.stable 
	& $\nu$-stable 
	&$\lambda^{\pi_{\omega_1}}=1$
	\\ [1ex]\hline 
\end{tabular}

\end{sidewaystable}

\addtocounter{table}{-1}

\begin{sidewaystable}

\rule{0mm}{150mm}

\renewcommand{\arraystretch}{1.25} 

\centering

\caption{Isolated standard Einstein manifolds $(G/H,g_{\st})$ with $G$ simple and reducible isotropy representation (continuation)}

 \begin{tabular}{cc c c cccc} 
 \hline 
No.&  $\fg$   & $\fh$    & Embedding   & $2E$& H.stability &$\nu$-stability &\begin{tabular}{c}
lower bound \\[-6pt]
for $\lambda_1$
\end{tabular}\\ [0.5ex]
 \hline 
 \hline 
\rule{0pt}{2.5ex} 
45&$\fe_8$ & $8\su(2)$ &
	$4\so(4)\subset\so(16)\overset{\text{max.}}{\subset}\fe_8$ 
	& $\frac{3}{5}$ 
	& H.stable 
	& $\nu$-stable
	&$\lambda^{\pi_{\omega_1}}=\frac{8}{5}$
	\\ [1ex]\hline 
\rule{0pt}{2.5ex} 
46&$\fe_8$ & $4\su(3)$   & 
	{\begin{tabular}{c}
	$3\su(3)\overset{\text{max.}}{\subset} \fe_6$, \\[-4pt] 
	$\fe_6\oplus\su(2)\overset{\text{max.}}{\subset}\fe_8$
	\end{tabular}} 
	& $\frac{19}{30}$  
	& H.stable 
	& $\nu$-stable 
	&$\lambda^{\pi_{\omega_1}}=\frac{8}{5}$
	\\ [1ex]\hline 
\rule{0pt}{2.5ex} 
47&$\fe_8$ & $4\su(2)$   & 
	{\begin{tabular}{c}
	$\su(2)\xhookrightarrow{2\eta_1} \su(3)$, \\[-4pt]
	$4\su(3)\subset \fe_8$ as above
	\end{tabular}} 
	& $\frac{11}{20}$ 
	& H.stable 
	& $\nu$-stable 
	&  $\lambda^{\pi_{\omega_1}}=\frac{8}{5}$
	\\ [1ex]\hline 
\rule{0pt}{2.5ex} 
48&$\fe_8$ & $2\su(3)$   & 
	{\begin{tabular}{c}
	$2\su(3)\xhookrightarrow{\eta_1+\eta_1'} \su(9)$,\\[-4pt] 
	$\su(9) \overset{\text{max.}}{\subset} \fe_8$
	\end{tabular}}  
	& $\frac{17}{30}$ 
	& H.stable 
	& $\nu$-stable 
	&$\lambda^{\pi_{\omega_1}}=\frac{8}{5}$
	\\ [1ex] \hline 
\rule{0pt}{2.5ex} 
49&$\fe_8$ & $2\su(5)$   & max.\ subalgebra  & $\frac{7}{10}$ 
	&H.stable 
	& $\nu$-stable 
	&$\lambda^{\pi_{\omega_1}}=\frac{8}{5}$
	\\ [1ex] \hline 
\rule{0pt}{2.5ex} 
50&$\fe_8$ & $\so(9)$   & $\so(9)\subset \su(9)\overset{\text{max.}}{\subset} \fe_8$  &  $\frac{13}{20}$ 
	&H.stable 
	& $\nu$-stable 
	&$\lambda^{\pi_{\omega_7}}=2$
	\\ [1ex] \hline 
\rule{0pt}{2.5ex} 
51&$\fe_8$ & $\so(9)$   & $\so(9)\xhookrightarrow{\eta_4} \so(16) \overset{\text{max.}}{\subset} \fe_8$  &  $\frac{13}{20}$ 
	&H.stable 
	& $\nu$-stable 
	&$\lambda^{\pi_{2\omega_8}}=\frac{31}{15}$
	\\ [1ex] \hline 
\rule{0pt}{2.5ex} 
52&$\fe_8$ & $2\so(8)$   & $2\so(8)\subset \so(16)\overset{\text{max.}}{\subset} \fe_8$ & $\frac{11}{15}$ 
	&H.stable 
	& $\nu$-stable
	&$\lambda^{\pi_{\omega_1}}=\frac{8}{5}$
	\\ [1ex] \hline 
\rule{0pt}{2.5ex} 
53&$\fe_8$ & $\so(5)$   & max.\ subalgebra  &  $\frac{13}{24}$ 
	&H.stable 
	& $\nu$-stable 
	&$\lambda^{\pi_{2\omega_8}}=\frac{31}{15}$
	\\ [1ex] \hline 
\rule{0pt}{2.5ex} 
54&$\fe_8$ & $2\spp(2)$   & $2\spp(2)\xhookrightarrow{\eta_1+\eta_1'} \so(16) \overset{\text{max.}}{\subset} \fe_8$ & $\frac{7}{12}$ 
	&H.stable 
	& $\nu$-stable 
	&$\lambda^{\pi_{\omega_1}}=\frac{8}{5}$\\ [1ex] \hline 
\rule{0pt}{2.5ex} 
55&$\fe_8$ & $\R^8$   & max.\ torus &  $\frac{8}{15}$  
	&H.stable 
	& $\nu$-stable 
	&$\lambda^{\pi_{\omega_8}}=1$
	\\ [1ex] \hline 
 \end{tabular}

\end{sidewaystable}

%\end{comment}

The homogeneous spaces $G/H$ with $G$ simple and reducible isotropy representation such that the standard metric $g_{\st}$ is Einstein were classified by Wang and Ziller~\cite{WangZiller85}. 
They consist of 12 infinite families (one of them is actually a conceptual construction) and 22 isolated examples (we have combined some families following \cite{Schwahn-Lichnerowicz} in such a way we have 9).

The $G$-stability of these spaces was studied in \cite{EJlauret-Stab}. 
Several cases were shown to be $G$-unstable, and consequently they are H.unstable and $\nu$-unstable. 
The goal of this section is to study the existence of $\nu$-unstable conformal directions of  them.
In other words, we want to determine whether $\lambda_1(G/H,g_{\st})> 2E$ holds or not.

\begin{remark}\label{rem5:Einsteinfactor}
There is a typo in \cite[Table 2]{Schwahn-Lichnerowicz} for the Einstein constant $E$ for Family XV. 
The correct value is $E=\frac{1}{4}+\frac{2n+1}{4n(2n^2-1)}$ as in Table~\ref{table5:isotropyred-families}. 
\end{remark}

\begin{remark}[Different presentations] \label{rem5:simplyconnected}
In \cite{WangZiller85}, the classification is given as pairs $(\fg,\fh)$ of Lie algebras (shown in Tables~\ref{table5:isotropyred-families}--\ref{table5:isotropyredexcepcions}) instead of simply connected effective realizations $G/H$ as in Section~\ref{sec:isotirred}. 
The embedding $(d\rho)_\C$ of $\fh$ into $\fg$ is also given in \cite{WangZiller85}, and it will be recalled in the cases where it is necessary. 

For each pair $(\fg,\fh)$, there may be several isotropy reducible spaces $G'/H'$ such that $\fg'=\fg$, $\fh'=\fh$ and $(G'/H',g_{\st})$ is Einstein. 
The Einstein factor $E$ is the same for all of them. 
We will always pick a realization with $G$ connected and $G/H$ simply connected, so $H$ is necessarily connected. 

Any other realization $G'/H'$ of $(\fg,\fh)$ satisfies that $(G/H,g_{\st})\to (G'/H',g_{\st})$ is a Riemannian covering, so 
$$
\lambda_1(G/H,g_{\st})\leq \lambda_1(G'/H',g_{\st}).
$$ 
Although the spaces covered by an $H$.unstable manifold are not necessarily $H$.unstable, one has that if $G/H$ is $G$-unstable then $\lambda_L(G'/H',g_{\st})<2E$ and therefore $(G'/H',g_{\st})$ is $\nu$-unstable for every presentation $G'/H'$ not necessarily simply connected.
We conclude that the $\nu$-stability type shown in Tables~\ref{table5:isotropyred-families}--\ref{table5:isotropyredexcepcions} holds for every presentation $G'/H'$. 
\end{remark}

The proof of Theorem~\ref{thm5:main} is given in Subsections~\ref{subsec5:flag}--\ref{subsec5:exceptional} as a case-by-case procedure using Proposition~\ref{prop:sufficient}.

\subsection{Full flag manifolds}\label{subsec5:flag}
Yamaguchi~\cite{Yamaguchi79} computed the first eigenvalue of the Laplace-Beltrami operator of all simply connected standard full flag manifolds $G/T_{\max}$ with $G$ simple. 
For those that are Einstein (Families XIa and XVIIa, and the isolated cases No.~39, 44, 55), he obtained that 
\begin{equation}\label{eq5:flag-lambda1}
\lambda_1(G/T_{\max},g_{\st})=\lambda^{\Ad_G}=1. 
\end{equation}

Indeed, in the notation of Remark~\ref{rem2:integralweightlattice}, for any compact connected simple Lie group $G$, one has that $\widehat G_{T_{\max}}$ corresponds via the Highest Weight Theorem (see Remark~\ref{rem2:HighestWeightTheorem}) to the dominant weights in the root lattice $\RR(\fg_\C)$. 
A simple inspection in Table~\ref{table:CasimirEigenvalues} returns that there is no $\Lambda\in \RR(\fg_\C)\cap \PP^+(\fg_\C)$ with $\lambda^{\pi_{\Lambda}}<1$ if $\fg_\C$ is of type $A$, $D$, or $E$ (this is no longer true for types $B$, $C$, $F_4$ and $G_2$). 
Moreover, the highest weight $\Lambda_{\Ad_G}$ of the adjoint representation $\Ad_G$ of $G$, which is irreducible because $G$ is simple, is always in the root lattice $\RR(\fg_\C)$, and then \eqref{eq5:flag-lambda1} follows.

\subsection{Symplectic cases for isotropy reducible spaces}\label{subsec5:Sp}
In this subsection we consider the cases $G/H$ from Tables~\ref{table5:isotropyred-families}--\ref{table5:isotropyredexcepcions} with $\fg=\spp(N)$ for some $N\in\N$, namely, Families XIII and XIV. 
We will prove that  $\lambda_1(G/H,g_{\st})=\lambda^{\pi_{\omega_2}}<1$ for them. 
Although we know in advance that these cases are $\nu$-unstable for being $G$-unstable, we will explicitly determine whether $\lambda_1(G/H,g_{\st})>2E$ in order to determine which cases have $\nu$-unstable conformal directions.

\smallskip
\noindent
$\bullet$
Family XIII: 
$\fg=\spp(kn)$ and $\fh\simeq \spp(n)\oplus\dots\oplus\spp(n)$ ($k$-times) for $k\geq3$ and $n\geq1$. 

Set $H'=\Sp(n)\times\dots\times\Sp(n)$ ($k$-times) and 
\begin{equation}\label{eq5:rhoSp(nk)}
\rho=
{\textstyle \bigoplus\limits_{i=1}^k}\,
	\sigma_{0}^{(1)}
	\widehat\otimes\,\dots\widehat\otimes \,
	\sigma_{0}^{(i-1)} 
	\widehat\otimes\,
	\sigma_{\eta_1^{(i)}}^{(i)}
	\widehat\otimes\,
	\sigma_{0}^{(i+1)} 
	\widehat\otimes\,\dots\widehat\otimes\,
	\sigma_{0}^{(k)}
: H'\to\GL\Big({\textstyle \bigoplus\limits_{i=1}^k}\, \C^{2n}\Big)
. 
\end{equation}
Here, the superscript index $(i)$ means the $i$-th component. 
Each term of $\rho$ is clearly of quaternionic type, so is $\rho$ and the image of $\rho$ is included in $G_\rho=\Sp(nk)$. 
Although we do not have the effective realization $G/H$, we know that $G$ is a compact Lie group with Lie algebra $\spp(nk)$ and $H=\rho_0(H')\simeq H'/\op{Ker}(\rho_0)$, where $\rho_0:H'\to G$ is the only Lie group homomorphism with $d\rho_0=d\rho$.

We have that $V_{\pi_{\omega_1}}^{\fh_\C} =V_{\pi_{\omega_1}}^{(d\rho)_\C(\fh_\C')} = V_\rho^{\fh_\C'}$ by Lemma~\ref{lem3:standard}. 
Since $(d\rho)_{\C}$ does not contain the trivial representation of $\fh_\C'$, we conclude that  $V_{\pi_{\omega_1}}^{\fh_\C}=0$.

\begin{claim}\label{claim5:Sp(nk)-dimV_omega2^h=k-1}
$\dim V_{\omega_2}^H=k-1$.
\end{claim}

\begin{proof}
\renewcommand{\qedsymbol}{$\square$}
The proof is very similar to the one given in Family IV in Subsection~\ref{subsec4:Sp}, so we will omit several details. 
It follows from \eqref{eq2:Lambda^2(VxW)-Sym^2(VxW)} that the $2$-exterior power of the $i$-th term in \eqref{eq5:rhoSp(nk)} is
\begin{multline}
{\textstyle \bigwedge^2}\big(
	\sigma_{0}^{(1)}
	\widehat\otimes\,\dots\widehat\otimes\, 
	\sigma_{\eta_1^{(i)}}^{(i)}
	\widehat\otimes\,\dots\widehat\otimes\,
	\sigma_{0}^{(k)}
\big)
=	\sigma_{0}^{(1)}
	\widehat\otimes\,\dots\widehat\otimes \,
	{\textstyle \bigwedge^2}(\sigma_{\eta_1^{(i)}}^{(i)})
	\widehat\otimes\,\dots\widehat\otimes\,
	\sigma_{0}^{(k)}
\\
=	\sigma_{0}^{(1)}
	\widehat\otimes\,\dots\widehat\otimes \,
	\big(\sigma_{\eta_2^{(i)}}^{(i)}\oplus \sigma_{0}^{(i)}\big)
	\widehat\otimes\,\dots\widehat\otimes\,
	\sigma_{0}^{(k)}
= 1_{\fh_\C'} 
\oplus \,
	\sigma_{0}^{(1)}
	\widehat\otimes\,\dots\widehat\otimes \,
	\sigma_{\eta_2^{(i)}}^{(i)}
	\widehat\otimes\,\dots\widehat\otimes\,
	\sigma_{0}^{(k)}
.
\end{multline}
Now, \eqref{eq2:Lambda^2(V+W)-Sym^2(V+W)} gives 
\begin{equation}
\begin{aligned}
{\textstyle\bigwedge^2}((d\rho)_\C)&
=k\cdot1_{\fh_\C'} \oplus  
{\textstyle \bigoplus\limits_{i=1}^k}\,
	\sigma_{0}^{(1)}
	\widehat\otimes\,\dots\widehat\otimes \,
	\sigma_{\eta_2^{(i)}}^{(i)}
	\widehat\otimes\,\dots\widehat\otimes\,
	\sigma_{0}^{(k)}
\\ & \quad 
\oplus
{\textstyle \bigoplus\limits_{1\leq i<j\leq k}}\,
	\sigma_{0}^{(1)}
	\widehat\otimes\,\dots\widehat\otimes \,
	\sigma_{\eta_1^{(i)}}^{(i)} 
	\widehat\otimes\,\dots\widehat\otimes \,
	\sigma_{\eta_1^{(j)}}^{(j)}
	\widehat\otimes\,\dots\widehat\otimes \,
	\sigma_{0}^{(k)}
.
\end{aligned}
\end{equation}
This tells us that $\big({\textstyle \bigwedge^2}(\pi_{\omega_1})\big)|_{\fh_\C'}$ contains $k$-times the trivial representation $1_{\fh_\C'}$. 
Since $\bigwedge^2\pi_{\omega_1}=\pi_{\omega_2}\oplus1_{\fg_\C}$, we conclude that $\pi_{\omega_2}|_{\fh_\C'}$ contains $(k-1)$-copies of $1_{\fh_\C'}$, which proves the assertion. 
\end{proof}

Note that $G$ is isomorphic to $\Sp(nk)$ or $\Sp(nk)/\Z_2$, and $\omega_2\in\PP^+(G)$ for both cases. 
From Theorem~\ref{thm:Spec(standard)}, combining Table~\ref{table:CasimirEigenvalues} and the facts $V_{\pi_{\omega_1}}^H=0$ and $V_{\pi_{\omega_2}}^H\neq0$ by Claim~\ref{claim5:Sp(nk)-dimV_omega2^h=k-1}, we conclude that 
\begin{equation*}
\lambda_1(G/H,g_{\st})=\lambda^{\pi_{\omega_2}}=\frac{nk}{nk+1}<1. 
\end{equation*}

Since $2E=\frac12+\frac{2n+1}{2(nk+1)}$, it is a simple matter to check that $\lambda_1(G/H,g_{\st})>2E$ if and only if $n(k-2)>2$, which always holds excepting the case $n=1$ and $k=3$. In particular, 
\begin{equation}
\lambda_1\left(G/H ,g_{\st}\right)=\frac{3}{4}
<
\frac78=2E
\end{equation}
when $\fg\simeq\spp(3)$ and $\fh\simeq \spp(1)\oplus\spp(1)\oplus\spp(1)$. 
This space has $\nu$-unstable conformal directions.

\smallskip
\noindent
$\bullet$
Family XIV: 
$\fg=\spp(3n-1)$ and $\fh\simeq \spp(n)\oplus\ut(2n-1)$ for $n\geq1$. 

Set $H'=\Sp(n)\times\Ut(2n-1)$ and 
\begin{equation}\label{eq5:rhoSp(3n-1)}
\rho=
	\sigma_{\eta_1}
	\widehat\otimes
	\sigma_{0}'
	\,\oplus\, 
	\sigma_{0}
	\widehat\otimes
	\big(\sigma_{\eta_1'}' \oplus \sigma_{\eta_{2n-2}'}'\big)
: H'\to\GL\left(\C^{2n}\oplus \C^{2(2n-1)} \right)
. 
\end{equation}
Note that $(\sigma_{\eta_1'}')^*\simeq\sigma_{\eta_{2n-2}'}'$, thus $\sigma_{\eta_1'}' \oplus \sigma_{\eta_{2n-2}'}'$ has a quaternionic structure (see Remark~\ref{rem2:types-of-representations}), as well as $\sigma_{\eta_1} \widehat\otimes \sigma_{0}'$. 
Thus $\rho(H')\subset \Sp(3n-1)$. 

We have that $V_{\pi_{\omega_1}}^{\fh_\C} =V_{\pi_{\omega_1}}^{(d\rho)_\C(\fh_\C')} = V_\rho^{\fh_\C'}$ by Lemma~\ref{lem3:standard}. 
Since $(d\rho)_{\C}$ does not contain the trivial representation of $\fh_\C'$, we conclude that  $V_{\pi_{\omega_1}}^{\fh_\C}=0$.

\begin{claim}\label{claim5:Sp(3n-1)-dimV_omega2^h=k-1}
$\dim V_{\omega_2}^H=1$.
\end{claim}

\begin{proof}
\renewcommand{\qedsymbol}{$\square$}
It follows from \eqref{eq2:Lambda^2(V+W)-Sym^2(V+W)} and \eqref{eq2:Lambda^2(VxW)-Sym^2(VxW)} that
\begin{equation}
\begin{aligned}
{\textstyle\bigwedge^2}\big((d\rho)_\C\big)&
= \sigma_{\eta_1}
	\widehat\otimes
	\big(\sigma_{\eta_1'}' \oplus \sigma_{\eta_{2n-2}'}'\big)
	\,\oplus\, 
{\textstyle\bigwedge^2}\big(  
	\sigma_{\eta_1}
	\widehat\otimes
	\sigma_{0}'
\big)
	\,\oplus\, 
{\textstyle\bigwedge^2}\big(  
	\sigma_{0}
	\widehat\otimes
	\big(\sigma_{\eta_1'}' \oplus \sigma_{\eta_{2n-2}'}'\big)
\big)
\\ & 
=
	\sigma_{\eta_1}
	\widehat\otimes
	\big(\sigma_{\eta_1'}' \oplus \sigma_{\eta_{2n-2}'}'\big)
	\,\oplus\, 
	\big(
		\sigma_{\eta_2}
		\oplus 
		\sigma_0
	\big)
	\widehat\otimes
	\sigma_{0}'
	\,\oplus\, 
\\ & \quad \oplus 
	\sigma_{\eta_2}
	\widehat\otimes
	\big(
		\sigma_{0}'
		\oplus
		\sigma_{\eta_2'}'
		\oplus
		\sigma_{\eta_{2n-3}'}'
		\oplus
		\sigma_{\eta_{1}'}'\widehat{\otimes} \sigma_{\eta_{2n-2}'}'
	\big)
.
\end{aligned}
\end{equation}
Since $\sigma_{\eta_{1}'}'\otimes \sigma_{\eta_{2n-2}'}'\simeq \sigma_0'\oplus \sigma_{\eta_1'+\eta_{2n-2}'}'$, $1_{\fh_\C'}$ occurs exactly twice in ${\textstyle\bigwedge^2}\big((d\rho)_\C\big)$. 
The assertion follows from $\bigwedge^2\pi_{\omega_1}|_{\fh_\C'}=\pi_{\omega_2}|_{\fh_\C'} \oplus1_{\fh_\C'}$. 
\end{proof}

Similarly as in the previous family, we conclude that 
$
\lambda_1(G/H,g_{\st})=\lambda^{\pi_{\omega_2}} =\frac{3n-1}{3n}<1. 
$
Since $2E=\frac{5}{6}$, $\lambda_1(G/H,g_{\st})>2E$ if and only if $n>2$.
In particular, only the cases $n=1,2$ have $\nu$-unstable conformal directions.

\subsection{Special unitary cases for isotropy reducible spaces}\label{subsec5:SU}
In this subsection we prove that  $\lambda_1(G/H,g_{\st})\geq1$ for all cases $G/H$ from Tables~\ref{table5:isotropyred-families}--\ref{table5:isotropyredexcepcions} with $\fg=\su(N)$ for some $N\in\N$, namely, Families XI and XII. 
The result was established for Family XIa in Subsection~\ref{subsec5:flag}, so we will omit it here. 

\smallskip
\noindent
$\bullet$
Family XIb: 
$\fg=\su(kn)$ and $\fh\simeq \mathfrak{s}(\ut(n)\oplus\dots\oplus\ut(n)\big)$ ($k$-times) for $k\geq3$ and $n\geq2$. 

Set $H'=\textup{S}(\Ut(n)\times\dots\times\Ut(n))$ ($k$-times) and 
\begin{equation}\label{eq5:rhoSU(nk)}
\rho=
{\textstyle \bigoplus\limits_{i=1}^k}\,
	\sigma_{0}^{(1)}
	\widehat\otimes\,\dots\widehat\otimes \,
	\sigma_{0}^{(i-1)} 
	\widehat\otimes\,
	\sigma_{\eta_1^{(i)}}^{(i)}
	\widehat\otimes\,
	\sigma_{0}^{(i+1)} 
	\widehat\otimes\,\dots\widehat\otimes\,
	\sigma_{0}^{(k)}
: H'\to\GL\Big({\textstyle \bigoplus\limits_{i=1}^k}\, \C^{n}\Big)
. 
\end{equation}
The image of $\rho$ is included in $G_\rho=\SU(nk)$. 
Although we do not have the effective realization $G/H$, we know that $G$ is a compact Lie group with Lie algebra $\su(nk)$ and $H=\rho_0(H')\simeq H'/\op{Ker}(\rho_0)$, where $\rho_0:H'\to G$ is the only Lie group homomorphism with $d\rho_0=d\rho$.

\begin{lemma}\label{lem5:V_omega_p^T_max=0}
Let $T_{\max}$ be a maximal torus of $\SU(N)$. 
Then $\dim V_{\pi_{\omega_p}}^{T_{\max}}=0$ for all $1\leq p\leq N-1$. 
\end{lemma}

\begin{proof}
We have that $V_{\pi_{\omega_p}}^{T_{\max}}=V_{\pi_{\omega_p}}(0)$, the weight space of $\pi_{\omega_p}$ associated to the weight $0$. 
Consequently, it is enough to show that $0\notin \PP(\pi_{\omega_p})$. This follows since the set of weights of $\pi_{\omega_p}$ are  $\pr(\ee_{i_1}+\ee_{i_2}+\dots+\ee_{i_p})$ (see Notation~\ref{notation2:Bourbaki}) for $1\leq i_1<i_2<\dots<i_p\leq N$, all of them with multiplicity one. 
This can be easily seen from $V_{\pi_{\omega_p}} \simeq{\textstyle \bigwedge^p} (\C^N)$ and the fact that $e_{i_1}\wedge\dots\wedge e_{i_p}$ is a weight vector with weight $\pr(\ee_{i_1}+\ee_{i_2}+\dots+\ee_{i_p})$.
\end{proof}

We set $N=kn$. 
The maximal torus $T_{\max}$ of $\SU(N)$ given by diagonal elements is also a maximal torus of $H$, that is, $H\cap T_{\max}=T_{\max}$. 
The fact $T_{\max}\subset H$ and Lemma~\ref{lem5:V_omega_p^T_max=0} yield 
\begin{equation}\label{eq5:V_omega_p^H=0}
V_{\pi_{\omega_p}}^H
\subset V_{\pi_{\omega_p}}^{T_{\max}}=0
\quad\forall \, 1\leq p\leq N-1. 
\end{equation}
We conclude by  Proposition~\ref{prop:sufficient}\eqref{item3:su-N!=6,7}--\eqref{item3:su-N==6,7} that $\lambda_1(G/H,g_{\st})\geq1$.

\smallskip
\noindent
$\bullet$
Family XII: 
$\fg=\su(pq+l)$ and $\fh\simeq \su(p)\oplus\su(q)\oplus \su(l)\oplus\R$ for positive integers $p,q,l$ satisfying $q\geq p\geq2$ and $lpq=p^2+q^2+1$.

According to Example 8 (page 576) in \cite{WangZiller85}, we define $H'=\textup{S}(\Ut(p)\times\Ut(q)\times\Ut(l))$ and 
\begin{equation}\label{eq5:rhoSU(pq+l)}
\rho=
	\sigma_{\eta_1}
	\widehat\otimes
	\sigma_{\eta_1}'
	\widehat\otimes
	\sigma_{0}''
	\,\oplus\, 
	\sigma_{0}
	\widehat\otimes
	\sigma_0'
	\widehat\otimes
	\sigma_{\eta_1''}''
: H'\to\GL\left(\C^{p}\otimes\C^q\oplus \C^{l} \right)
. 
\end{equation}
Note that 
$(d\rho)_\C$ is not injective since its first term above has a one-dimensional kernel, so $\fh\simeq \fh'/\op{Ker}((d\rho)_\C)$ and $H\simeq H'/\op{Ker}(\rho_0)$, where $\rho_0:H'\to G$ is the only Lie group homomorphism with $d\rho_0=d\rho$.

One can easily check that the solution $(p,q,l)$ with $q$ as small as possible is $(p,q,l)=(2,5,3)$. 
Hence $N:=pq+l\geq 5\cdot 2+l>10$. 
By Proposition~\ref{prop:sufficient}\eqref{item3:su-N!=6,7}, it is sufficient to show that $V_{\pi_{\omega_i}}^{\fh_\C}=0$ for $i=1,2$. 

We have that $V_{\pi_{\omega_1}}^{\fh_\C}=0$ since $\pi_{\omega_1}$ is the standard representation of $\fg_\C$ and the one-dimensional kernel of $(d\rho)_\C$ is ruled out because $\fh\simeq \fh'/\op{Ker}((d\rho)_\C)$.

\begin{claim}\label{claim5:SU(pq+l)-dimV_omega2^h=0}
$\dim V_{\omega_2}^H=0$.
\end{claim}

\begin{proof}
\renewcommand{\qedsymbol}{$\square$}
It is enough to show that $V_{\pi_{\omega_2}}^{\fh_\C'}=0$ since $V_{\pi_{\omega_2}}^H =V_{\pi_{\omega_2}}^{\fh_\C} \subset V_{\pi_{\omega_2}}^{\fh_\C'}$. 

We have that ${\textstyle\bigwedge^2}(\pi_{\omega_1})\simeq \pi_{\omega_2}$. 
Thus $\pi_{\omega_2}|_{\fh_\C'}\simeq {\textstyle\bigwedge^2}(\pi_{\omega_1}|_{\fh_\C'})
= {\textstyle\bigwedge^2}( (d\rho)_\C )$. 
Hence, it is sufficient to show that ${\textstyle\bigwedge^2}( (d\rho)_\C )$ does not contain the trivial representation.
It follows from \eqref{eq2:Lambda^2(V+W)-Sym^2(V+W)} and \eqref{eq2:Lambda^2(VxW)-Sym^2(VxW)} that
\begin{equation}
\begin{aligned}
{\textstyle\bigwedge^2}\big((d\rho)_\C\big)&
= 
{\textstyle\bigwedge^2}\big(
	\sigma_{\eta_1}
	\widehat\otimes
	\sigma_{\eta_1}'
	\widehat\otimes
	\sigma_{0}''
\big)
\,\oplus\, 
{\textstyle\bigwedge^2}\big(
	\sigma_{0}
	\widehat\otimes
	\sigma_0'
	\widehat\otimes
	\sigma_{\eta_1''}''\big)
\,\oplus\, 
\big(
	\sigma_{\eta_1}
	\widehat\otimes
	\sigma_{\eta_1}'
	\widehat\otimes
	\sigma_{0}''
\big)
\otimes 
\big(
	\sigma_{0}
	\widehat\otimes
	\sigma_0'
	\widehat\otimes
	\sigma_{\eta_1''}''
\big)
\\ & 
=
{\textstyle\bigwedge^2}\big(
	\sigma_{\eta_1}
	\widehat\otimes
	\sigma_{\eta_1}'
\big)
	\widehat\otimes
	\sigma_{0}''
	\,\oplus\, 
	\sigma_{0}
	\widehat\otimes
	\sigma_0'
	\widehat\otimes
{\textstyle\bigwedge^2}\big(
	\sigma_{\eta_1''}''\big)
	\,\oplus\, 
	\sigma_{\eta_1}
	\widehat\otimes
	\sigma_{\eta_1}'
	\widehat\otimes
	\sigma_{\eta_1}''
\\ & 
= 
\big(
	\sigma_{\eta_2}
	\widehat\otimes
	\sigma_{2\eta_1'}'
	\oplus
	\sigma_{2\eta_1}
	\widehat\otimes
	\sigma_{\eta_2'}'
\big)
	\widehat\otimes
	\sigma_{0}''
	\,\oplus\, 
	\sigma_{0}
	\widehat\otimes
	\sigma_0'
	\widehat\otimes
	\sigma_{\eta_2''}''
	\,\oplus\, 
	\sigma_{\eta_1}
	\widehat\otimes
	\sigma_{\eta_1}'
	\widehat\otimes
	\sigma_{\eta_1}''
,
\end{aligned}
\end{equation}
as required. 
\end{proof}

We conclude by  Proposition~\ref{prop:sufficient}\eqref{item3:su-N!=6,7} that $\lambda_1(G/H,g_{\st})\geq 1$.

\subsection{Orthogonal cases for isotropy reducible spaces}\label{subsec5:SO}
In this subsection, we prove that  $\lambda_1(G/H,g_{\st})\geq1$ for all cases $G/H$ from Tables~\ref{table5:isotropyred-families}--\ref{table5:isotropyredexcepcions} with $\fg=\so(N)$ for some $N\in\N$ except for one, namely, Families XV--XIX, and isolated cases No.~34 and 35.  
The result was already established for Family XVIIa in Subsection~\ref{subsec5:flag}, so we will omit it here.

\smallskip
\noindent
$\bullet$
Family XV: 
$\fg=\so(4n^2)$ and $\fh\simeq \spp(n)\oplus\spp(n)$ for $n\geq 2$.  

Set $H'=\Sp(n)\times\Sp(n)$ and $\rho=\sigma_{\eta_1}\widehat\otimes \,\sigma_{\eta_1'}': H'\to\GL(\C^{2n}\otimes \,\C^{2n})$.
We have that $\rho$ is of real type since the standard representations $\sigma_{\eta_1}$ and $\sigma_{\eta_1'}'$ of the factors $\Sp(n)$ are both of quaternionic type. 
Although we do not have the information of $G/H$, we know that $G$ is a compact Lie group with Lie algebra $\so(4n^2)$ and $H=\rho_0(H')\simeq H'/\op{Ker}(\rho_0)$, where $\rho_0:H'\to G$ is the only Lie group homomorphism with $d\rho_0=d\rho$. 

Since $n\geq2$, we have that $N=4n^2\geq 16$. 
By Proposition~\ref{prop:sufficient}\eqref{item3:SO}, it is sufficient to show that $V_{\pi_{\omega_1}}^{H} = V_{\pi_{\omega_1}}^{\fh_\C}=0$. 
Lemma~\ref{lem3:standard} yields $V_{\pi_{\omega_1}}^{\fh_\C}= V_\rho^{\fh_\C'}$, which is $0$ since $1_{\fh_\C'}$ does not occur in the decomposition of $(d\rho)_\C$ in irreducible factors.  

We conclude that $\lambda_1(G/H,g_{\st})\geq1$ in all cases. 

\smallskip
\noindent
$\bullet$
Family XVI: 
$\fg=\so(n^2)$ and $\fh\simeq \so(n)\oplus\so(n)$ for $n\geq 3$.  

Set $H'=\SO(n)\times\SO(n)$ and $\rho=\st_{\SO(n)} \widehat\otimes \,\st_{\SO(n)}': H'\to\GL(\C^{n}\otimes \,\C^{n})$.
Clearly $\rho$ is of real type. 
Again, we do not know the realization $G/H$, though we know that $G$ is a compact Lie group with Lie algebra $\so(n^2)$ and $H=\rho_0(H')\simeq H'/\op{Ker}(\rho_0)$, where $\rho_0:H'\to G$ is the only Lie group homomorphism with $d\rho_0=d\rho$. 

For $n\geq4$, we have that $N=n^2\geq 16$. 
By Proposition~\ref{prop:sufficient}\eqref{item3:so-N>=15}, it is sufficient to show that $V_{\pi_{\omega_1}}^{H} = V_{\pi_{\omega_1}}^{\fh_\C}=0$, which follows from Lemma~\ref{lem3:standard} as above. 

For $n=3$, we have $N=9$. 
By Proposition~\ref{prop:sufficient}\eqref{item3:so-N<15}, it is sufficient to show that $V_{\pi_{\omega_1}}^{H} =0$ and $V_{\pi_{\omega_{4}}}^{H}=0$. 
The first one follows from Lemma~\ref{lem3:standard}, and the second one from Lemma~\ref{lem3:dimV_pi^H} since $\pi_{\omega_{4}}|_{\fh_\C'}\simeq \sigma_{\eta_1} \widehat{\otimes} \sigma'_{3\eta_1'}\oplus \sigma_{3\eta_1}\widehat{\otimes}\sigma'_{\eta'_1}$ (see e.g.\ \cite[pp. 224]{LieART}; $\dim V_{\pi_{\omega_{4}}}=16$).

We conclude that $\lambda_1(G/H,g_{\st})\geq1$ in all cases. 

\smallskip
\noindent
$\bullet$
Family XVIIb: 
$\fg=\so(kn)$ and $\fh\simeq \so(n)\oplus\dots\oplus\so(n)$ ($k$-times) for $k\geq3$ and $n\geq3$. 

Set $H'=\SO(n)\times\dots\times\SO(n)$ ($k$-times) and 
\begin{equation}\label{eq5:rhoSO(nk)}
\rho=
{\textstyle \bigoplus\limits_{i=1}^k}\,
	\sigma_{0}^{(1)}
	\widehat\otimes\,\dots\widehat\otimes \,
	\sigma_{0}^{(i-1)} 
	\widehat\otimes\,
	\st_{\SO(n)}^{(i)}
	\widehat\otimes\,
	\sigma_{0}^{(i+1)} 
	\widehat\otimes\,\dots\widehat\otimes\,
	\sigma_{0}^{(k)}
: H'\to\GL\Big({\textstyle \bigoplus\limits_{i=1}^k}\, \C^{n}\Big)
. 
\end{equation}
Each term of $\rho$ is of real type, so is $\rho$ and $G_\rho=\SO(kn)$. 

We set $N=kn$. 
If $N\geq15$, Proposition~\ref{prop:sufficient}\eqref{item3:so-N>=15} yields that it is sufficient to show that $V_{\pi_{\omega_1}}^{\fh_\C}=0$, which follows immediately from Lemma~\ref{lem3:standard}. 

The only cases satisfying $N\leq 14$ are $(n,k)=(3,3), (3,4), (4,3)$. 
By Proposition~\ref{prop:sufficient}\eqref{item3:so-N<15}, it is enough to show that $V_{\pi_{\omega_1}}^{\fh_\C'}=0$ (which holds by Lemma~\ref{lem3:standard}) and  $V_{\pi_{\omega_{\lfloor {N}/{2}\rfloor}}}^{\fh_\C'}=0$. 
We have that 
$
\pi_{\omega_{4}}|_{\fh_\C'}= 2(\sigma_{\eta_1^{(1)}}^{(1)}\widehat{\otimes}\sigma_{\eta_1^{(2)}}^{(2)}\widehat{\otimes}\sigma_{\eta_1^{(3)}}^{(3)})
$ 
for $(n,k)=(3,3)$ where $\fh_\C'\simeq \su(2)_\C\oplus\su(2)_\C\oplus\su(2)_\C$
(see e.g.\ \cite[pp.\ 219, 152]{LieART}; $\dim V_{\pi_{\omega_{4}}|_{\fh_\C'}}=16$),
$
{\pi_{\omega_{6}}}|_{\fh_\C'} =
2(\sigma_{\eta_1^{(1)}}^{(1)}\widehat{\otimes}		\sigma_{\eta_1^{(2)}}^{(2)}\widehat{\otimes}
\sigma_{\eta_1^{(3)}}^{(3)}\widehat{\otimes}		\sigma_{\eta_1^{(4)}}^{(4)})
$ 
for $(n,k)=(3,4)$ where $\fh_\C'\simeq \su(2)_\C\oplus\su(2)_\C\oplus\su(2)_\C\oplus\su(2)_\C$
(see e.g.\ \cite[pp.\ 239, 152]{LieART}; $\dim V_{\pi_{\omega_{4}}|_{\fh_\C'}}=32$), and 
$
{\pi_{\omega_{6}}}|_{\fh_\C'} = (\sigma_{\eta_1^{(1)}}^{(1)}\widehat{\otimes} \sigma_{\eta_1^{(2)}}^{(2)}\widehat{\otimes}\sigma_{\eta_2^{(3)}}^{(3)})\oplus (\sigma_{\eta_1^{(1)}}^{(1)}\widehat{\otimes} \sigma_{\eta_2^{(2)}}^{(2)}\widehat{\otimes}\sigma_{\eta_1^{(3)}}^{(3)})\oplus(\sigma_{\eta_2^{(1)}}^{(1)}\widehat{\otimes} \sigma_{\eta_1^{(2)}}^{(2)}\widehat{\otimes}\sigma_{\eta_1^{(3)}}^{(3)})\oplus (\sigma_{\eta_2^{(1)}}^{(1)}\widehat{\otimes}\sigma_{\eta_2^{(2)}}^{(2)}\widehat{\otimes} \sigma_{\eta_2^{(3)}}^{(3)})$ 
for $(n,k)=(4,3)$ where $\fh_\C'= \so(4)_\C\oplus\so(4)_\C\oplus\so(4)_\C \simeq \su(2)_\C\oplus\su(2)_\C\oplus\su(2)_\C\oplus\su(2)_\C\oplus\su(2)_\C\oplus\su(2)_\C$
(see e.g.\ \cite[pp.\ 238, 197]{LieART}; $\dim V_{\pi_{\omega_{4}}|_{\fh_\C'}}=32$).
Since the trivial representation does not appear, Lemma~\ref{lem3:dimV_pi^H} gives $V_{\pi_{\omega_{\lfloor {N}/{2}\rfloor}}}^{\fh_\C}=0$ in the three cases as required. 

We conclude that $\lambda_1(G/H,g_{\st})\geq1$ in all cases.

\smallskip
\noindent
$\bullet$
Family XVIII: 
$\fg=\so(3n+2)$ and $\fh\simeq \so(n)\oplus \ut(n+1)$ for $n\geq3$. 

Set $H'=\SO(n)\times\Ut(n+1)$ and 
\begin{equation}\label{eq5:rhoSO(3n+2)}
\rho=
	\st_{\SO(n)}
	\widehat\otimes
	\sigma_{0}'
\,\oplus\,
	\sigma_{0}
	\widehat\otimes
\big(
	\sigma_{\eta_1'}'\oplus \sigma_{\eta_n'}'
\big)
: H'\to\GL\Big({\textstyle \bigoplus\limits_{i=1}^k}\, \C^{n}\Big)
. 
\end{equation}
Since $(\sigma_{\eta_1'}')^*\simeq \sigma_{\eta_n'}'$, we have that $\sigma_{\eta_1'}'\oplus \sigma_{\eta_n'}'$ admits a real structure (see Remark~\ref{rem2:types-of-representations}).
Consequently $\rho$ is of real type and $G_\rho=\SO(3n+2)$.

Lemma~\ref{lem3:standard} implies that  $V_{\pi_{\omega_1}}^{H}=V_{\pi_{\omega_1}}^{\fh_\C}= V_{\rho}^{\fh_\C'}$, which vanishes since $(d\rho)_\C$ does not contain $1_{\fh_\C'}$ in its irreducible decomposition. 

We set $N=3n+2$. 
For $n\geq5$, we have that $N\geq 17$, and $\lambda_1(G/H,g_{\st})\geq1$ follows from  Proposition~\ref{prop:sufficient}\eqref{item3:so-N>=15}.

For $n=4$, $N=14$, from Proposition~\ref{prop:sufficient}\eqref{item3:so-N<15} it remains to show that $V_{\pi_{\omega_7}}^H=0$. 
We have that (see e.g.\ \cite[pp.\ 226, 252]{LieART}; $\dim V_{\pi_{\omega_7}}=64$)
$
\pi_{\omega_7}|_{\fh_\C'}
\simeq (\sigma_{\eta_1}\widehat{\otimes}\sigma'_0\widehat{\otimes}\tau_{-5})\oplus (\sigma_{\eta_1}\widehat{\otimes} \sigma_{\eta'_4}'\widehat{\otimes}\tau_3) \oplus  (\sigma_{\eta_1}\widehat{\otimes}\sigma'_{\eta'_2}\widehat{\otimes}\tau_{-1})
\oplus (\sigma_{\eta_2}\widehat{\otimes}\sigma'_0\widehat{\otimes}\tau_5) \oplus  (\sigma_{\eta_2}\widehat{\otimes}\sigma'_{\eta'_1}\widehat{\otimes}\tau_{-3})\oplus (\sigma_{\eta_2}\widehat{\otimes}\sigma'_{\eta'_3}\widehat{\otimes}\tau_1),
$
where $\fh'_\C$ is being identified with $\so(4)\oplus \su(5)\oplus \mathfrak{u}(1)$. 
Therefore $V_{\pi_{\omega_7}}^H=0$ by Lemma~\ref{lem3:dimV_pi^H}. 

For $n=3$, $N=11$, from Proposition~\ref{prop:sufficient}\eqref{item3:so-N<15} it remains to show that $V_{\pi_{\omega_5}}^H=0$. 
We have that (see e.g.\ \cite[pp.~204, 233]{LieART}; $\dim V_{\pi_{\omega_7}}=32$)
$
\pi_{\omega_5}|_{\fh_\C'}\simeq (\sigma_{\eta_1}\widehat{\otimes}\sigma'_0\widehat{\otimes}\tau_2)\oplus (\sigma_{\eta_1}\widehat{\otimes}\sigma'_0\widehat{\otimes}\tau_{-2})\oplus (\sigma_{\eta_1} \widehat{\otimes} \sigma'_{\eta'_2}\widehat{\otimes}\tau_0)\oplus (\sigma_{\eta_1} \widehat{\otimes} \sigma'_{\eta'_1}\widehat{\otimes}\tau_{-1})\oplus (\sigma_{\eta_1}\widehat{\otimes}\sigma'_{\eta'_3}\widehat{\otimes}\tau_1),
$
where $\fh_\C'\simeq \so(3)_\C\oplus \su(4)_\C\oplus \mathfrak{u}(1)_\C$. Therefore $V_{\pi_{\omega_5}}^H=0$ by Lemma~\ref{lem3:dimV_pi^H}. 

We conclude that $\lambda_1(G/H,g_{\st})\geq1$ in all cases.

\smallskip
\noindent
$\bullet$
Family XIX: 
This is a conceptual construction explained in \cite[p.574]{WangZiller85}. 
We will prove for every member $G/H$ in this family that $\lambda_1(G/H,g_{\st})\geq1$. 
We will follow the notation and the procedure as in \cite[\S{}7]{EJlauret-Stab}. 

\begin{table}
\caption{Irreducible symmetric spaces admitted as $G_i/K_i$ in Family XIX.} 
\label{table5:G_i/K_i}
{\small 
$
\begin{array}{ccccc}
G_i/K_i & \text{Cond.} &
m_i & \tfrac{\dim\fk_i}{m_i} 
\\[0mm] \hline \hline \rule{0pt}{14pt}
	\SO(2n)/\SO(n)\times\SO(n) & n\geq 3 & 
	n^2 & \tfrac{n-1}{n}
\\[0mm] \hline \rule{0pt}{14pt}
	\Sp(2n)/\Sp(n)\times\Sp(n) & n\geq 2 & 
	4n^2 & \tfrac{2n+1}{2n}
\\[0mm] \hline \rule{0pt}{14pt} 
	\SO(n+1)/\SO(n) & n\geq 2 & 
	n & \tfrac{n-1}{2}
\\[0mm] \hline \rule{0pt}{14pt}
	\SU(n)/\SO(n) & n\geq3,\,n\neq4 & 
	\tfrac{(n-1)(n+2)}{2} & \tfrac{n}{n+2}
\\[0mm] \hline \rule{0pt}{14pt}
	\SU(2n)/\Sp(n) & n\geq 3 & 
	(n-1)(2n+1) & \tfrac{n}{n-1}
\\[0mm] \hline \rule{0pt}{14pt}
	E_6/\Sp(4) & - & 
	42 & \tfrac{6}{7}
\\[0mm] \hline \rule{0pt}{14pt}
	E_6/F_4 & -& 
	26& 2
\\[0mm] \hline \rule{0pt}{14pt}
	E_7/\SU(8)& - & 
	70& \tfrac{9}{10}
\\[0mm] \hline \rule{0pt}{14pt}
	E_8/\Spin(16) & - & 
	128&\tfrac{15}{16}
\\[0mm] \hline \rule{0pt}{14pt}
	F_4/\Spin(9) & - & 
	16 & \tfrac{9}{4}
\\[0mm] \hline \rule{0pt}{14pt}
	(H\times H)/\Delta H & \dim H>3 & 
	\dim{H} & 1
\\[0mm] \hline 
\end{array}
$}\\[0.5mm]
$H$ is any compact simple Lie group in the last row.
\end{table}

Let $l$ be an integer with $l\geq2$. 
For each $i=1,\dots,l$, we pick a compact irreducible symmetric space $G_i/K_i$ from the list in Table~\ref{table5:G_i/K_i}. 
We write $m_i=\dim (G_i/K_i)$, $\fg_i=\fk_i\oplus\fm_i$ the corresponding Cartan decomposition, and let $\rho_i:K_i\to \SO(\fm_i)$ be the isotropy representation of $G_i/K_i$, which is irreducible and of real type. 

We set $H'=K_1\times\dots\times K_l$ and 
\begin{equation}\label{eq5:rhoXIX}
\rho=
{\textstyle \bigoplus\limits_{i=1}^k}\,
	\sigma_{0}^{(1)}
	\widehat\otimes\,\dots\widehat\otimes \,
	\sigma_{0}^{(i-1)} 
	\widehat\otimes\,
	\rho_i
	\widehat\otimes\,
	\sigma_{0}^{(i+1)} 
	\widehat\otimes\,\dots\widehat\otimes\,
	\sigma_{0}^{(k)}
: H'\to\GL\Big({\textstyle \bigoplus\limits_{i=1}^k}\, \C^{m_i}\Big)
. 
\end{equation}
It follows that $\rho$ is of real type, thus $\rho(H')\subset G_{\rho}=\SO(N)$, where $N=m_1+\dots+m_l$.

The space ${\SO(N)}/{\rho(H')}$ is not necessarily simply connected. 
However, we know there is a connected Lie group $G$ with Lie algebra $\so(N)$ such that $G/H$ is simply connected, where $H=\rho_0(H')/\op{Ker}(\rho_0)$ and $\rho_0:H'\to G$ is the only Lie group homomorphism with $d\rho_0=d\rho$. 

We have that $V_{\pi_{\omega_1}}^{\fh_\C} 
= V_\rho^{\fh_\C'}$ by Lemma~\ref{lem3:standard}, thus  $V_{\pi_{\omega_1}}^{\fh_\C}=0$ since $(d\rho)_{\C}$ does not contain the trivial representation of $\fh_\C'$. 
Now,  Proposition~\ref{prop:sufficient}\eqref{item3:so-N>=15} 
implies the desired result if $N\geq15$.

It is proven in \cite[Example 3]{WangZiller85} that the standard metric $g_{\st}$ on $G/H$ (or equivalently on $\SO(N)/\rho(H')$) is Einstein if and only if $\frac{\dim\fk_i}{m_i}$ is independent of $i$, that is, 
\begin{equation}\label{eq5:dim(k_i)/m_i}
\frac{\fk_1}{m_1}=\dots= \frac{\fk_l}{m_l}
, 
\end{equation}
which will be assumed from now on. 

We now examine the cases satisfying $N\leq 14$. 
Since $m_i\geq2$, we first note that any $G_i/K_i$ with $m_i\geq 13$ cannot participate. 
Therefore, via a simple inspection at Table~\ref{table5:G_i/K_i}, the choices for each $G_i/K_i$ are reduced to 
$\SO(6)/\SO(3)\times\SO(3)$, 
$\SO(n+1)/\SO(n)$ for $2\leq n\leq 12$, 
$\SU(3)/\SO(3)$, 
and
$(H_i\times H_i)/H_i$ with $H_i$ a compact simple Lie group of dimension $\leq 12$, so $\fh_i$ is isomorphic to either $\su(2)$, $\su(3)$ or $\so(5)$. 

One can easily check that the condition \eqref{eq5:dim(k_i)/m_i} forces that all members $G_i/K_i$ for all $i=1,\dots,l$ come from the same row in Table~\ref{table5:G_i/K_i} (see \cite[Table~8]{EJlauret-Stab}). 
Furthermore, if $G_i/K_i=\SO(n+1)/\SO(n)$ for all $i$, the corresponding space belongs to Family XVIIb (see \cite[Rem.~7.2]{EJlauret-Stab}), which was already considered so we omit it.

If $G_i/K_i=\SO(6)/\SO(3)\times\SO(3)$ for all $i$, then $N=l m_i=9l\geq18$ since $l\geq2$, so this case is not possible.

If $G_i/K_i=\SU(3)/\SO(3)$ for all $i$, then $N=l m_i=5l$, so $l=2$. 
We have that $\fg\simeq \so(10)$ and $\fh\simeq \su(2)\oplus\su(2)$, and $(d\rho)_\C\simeq \sigma_{4\eta_1}\widehat\otimes\sigma_{0}'\oplus \sigma_{0}\widehat\otimes\sigma_{4\eta_1'}'$. 
By Proposition~\ref{prop:sufficient}\eqref{item3:so-N<15}, it only remains to show that $V_{\pi_{\omega_5}}^H=0$. 
We have that (see e.g.\ \cite[pp.\ 232, 330]{LieART}; $\dim V_{\pi_{\omega_5}}=16$) that $\pi_{\omega_5}|_{\fh_\C'}\simeq \sigma_{3\eta_1}\widehat{\otimes} \sigma'_{3\eta_1'}$.

Suppose $G_i/K_i=(H_i\times H_i)/H_i$ for some compact simple Lie groups $H_i$ for $i=1,\dots,l$. 
Since $m_i=\dim \fh_i$, the only choices satisfying  $N\leq 14$ are as follows:
\begin{equation*}
\begin{array}{ccccc}
\fh'\simeq \fh_1\oplus\dots\oplus\fh_l & N & \rho & \pi_{\omega_{\lfloor N/2\rfloor}}|_{\fh_\C'}
\\ \hline\hline \rule{0pt}{13pt}
2\su(2) & 6 & 
\sigma_{2\eta_1}\widehat\otimes\sigma_0' \oplus \sigma_0\widehat\otimes\sigma_{2\eta_1'}' & \sigma_{\eta_1}\widehat{\otimes} \sigma'_{\eta_1'}
\\ [4pt]
3\su(2) &9 &
	\sigma_{2\eta_1}
	\widehat\otimes
	\sigma_0'
	\widehat\otimes
	\sigma_0''
\oplus 
	\sigma_0
	\widehat\otimes
	\sigma_{2\eta_1'}'
	\widehat\otimes
	\sigma_0''
\oplus 
	\sigma_0
	\widehat\otimes
	\sigma_{0}'
	\widehat\otimes
	\sigma_{2\eta_ 1''}'' & 2(\sigma_{\eta_1}\widehat{\otimes}\sigma'_{\eta'_1}\widehat{\otimes} \sigma''_{\eta''_1})
\\ [4pt]
\su(2)\oplus\su(3) &11&
\sigma_{2\eta_1}\widehat\otimes\sigma_0' \oplus \sigma_0\widehat\otimes\sigma_{\eta_1'+\eta_2'}' & 2(\sigma_{\eta_1}\widehat{\otimes} \sigma'_{\eta_1'+\eta_2'})
\\ [4pt]
4\su(2) &12& 
\text{\small \(
\begin{array}{l}
	\sigma_{2\eta_1}
	\widehat\otimes
	\sigma_0'
	\widehat\otimes
	\sigma_0''
	\widehat\otimes
	\sigma_0'''
\oplus 
	\sigma_0
	\widehat\otimes
	\sigma_{2\eta_1'}'
	\widehat\otimes
	\sigma_0''
	\widehat\otimes
	\sigma_0'''
\\
\oplus 
	\sigma_0
	\widehat\otimes
	\sigma_{0}'
	\widehat\otimes
	\sigma_{2\eta_1''}''
	\widehat\otimes
	\sigma_0'''
\oplus 
	\sigma_0
	\widehat\otimes
	\sigma_{0}'
	\widehat\otimes
	\sigma_{0}'''
	\widehat\otimes
	\sigma_{2\eta_1''''}'''' 
\end{array}
\)}
& 
2(\sigma_{\eta_1}\widehat{\otimes} \sigma'_{\eta'_1} \widehat{\otimes} \sigma''_{\eta''_1} \widehat{\otimes} \sigma'''_{\eta'''_1} )
\\ [4pt]
\su(2)\oplus\so(5) &13&
\sigma_{2\eta_1}\widehat\otimes\sigma_0' \oplus \sigma_0\widehat\otimes\sigma_{2\eta_2'}' & 2(\sigma_{\eta_1}\widehat{\otimes} \sigma'_{\eta'_1+\eta'_2})
\\ [4pt]
2\su(2)\oplus\su(3) &14&
	\sigma_{2\eta_1}
	\widehat\otimes
	\sigma_0' 
	\widehat\otimes
	\sigma_0''
\oplus 
	\sigma_0
	\widehat\otimes
	\sigma_{2\eta_1'}' 
	\widehat\otimes
	\sigma_0''
\oplus 
	\sigma_0
	\widehat\otimes
	\sigma_0'
	\widehat\otimes
	\sigma_{\eta_1''+\eta_2''}'' & 2(\sigma_{\eta_1}\widehat{\otimes}\sigma'_{\eta'_1}\widehat{\otimes}\sigma''_{\eta''_1+\eta''_2})
\\
\end{array}
\end{equation*}
Proposition~\ref{prop:sufficient}\eqref{item3:so-N<15} ensures that it is sufficient to show that $V_{\pi}^{\fh_\C}=0$, where $\pi$ is the spin representation with highest weight $\omega_{\lfloor N/2\rfloor}$. 
The last column in the table shows the decomposition in irreducible components of such representation, and it follows that $V_{\pi}^{\fh_\C'}=0$ since $1_{\fh_\C'}$ never occurs.

We conclude that $\lambda_1(G/H,g_{\st})\geq1$ in all cases. 

\smallskip
\noindent
$\bullet$
Isolated case No.~34: 
$\fg=\so(8)$ and $\fh\simeq \fg_2$. 

Set $H'=\op{G}_2$. 
$\sigma_{\eta_1}$ is of real type so $\sigma_{\eta_1}(H')\subset \SO(7)$. 
Let $\rho=\iota\circ\sigma_{\eta_1}$, where $\iota$ is the standard embedding $\SO(7)\subset \SO(8)$, and let $\rho_0:H'\to \Spin(8)$ be the only Lie group homomorphism such that $d\rho=d\rho_0$. 
The corresponding connected and simply connected space is $G/H=\Spin(8)/\rho_0(H')$ with $H\simeq H'$. 
It turns out that $G/H$ is diffeomorphic to $S^7\times S^7$, though $g_{\st}$ is not isometric to a product  metric. 

In this very peculiar case, we have that $\dim V_{\pi_{\omega_1}}^H=1$. 
Indeed, Lemma~\ref{lem3:standard} yields $V_{\pi_{\omega_1}}^H=V_\rho^{\fh_\C'}$, which has dimension one due to the inclusion $\iota$. 
Now, Theorem~\ref{thm:Spec(standard)} and Table~\ref{table:CasimirEigenvalues} ensure that 
\begin{equation}
\lambda_1(\Spin(8)/\op{G}_2,g_{\st}) =\lambda^{\pi_{\omega_1}}=\frac{7}{12}<\frac35=2E.
\end{equation}
We conclude that the Einstein standard manifold $(\Spin(8)/\op{G}_2,g_{\st})$ has $\nu$-unstable conformal directions. 
We knew in advance that this space is $\nu$-unstable since it was established as $G$-unstable in \cite{EJlauret-Stab}.

\smallskip
\noindent
$\bullet$
Isolated case No.~35: 
$\fg=\so(26)$ and $\fh\simeq \spp(1)\oplus \spp(5)\oplus \so(6)$. 

Set $H'=\Sp(1)\times\Sp(5)\times\SO(6)$ and 
$
\rho=
	\sigma_{\eta_1}
	\widehat\otimes
	\sigma_{\eta_1'}'
	\widehat\otimes
	\sigma_{0}''
\,\oplus\, 
	\sigma_{0}
	\widehat\otimes
	\sigma_{0}'
	\widehat\otimes
	\sigma_{\eta_1''}''
: H'\to\GL\left(\C^{20}\oplus \C^{6} \right)
. 
$
Both terms of $\rho$ are of real type:
the second one is obvious and the first one is a tensor product of two quaternionic representations. 
Therefore, $\rho(H')\subset G_\rho= \SO(26)$.

Lemma~\ref{lem3:standard} implies that  $V_{\pi_{\omega_1}}^{H}=V_{\pi_{\omega_1}}^{\fh_\C}= V_{\rho}^{\fh_\C'}$, which vanishes since $(d\rho)_\C$ does not contain $1_{\fh_\C'}$ in its irreducible decomposition. 
Now, Proposition~\ref{prop:sufficient}\eqref{item3:so-N>=15} yields that $\lambda_1(G/H,g_{\st})\geq 1$.

\subsection{Exceptional cases for isotropy reducible spaces}\label{subsec5:exceptional}
In this subsection we prove that  $\lambda_1(G/H,g_{\st})\geq1$ for all cases $G/H$ from Tables~\ref{table5:isotropyred-families}--\ref{table5:isotropyredexcepcions} with $\fg$ of exceptional type except for one, namely, Isolated cases No.~36--55.  

\smallskip
\noindent
$\bullet$
Isolated case No.~36: 
$\fg=\ff_4$ and $\fh\simeq \so(8)$. 

According to \cite[Example 5]{WangZiller85}, $H:=\Spin(8)$ embeds into $G:=\op{F}_4$ via the standard inclusion $H\subset \Spin(9)$ and the embedding $\Spin(9)\subset G$ is determined by the maximal subalgebra $\so(9)_\C$ of $(\ff_4)_\C$. 

We have that   
$
\pi_{\omega_4}|_{H}\simeq 2\sigma_0\oplus \sigma_{\eta_1}\oplus\sigma_{\eta_3}\oplus\sigma_{\eta_4}
$
(see e.g.\ \cite[pp.\ 217, 387]{LieART}; $\dim V_{\pi_{\omega_4}}=26$).
Theorem~\ref{thm:Spec(standard)} and Table~\ref{table:CasimirEigenvalues} yield
\begin{equation}
\lambda_1(G/H,g_{\st})=\lambda^{\pi_{\omega_4}} = \frac{2}{3} 
< \frac{8}{9}=2E
.
\end{equation}
We conclude that the Einstein standard manifold $(\op{F}_4/\Spin(8),g_{\st})$ has $\nu$-unstable conformal directions. 
We knew that this space is $G$-unstable from \cite{EJlauret-Stab}.

\smallskip
\noindent
$\bullet$
Isolated cases No.~37--39: 
$\fg=\fe_6$. 

By Proposition~\ref{prop:sufficient}\eqref{item3:E6}, $\lambda_1(G/H,g_{\st})\geq 1$ provided that $V_{\pi_{\omega_1}}^{\fh_\C}=0$. 
We have $\dim V_{\pi_{\omega_1}}=27$.

For the case No.~37, $\fh'\simeq \su(2)\oplus\su(2)\oplus \su(2)$. 
We have that $\pi_{\omega_1}|_{\fh'_\C}\simeq \sigma_{2\eta_1}\widehat{\otimes} \sigma'_{2\eta'_1}\widehat{\otimes}\sigma''_0\oplus \sigma_{2\eta_1}\widehat{\otimes} \sigma'_{0}\widehat{\otimes}\sigma''_{2\eta''_1}\oplus \sigma_{0}\widehat{\otimes}\sigma'_{2\eta'_1}\widehat{\otimes}\sigma''_{2\eta''_1}$ 
(see e.g.\ \cite[pp. 375, 150]{LieART}).

For the case No.~38, $\fh'\simeq \su(2)\oplus\so(6)$. 
We have that 
$\pi_{\omega_1}|_{\fh'_\C}\simeq \sigma_{\eta_1}\widehat{\otimes}\sigma'_{\eta'_1}\oplus \sigma_0\widehat{\otimes} \sigma'_{\eta'_2+\eta'_3}$ 
(see e.g.\ \cite[pp. 374, 157]{LieART}).

For the case No.~39, $\fh'\simeq\so(8)\oplus \mathfrak{u}(1)\oplus\mathfrak{u}(1)$.
We have that 
$\pi_{\omega_1}|_{\fh'_\C}=(\sigma_0 \widehat{\otimes} \tau_1 \widehat{\otimes} \tau_4)\oplus (\sigma_0 \widehat{\otimes} \tau_2\widehat{\otimes} \tau_2)\oplus (\sigma_0 \widehat{\otimes} \tau_{-2} \widehat{\otimes} \tau_2)\oplus (\sigma_{\eta_1}\widehat{\otimes} \tau_0 \widehat{\otimes} \tau_2)\oplus (\sigma_{\eta_3} \widehat{\otimes} \tau_{-1} \widehat{\otimes} \tau_{-1}) \oplus (\sigma_{\eta_4} \widehat{\otimes} \tau_1 \widehat{\otimes} \tau_{-1})$ (see e.g.\ \cite[pp.\ 374, 228]{LieART}).

In the three cases we have that $V_{\pi_{\omega_1}}^{\fh_\C'}=0$ by Lemma~\ref{lem3:dimV_pi^H}, so $\lambda_1(G/H,g_{\st})\geq 1$.

\smallskip
\noindent
$\bullet$
Isolated cases No.~41--43: 
$\fg=\fe_7$. 

Proposition~\ref{prop:sufficient}\eqref{item3:E7} yields that $\lambda_1(G/H,g_{\st})\geq 1$ provided $V_{\pi_{\omega_7}}^{\fh_\C}=0$. We have that $\dim V_{\pi_{\omega_7}}=56$.

For the case No.~41, $\fh'\simeq 7\su(2)$.
We have that 
$
\pi_{\omega_7}|_{\fh_\C'}\simeq 
\sigma^{(1)}_{\eta^{(1)}_1}\widehat{\otimes}\sigma^{(2)}_{\eta^{(2)}_1}\widehat{\otimes} \sigma^{(3)}_{\eta^{(3)}_1}\oplus \sigma^{(1)}_{\eta^{(1)}_1}\widehat{\otimes}\sigma^{(4)}_{\eta^{(4)}_1}\widehat{\otimes}\sigma^{(5)}_{\eta^{(5)}_1}\oplus \sigma^{(1)}_{\eta^{(1)}_1}\widehat{\otimes}\sigma^{(6)}_{\eta^{(6)}_1}\widehat{\otimes}\sigma^{(7)}_{\eta^{(7)}_1}
\oplus
\sigma^{(2)}_{\eta^{(2)}_1}\widehat{\otimes} \sigma^{(4)}_{\eta^{(4)}_1}\widehat{\otimes}\sigma^{(7)}_{\eta^{(7)}_1}
\oplus \sigma^{(2)}_{\eta^{(2)}_1}\widehat{\otimes}\sigma^{(5)}_{\eta^{(5)}_1}\widehat{\otimes}\sigma^{(6)}_{\eta^{(6)}_1}
\oplus 
\sigma^{(3)}_{\eta^{(3)}_1}\widehat{\otimes}\sigma^{(4)}_{\eta^{(4)}_1}\widehat{\otimes} \sigma^{(6)}_{\eta^{(6)}_1}\oplus \sigma^{(3)}_{\eta^{(3)}_1}\widehat{\otimes}\sigma^{(5)}_{\eta^{(5)}_1}\widehat{\otimes}\sigma^{(7)}_{\eta^{(7)}_1}
$
(see e.g.\ \cite[pp.\ 379, 238, 197]{LieART}).
Here, tensoring with the identity representation is omitted for brevity.

For the case No.~42, $\fh'\simeq \so(8)$.
We have that 
$
\pi_{\omega_7}|_{\fh_\C'}\simeq 
2\sigma_{\eta_2}
$
(see e.g.\ \cite[pp.\ 379, 164]{LieART}).
In \cite{SchwahnSemmelmannWeingart22}, the case $G/H=\op{E}_7/(\SO(8)/\Z_2)$ was established as H.stable prior to \cite{Schwahn-Lichnerowicz}.

For the case No.~43, $\fh'\simeq 3\su(2)\oplus \so(8)$.
We have that (see e.g.\ \cite[pp.\ 379, 238]{LieART})
$
\pi_{\omega_7}|_{\fh_\C'}\simeq 
\sigma_{\eta_1}\widehat{\otimes}\sigma'_{\eta'_1}\widehat{\otimes}\sigma''_{\eta''_1}\widehat{\otimes}\sigma'''_0\oplus \sigma_{\eta_1}\widehat{\otimes}\sigma'_0\widehat{\otimes}\sigma''_0\widehat{\otimes}\sigma'''_{\eta'''_1}
\oplus
\sigma_0\widehat{\otimes}\sigma'_{\eta'_1}\widehat{\otimes}\sigma''_0\widehat{\otimes}\sigma'''_{\eta'''_3}\oplus \sigma_0\widehat{\otimes}\sigma'_0\widehat{\otimes}\sigma''_{\eta''_1}\widehat{\otimes}\sigma'''_{\eta'''_4}
$.

In the three cases we have that $V_{\pi_{\omega_7}}^{\fh_\C'}=0$ by Lemma~\ref{lem3:dimV_pi^H}, so $\lambda_1(G/H,g_{\st})\geq 1$.

\smallskip
\noindent
$\bullet$
Isolated cases No.~45--54: 
$\fg=\fe_8$. 

Proposition~\ref{prop:sufficient}\eqref{item3:E8} immediately implies $\lambda_1(G/H,g_{\st})\geq 1$ for any of these isolated cases.

\subsection{Computational results for isotropy reducible cases} \label{subsec5:computationalresults}

Analogously to the calculations of $\lambda_1$ with the help of a computer in Subsection~\ref{subsec4:computational} for isotropy irreducible spaces, we compute with the same algorithm, for each entry $(\fg,\fh)$ in Table~\ref{table5:isotropyredexcepcions}, the number
\begin{equation}\label{eq5:lowerbound}
\min\{\lambda^\pi: \pi\in\widehat {G}_{\text{sc}}\text{ and }V_\pi^H\neq0\},
\end{equation}
where $G_{\text{sc}}$ denotes the (unique up to isomorphism) simply connected and connected compact Lie group with Lie algebra $\fg$, and $H$ denotes the only connected subgroup of $G_{\text{sc}}$ with Lie algebra $\fh$. 

Since we do not know that all the corresponding spaces are presented as $G_{\text{sc}}/H$ with $H$ connected, the value \eqref{eq5:lowerbound} is a lower bound of $\lambda_1(G/H,g_{\st})$ for any isotropically reducible space $G/H$ corresponding to $(\fg,\fh)$.

No members of the families listed in Table~\ref{table5:isotropyred-families} were computed.

\section*{Funding information}

The first named author was supported by grants from FONCyT (PICT-2019-01054), CONICET (PIP 11220210100343CO), and SGCYT--UNS. The second named author was partially supported by CONICET and the PRIN 2022 project “Interactions between Geometric Structures and Function Theories” (code 2022MWPMAB)

\bibliographystyle{plain}

\end{document}